\documentclass[10pt,a4paper]{article}
\usepackage[utf8]{inputenc}
\usepackage[english]{babel}
\usepackage[all]{xy}
\usepackage{enumitem}
\usepackage[dvipsnames]{xcolor}
\usepackage{graphicx}
\usepackage{stmaryrd}
\usepackage{amsfonts,amstext,amssymb,amsmath,amsthm,pspicture,bigints}
\usepackage{fullpage}
\usepackage{moreverb}
\usepackage{pifont}
\usepackage{pst-all}
\usepackage[left=2cm,right=2cm,top=2cm,bottom=2cm]{geometry}
\usepackage{esint}
\usepackage{fourier-orns}
\usepackage{mathrsfs}
\usepackage{array}

\usepackage{hyperref}
\usepackage{authblk}
\newcommand{\ii }{{\rm i} }
\newcolumntype{C}[1]{>{\centering\arraybackslash}b{#1}}
\newcolumntype{R}[1]{>{\raggedleft\arraybackslash}b{#1}}
\newcolumntype{L}[1]{>{\raggedright\arraybackslash}b{#1}}
\newcolumntype{M}[1]{>{\centering}m{#1}}
\newtheorem{theo}{Theorem}[section]
\newtheorem{defin}{Definition}[section]
\newtheorem{lem}{Lemma}[section]
\newtheorem{prop}{Proposition}[section]

\newtheorem{remark}{Remark}[section]

\usepackage{bbm}
\numberwithin{equation}{section}

\date{}
\title{Geometric effects on the persistence of patch
structures for\\ 
2D Euler equations in simply-connected bounded domains}
\author{Emeric Roulley}
\begin{document}
\maketitle
\begin{abstract}
	We first prove a rigidity result for uniformly rotating vortex patch solutions to Euler equations in bounded simply-connected planar domains. Namely, if a bounded simply-connected domain with real-analytic boundary contains a nontrivial vortex patch that remains strictly inside the domain while undergoing a uniform rigid rotation, then the ambient domain must be a disc centered at the rotation center. The proof combines the rotating-frame formulation with harmonic unique continuation and the generalized Schwarz's reflection principle across analytic boundary arcs. These tools imply that, unless the outer domain is rotationally invariant, the velocity must vanish in an open region surrounding the patch, contradicting the nonzero circulation generated by the patch. In the same spirit, we show that only circular domains possess a circular point vortex orbit.\\
    
   As a second result, we construct quasi-periodic vortex patch solutions to the Euler equations in domains that are small perturbations of the unit disc $\mathbb{D}$. The amplitudes of the solutions are linked to the size of the domain deformation, allowing the classical Rankine vortices $b\cdot\mathbb{D}$, $b\in(0,1)$, to be regarded as equilibrium states. Our analysis is perturbative with respect to the construction developed in the undeformed case \cite{HR21-1}, and the resulting quasi-periodic solutions exist for most values of the inner radius $b$. This provides the first construction of quasi-periodic vortex patch solutions to the Euler equations in non-radial domains. Moreover, combined with the rigidity result established above, this shows that, in the periodic case, these solutions cannot arise from a purely rigid rotation.
\end{abstract}
\tableofcontents
\section{Introduction}
The present paper proposes to explore some geometric effects on the construction of periodic and quasi-periodic vortex patches within simply-connected bounded domains. In this introductive section, we first recall the 2D Euler equations in simply-connected domains. Then, we discuss some historical background on the construction of V-states and quasi-periodic vortex patch solutions. Finally, we state our main results and give some elements of their proofs.
\subsection{Euler equations in a simply-connected domain}\label{sec intro 1}
We consider the incompressible Euler system set in a simply-connected domain $\mathbf{D}\subset\mathbb{R}^{2}\equiv\mathbb{C}$ and written in the velocity-vorticity formulation
\begin{equation}\label{Euler eq bounded domain}
	\begin{cases}
		\partial_{t}\boldsymbol{\omega}+\mathbf{v}\cdot\nabla\boldsymbol{\omega}=0, & \textnormal{in }\mathbb{R}\times\mathbf{D},\\
		\mathbf{v}=\nabla^{\bot}\boldsymbol{\Psi}, & \textnormal{in }\mathbb{R}\times\mathbf{D},\\
		\boldsymbol{\omega}(0,\cdot)=\boldsymbol{\omega}_0, & \textnormal{in }\mathbf{D}.
	\end{cases}
\end{equation}
This is an active scalar model with unknown the vorticity $\boldsymbol{\omega}=\operatorname{curl}(\mathbf{v})$ related to the solenoidal (second equation in \eqref{Euler eq bounded domain}) velocity field $\mathbf{v}$ of the fluid. In the case where the domain $\mathbf{D}$ has a boundary, we impose the additional impermeability condition stating that no fluid enters or escapes the domain $\mathbf{D}$ at its boundary
\begin{equation}\label{impermea}
    \mathbf{v}\cdot\mathbf{n}_{\partial\mathbf{D}}=0,\qquad \textnormal{on }\mathbb{R}\times\partial\mathbf{D}.
\end{equation}
The notation $\mathbf{n}_{\partial\mathbf{D}}$ stands for the outward normal vector to the boundary $\partial\mathbf{D}$. By construction, the velocity potential (or stream function) $\boldsymbol{\Psi}$ solves the elliptic equation
\begin{equation}\label{elliptic_eq:stream-intro}
    \Delta\boldsymbol{\Psi}=\boldsymbol{\omega},\quad\textnormal{in }\mathbb{R}\times\mathbf{D}.
\end{equation}
In addition, the impermeability condition \eqref{impermea} translates into a Dirichlet boundary condition for the velocity potential, namely
\begin{equation}\label{Dirichlet_stream}
    \boldsymbol{\Psi}\equiv0,\quad\textnormal{on }\mathbb{R}\times\partial\mathbf{D}.
\end{equation}
From \cite{Evans}, solving the previous PDE gives the following integral representation of the stream function
\begin{equation}\label{Psi EDP}
	\boldsymbol{\Psi}(t,x)=\frac{1}{2\pi}\int_{\mathbf{D}} G_{\mathbf{D}}(x,y)\boldsymbol{\omega}(t,y)dy,
\end{equation}
where $G_{\mathbf{D}}:\mathbf{D}\times\mathbf{D}\rightarrow\mathbb{R}$ is the Green function associated to the Dirichlet-Laplace problem set in $\mathbf{D}.$ In particular,
$$G_{\mathbb{R}^2}(x,y)=\log|x-y|$$
and, if the domain $\mathbf{D}$ is bounded, then according to \cite[Chap. 6]{A66} we can write
\begin{equation}\label{def:Green}
    G_{\mathbf{D}}(x,y)=G_{\mathbb{D}}\big(\boldsymbol{\Phi}(x),\boldsymbol{\Phi}(y)\big),\qquad  G_{\mathbb{D}}(z,\xi)=\log\left|\frac{z-\xi}{1-\overline{\xi}z}\right|,
\end{equation}
where $\boldsymbol{\Phi}:\mathbf{D}\to\mathbb{D}$ is a conformal mapping and $\mathbb{D}$ is the unit disc. The function $G_{\mathbf{D}}$ admits the following decomposition
$$G_{\mathbf{D}}(x,y)=\log|x-y|+K_{\mathbf{D}}(x,y),$$
where
$$K_{\mathbf{D}}(x.y)\triangleq\log\left|\frac{\frac{\boldsymbol{\Phi}(x)-\boldsymbol{\Phi}(y)}{x-y}}{1-\overline{\boldsymbol{\Phi}(y)}\boldsymbol{\Phi}(x)}\right|$$
is smooth on $\mathbf{D}\times\mathbf{D}$ and actually harmonic in both variables $x$ and $y.$ We introduce the Robin function
$$\mathcal{R}_{\mathbf{D}}(x)\triangleq K_{\mathbf{D}}(x,x)=-\log\left(\frac{1-|\boldsymbol{\Phi}(x)|^2}{|\boldsymbol{\Phi}'(x)|}\right).$$
Note that the Robin function is real-analytic in $\mathbf{D}$ and explodes at the boundary
\begin{equation}\label{Robin bndry}
    \lim_{x\to\partial\mathbf{D}}\mathcal{R}_{\mathbf{D}}(x)=+\infty.
\end{equation}
In addition, it solves the Liouville equation
\begin{equation}\label{Liouville intro}
    \Delta\mathcal{R}_{\mathbf{D}}=4e^{2\mathcal{R}_{\mathbf{D}}},\quad\textnormal{in }\mathbf{D}.
\end{equation}
A point vortex is a formal solution to Euler equations where the vorticity is concentrated in a point
$$\boldsymbol{\omega}(t)=\Gamma\delta_{\xi(t)},\qquad\Gamma\in\mathbb{R}^*.$$
Its dynamics is generated by the boundary $\partial\mathbf{D}$ and given by the following Hamiltonian system 
\begin{equation}\label{Hamiltonian_system_1PV}
    \dot{\xi}(t)=\frac{\Gamma}{4\pi}\nabla^{\bot}\mathcal{R}_{\mathbf{D}}\big(\xi(t)\big).
\end{equation}
 
\subsection{Historical background on (quasi-)periodic vortex patches}
We shall now discuss some important results on the emergence of periodic and quasi-periodic vortex patch solutions. Recall that the Yudovich theory \cite{Y63} (see also \cite{GVL13,LZ19} for domains with rough boundaries and holes) ensures the existence and uniqueness of a global in time weak Lagrangian solution to \eqref{Euler eq bounded domain} provided that $\boldsymbol{\omega}_0\in L^1(\mathbf{D})\cap L^{\infty}(\mathbf{D}).$ In particular, if $\boldsymbol{\omega}$ is the characteristic function of a bounded domain $D_0$, then, at any time $t>0$, the solution $\boldsymbol{\omega}(t)$ is also the characteristic function of a bounded domain $D_t$ which has the same area as $D_0.$ This type of solutions is called \textit{vortex patch}
$$\boldsymbol{\omega}(t)=\mathbf{1}_{D_t},\qquad|D_t|=|D_0|.$$
It is well known that the dynamics of such solutions is characterized by the evolution of the boundary $\partial D_t=\{z(t,\theta),\,\theta\in\mathbb{T}\}$, with $\mathbb{T}\triangleq\mathbb{R}/2\pi\mathbb{Z},$ through the so-called contour dynamics equation, see \cite{HMV13}
\begin{equation}\label{CDE}
	\textnormal{Im}\Big\{\partial_{t}z(t,\theta)\overline{\partial_{\theta}z(t,\theta)}\Big\}=\partial_{\theta}\Big[\boldsymbol{\Psi}\big(t,z(t,\theta)\big)\Big].
\end{equation}
Now, we shall review the literature by dividing the various results into three different topics: V-states, quasi-periodic solutions and non-rigid periodic motion. For a more complete presentation, we refer the reader to the monograph \cite{H26}.
\paragraph{V-states}
Vortex patches for which the motion corresponds to a uniform rotation of the initial domain are called \textit{V-states}. More precisely, there exists an angular velocity $\Omega\in\mathbb{R}$ such that 
$$D_t=\mathcal{R}_{\Omega t}D_0,\qquad\mathcal{R}(\theta)\triangleq\begin{bmatrix}
    \cos(\theta) & -\sin(\theta)\\
    \sin(\theta) & \cos(\theta)
\end{bmatrix}.$$
$\cdot$ \texttt{Case $\mathbf{D}=\mathbb{R}^2$ :} Obviously, radial patches are trivial V-states commonly called \textit{Rankine vortices}. The first nontrivial example of V-states have elliptic shape and were found by Kirchhoff \cite{K74}. More precisely, an ellipse with semi-axes $a$ and $b$ rotates uniformly about its center with speed
$$\Omega_{a,b}\triangleq\frac{ab}{(a
+b)^2}\cdot$$
Solutions with higher symmetry were then numerically obtained a century after by Deem \& Zabusky in \cite{DZ78}. The patch domains they found are $\mathbf{m}$-fold, namely invariant under $\tfrac{2\pi}{\mathbf{m}}$-rotation for a given positive integer $\mathbf{m}\geqslant3.$ Their rigorous construction has been performed by Burbea \cite{B82} by means of local bifurcation theory: the nontrivial V-states emerge from the Rankine unit disc at the angular velocities
$$\Omega_{\mathbf{m}}=\frac{\mathbf{m}-1}{2\mathbf{m}},\qquad\mathbf{m}\geqslant2.$$
In particular, the $\mathbf{m}=2$ branch corresponds to Kirchhoff solutions. The local Burbea branches can be extended via global bifurcation \cite{HMW20}. Note that Burbea's bifurcation points belong to the interval $(0,\tfrac{1}{2}).$ Outside this interval, no nontrivial V-states exist as proved in the series of works \cite{F00,GPSY20,H15}. Inside, this interval, the global phase portrait description is more delicate. Locally near the junction between the trivial branch and bifurcation branches rigidity also occurs \cite{H25}. In \cite{P22}, the Liouville properties of the $\mathbf{m}$-fold symmetric, simply-connected V-states were explored. Variational analysis was used to demonstrate that for a general family of $\mathbf{m}$-fold symmetric V-states, which satisfy reasonable symmetry constraints, the
angular velocity $\Omega$ must be greater than $\tfrac{1}{2}-\tfrac{C}{\mathbf{m}}$ for a universal positive constant $C$. The doubly-connected case where the patch domain possesses a hole  has been studied by de la Hoz, Hmidi, Mateu \& Verdera \cite{HHMV16} who constructed, for every positive integer $\mathbf{m}$ satisfying the condition 
\begin{equation}\label{cond:DC-intro}
    1+b^{\mathbf{m}}-\frac{1-b^2}{2}\mathbf{m}<0,
\end{equation}
local curves of non-annular $\mathbf{m}$-fold symmetric doubly-connected V-states bifurcating from the renormalized annulus
\begin{equation}\label{def:annulus}
    \mathbb{A}_b\triangleq\big\{z\in\mathbb{C}\quad\textnormal{s.t.}\quad b<|z|<1\big\},\qquad b\in(0,1),
\end{equation}
at the angular velocities
$$\Omega_\mathbf{m}^\pm(b)=\frac{1-b^2}{4}\pm\frac{1}{2\mathbf{m}} \sqrt{\left(\frac{\mathbf{m}\left(1-b^2\right)}{2} - 1\right)^2 - b^{2\mathbf{m}}}.$$
The degenerate case where the relation \eqref{cond:DC-intro} exactly vanishes has been partially treated in \cite{HM16-2,WXZ22}. An other way to get V-states is via implicit function theorem for desingularization of point vortices. We refer the reader to \cite{BCGHM26,G20,G21,HH21,HW21,HM17,XZ26} for the contour dynamics method and to \cite{CLZ21} for their variational counterparts.\\

\noindent $\cdot$ \texttt{Case $\mathbf{D}=\mathbb{D}$ :} In \cite[Thm. 1]{HHHM15}, de la Hoz, Hassainia, Hmidi \& Mateu could implement the local bifurcation theory and obtained for any $\mathbf{m}\in\mathbb{N}^*$ the existence of $\mathbf{m}$-fold symmetric V-states emerging from the Rankine vortices $b\cdot\mathbb{D}$, $b\in (0,1)$ at the angular velocities
$$\Omega_{\mathbf{m}}(b)=\frac{\mathbf{m}-1+b^{2\mathbf{m}}}{2\mathbf{m}}\cdot$$
The global phase portrait description in this case is less complete with only few results \cite{FWZ24,W24,WZ21}. The doubly-connected case is also analyzed in \cite[Thm. 2]{HHHM15}, obtaining, under the constraint
$$\mathbf{m}>\frac{2+2b^{\mathbf{m}}-(b_1^{\mathbf{m}}+b_2^{\mathbf{m}})}{1-b^2},\qquad b\triangleq\frac{b_2}{b_1},$$
the birfurcation from the renormalized annulus
$$\mathbb{A}_{b_1,b_2}\triangleq\big\{z\in\mathbb{C}\quad\textnormal{s.t.}\quad b_2<|z|<b_1\big\},\qquad0<b_2<b_1<1$$
at
$$\Omega_{\mathbf{m}}^{\pm}(b_1,b_2)\triangleq\frac{1-b^2}{4}+\frac{b_1^{2\mathbf{m}}-b_2^{2\mathbf{m}}}{4\mathbf{m}}\pm\frac{1}{2\mathbf{m}}\sqrt{\left(\frac{\mathbf{m}\left(1-b^2\right)}{2} -\frac{2-b_1^{2\mathbf{m}}-b_2^{2\mathbf{m}}}{2}\right)^2 - b^{2\mathbf{m}}(1-b_1^{\mathbf{m}})}.$$
The desingularization approach has been implemented in \cite[Thm. 6.1]{HHR24} near point vortex circular orbits. Variational methods also have been implemented \cite{CWWZ21,CWZ21}.\\

We mention that in general, for Euler equations, the steady solutions exhibit a wide variety of structures with flexibility results \cite{CS12,CDG21,CZEW23,EH25,EH26,EFR24} and rigidity ones \cite{DN26,EHSX24,GXX26,HN17,HN19,Ruiz23,WZ23}. In this context and together with the above mentioned literature, a natural question arises:
\begin{center}
    \textit{Do there exist nontrivial V-states in non-radial domains (i.e. neither $\mathbb{R}^2$ nor $\mathbb{D}$) ?}
\end{center}
The stationary case $\Omega=0$ admits positive answers \cite{CW22,LWZ19,WZ20}. One of the main objectives of the present paper is to bring a negative answer to the previous question in the case $\Omega\neq0$ within analytic simply-connected bounded ambient domains; see Theorem \ref{thm:inverse-pb}.\\ 

We conclude this paragraph on V-states by mentioning their constructions for other models related to the Euler equations such as the generalized suface quasi-geostrophic equations \cite{CCG16,G19,HH15,HHH16,HXX22,R17}, the shallow-water quasi-geostrophic equations \cite{BGR26,DHR19,R21}, the Euler-$\alpha$ equation \cite{R22}, the Euler equations on the rotating sphere \cite{GHR23}, the Lake equations \cite{HHRZ25} and the Vlasov-Poisson/Euler-Poisson systems \cite{GHHR26,R23,R26}. Finally, a unified theory of V-state constructions for generic active scalar planar models covering a small part of the previous results have been obtained in \cite{HXX26,HXX26-1}.
\paragraph{Time quasi-periodic patches} 
Recall that that a function $f:\mathbb{R}\to\mathbb{R}$ is said to be \textit{quasi-periodic} if and only if there exist $d\in\mathbb{N}^*,$ $F:\mathbb{T}^{d}\to\mathbb{R}$ and $\omega\in\mathbb{R}^d$ such that
$$\forall t\in\mathbb{R},\quad f(t)=F(\omega t)\qquad\textnormal{and}\qquad\forall l\in\mathbb{Z}^d\setminus\{0\},\quad\omega\cdot l\neq 0.$$
In particular, the case $d=1$ corresponds to a periodic function with frequency $\omega\in\mathbb{R}^*$. The construction of quasi-periodic solutions for fluid models was initiated by the seminal works of Berti \& Montalto \cite{BM18} and Baldi, Berti, Haus \& Montalto \cite{BBHM18}, where quasi-periodic standing water waves were constructed for the two-dimensional gravity-capillary and pure gravity water-wave equations on the torus. This was subsequently extended to traveling quasi-periodic water waves, allowing for the inclusion of constant vorticity \cite{BFM21,BFM21-1}. The deep-water case was later addressed by Feola \& Giuliani \cite{FG21} using weak Birkhoff normal form techniques, and more recently the corresponding three-dimensional problem was treated in \cite{FMT25}. Further developments include the construction of space-quasi-periodic steady solutions for the Euler equations in a bounded channel, which become quasi-periodic traveling waves in a suitable moving frame \cite{FMM24}, as well as large-amplitude traveling waves for the $\beta$-plane approximation of rotating fluids \cite{BFMT25}. We also mention several constructions of time-quasi-periodic or almost-periodic solutions to the Euler equations on the $d$-dimensional torus that do not rely on KAM theory \cite{CF13,EPSL23,FM24-1}. These solutions are obtained by superposing finitely or infinitely many rescaled, non-interacting traveling waves generated by compactly supported stationary building blocks. Finally, in the non-autonomous setting, time-quasi-periodic solutions have been constructed for the three-dimensional Euler equations with quasi-periodic forcing \cite{BM21} and for the two-dimensional Navier--Stokes equations \cite{FM24}. In the latter case, the solutions converge, in the zero-viscosity limit, to time-quasi-periodic solutions of the two-dimensional Euler equations on arbitrarily long time intervals.\\

The first constructions of quasi-periodic vortex patches were obtained in 2021 by Hmidi \& Roulley \cite{HR21} and Hassainia, Hmidi \& Masmoudi \cite{HHM21}, who considered the quasi-geostrophic shallow-water and generalized quasi-geostrophic equations near circular Rankine vortices. In both works, an external parameter in the equation is used to avoid resonances and verify the non-resonance conditions required by the KAM scheme. The same strategy was subsequently applied to the Euler-$\alpha$ model \cite{R22}, with the parameter $\alpha$ playing the role of the external parameter, again near circular Rankine patches. The remaining parameter range for the surface quasi-geostrophic equations has been covered by G\'omez-Serrano, Ionescu \& Park in \cite{GIP26}. The Euler equations themselves, however, do not possess such an external parameter. This led to the idea of exploiting instead the geometric parameters of suitable families of trivial solutions. This approach was first implemented in 2022 by Berti, Hassaini \& Masmoudi \cite{BHM23}, who used the aspect ratio of the Kirchhoff ellipses as a parameter. Around the same time, Hassainia \& Roulley used the radius $b$ of the Rankine patch in the unit disc as a parameter, while Hassainia, Hmidi \& Roulley \cite{HHR23} later considered the inner radius $b$ of the renormalized annulus $\mathbb{A}_b$ defined in \eqref{def:annulus}. In all these works, however, the ambient domain $\mathbf{D}$ in which the fluid evolves is either $\mathbb{R}^2$ or $\mathbb{D}$, both of which are radially symmetric. This naturally raises the following question:
\begin{center}
Are there non-radial domains in which quasi-periodic vortex patch solutions exist?
\end{center}
The second part of the present paper gives a positive answer to this question. We construct quasi-periodic vortex patch solutions in domains that are small perturbations of the unit disc, by perturbing the construction of the circular case developed in \cite{HR21-1}; see Theorem \ref{thm:QPS_in_nonradial_domains}.

\paragraph{Non-rigid periodic motion} We conclude this historical subsection by mentioning some recent works lying at the intermediate between the previous two presentations, namely periodic solutions but that do not perform a uniform rotation or translation. The first work in this direction has been obtained by Hassainia, Hmidi \& Masmoudi \cite{HHM25} for the planar leapfrogging motion. This work has been recently extended to the 3D vortex ring leapfrogging phenomenon in \cite{GHH26}. In bounded simply-connected domains, Hassainia, Hmidi \& Roulley \cite{HHR24} could desingularize point vortex periodic orbits into periodic patch motion. Finally, Franzoi, Hmidi \& Montalto contructed periodic vortex patch pairs with oscillations around far away translating point vortex pairs in \cite{FHM26}.
\subsection{Main results}
Our first main contribution solves the inverse problem consisting in the non-existence of nontrivial V-states in non-radial real-analytic simply-connected bounded domains. The result reads as follows and constitute the most important novel contribution of the paper.
\begin{theo}\label{thm:inverse-pb}
\textbf{(Inverse problem for V-states : a rigidity criterion)}\\
Let $\mathbf{D}\subset\mathbb{R}^2$ be a bounded simply-connected domain with
real-analytic boundary. Assume that the Euler equations \eqref{Euler eq bounded domain} in
$\mathbf{D}$ admit a nontrivial uniformly rotating patch solution
$$\boldsymbol{\omega}(t,x)=\mathbf 1_{\mathcal{R}_{\Omega t}D_0}(x),\qquad\Omega\neq0,$$
where $D_0\Subset\mathbf{D}$ is a nonempty open simply-connected set and
$\mathcal{R}_\theta$ denotes the rotation of angle $\theta$ about the origin
$$\mathcal{R}_{\theta}\triangleq\begin{bmatrix}
    \cos(\theta) & -\sin(\theta)\\
    \sin(\theta) & \cos(\theta)
\end{bmatrix}.$$
Then,
$\mathbf{D}$ is a disc centered at the origin.
\end{theo}
\begin{remark} Let us make the following comments.
\begin{enumerate}
    \item The Theorem \ref{thm:inverse-pb} actually means that the nontrivial non-stationary V-states constructed inside the unit disc in \cite[Thm. 1 and 2]{HHHM15} or \cite[Thm. 6.1]{HHR24} are exceptional. Generally, the dynamics must be more complicated than a pure rotation. Indeed, the periodic vortex patch motions obtained in \cite[Thm. 1.1]{HHR24} following a point vortex orbits revealed rich dynamics. 
    \item The condition $\Omega\neq0$ is fundamental because the proof relies on the patch actually exploring a nontrivial family of rotated positions. Actually, nontrivial stationary vortex patches do exist in planar domains as can be observed for instance in the series of works \cite{CW22,LWZ19,WZ20}.
    \item The real-analyticity of the boundary is essential in our argumentation (generalized Schwarz's reflexion across real-analytic arcs for harmonic functions, real-analytic uniqueness extension theorem). The result is probably true also when assuming less regularity but we leave it as an open problem.
    \item The hypothesis that $D_0$ is simply-connected can be discussed. It might be true that our proof works (up to slight modifications taking into account the different values of the relative stream function at the patch boundaries) in the multiply-connected situation, namely when $D_0$ has holes. However, it is unlikely that the same result holds true in multiply-connected ambient domains.
\end{enumerate}
\end{remark}

We also prove is this study, a singular version of the previous theorem for point vortices.

\begin{theo}\label{thm:rigidityPV}
\textbf{(Inverse problem for point vortex orbits)}\\
    Let $\mathbf{D}\subset\mathbb{R}^2$ be a simply-connected bounded domain with $C^1$ boundary. Assume that there exists, within $\mathbf{D}$ a circular point vortex orbit. Then, $\mathbf{D}$ is a disc.
\end{theo}
\begin{remark}
    Our proof exploit some $C^1$ regularity of the boundary since we make use of the behavior of the Robin function near the boundary. However, we strongly believe that this hypothesis can be removed with more sophisticated arguments. 
\end{remark}
Our next goal is to expose our last main result dealing with the construction of quasi-periodic vortex patches in domains that are perturbations of the unit disc. The analysis will be done in a perturbative way from the unit disc which has been recently studied in \cite{HR21-1}. The main result of this work reads as follows.
\begin{theo}\label{thm:QPS_in_nonradial_domains}\textbf{(Quasi-periodic vortex patches in non-radial domains)}\\
	Let $0<b_0<b_1<1$ and  $\mathbb{S}\subset\mathbb{N}^*$ a finite set.
	Fix some amplitudes $(\mathtt{a}_j)_{j\in\mathbb{S}}$, there exists  $\varepsilon_{0}\in(0,1)$ small enough such that for any $\varepsilon\in(0,\varepsilon_0)$, for any domain $\mathbf{D}$ which is a perturbation of the unit disc and associated to a conformal mapping $\boldsymbol{\Phi}:\mathbf{D}\rightarrow\mathbb{D}$ in the form
	$$\boldsymbol{\Phi}(x)=x+\varepsilon^2f_{\varepsilon}(x),\qquad f_{\varepsilon}\textnormal{ analytic in }\mathbf{D},\qquad f_{\varepsilon}(0)=0,\qquad\overline{f_{\varepsilon}(x)}=f_{\varepsilon}(\overline{x}),\qquad\varepsilon\mapsto f_{\varepsilon}\textnormal{ is continuous}$$
	there exists a Cantor-like set $\mathscr{C}_{\infty}^{\varepsilon}\subset(b_{0},b_{1})$ with asymptotically full Lebesgue  measure as $\varepsilon\rightarrow 0,$ i.e.
	$$\lim_{\varepsilon\rightarrow 0}|\mathscr{C}_{\infty}^{\varepsilon}|=b_{1}-b_{0},$$
	such that for any $b\in\mathscr{C}_{\infty}^{\varepsilon}$, the  equation \eqref{Euler eq bounded domain} admits a time quasi-periodic vortex patch solution with diophantine frequency vector ${\omega}_{\tiny{\textnormal{p}}}(b,\varepsilon)\triangleq \big(\omega_{j}(b,\varepsilon)\big)_{j\in\mathbb{S}}\in\mathbb{R}^{|\mathbb{S}|}$ and taking  the form
	$$t\mapsto\mathbf{1}_{D_{t}^{\varepsilon}},\qquad D_{t}^{\varepsilon}=\big\{\ell e^{\ii\theta},\, 0\leqslant \theta\leqslant 2\pi,\, 0\leqslant \ell< R_{\varepsilon}(b,t,\theta) \big\},\qquad R_{\varepsilon}(b,t,\theta)\triangleq \sqrt{b^2+2\varepsilon r(t,\theta)},$$
	with
	\begin{equation}\label{QP:sol:thm}
	    r(t,\theta)=\sum_{j\in\mathbb{S}}{\mathtt{a}_{j}}\cos\big(j\theta+\omega_{j}(b,\varepsilon)t\big)+\mathbf{p}\big({\omega}_{\tiny{\textnormal{p}}}(b,\varepsilon)t,\theta\big).
	\end{equation}
	The frequency vector $\mathtt{\omega}_{\tiny{\textnormal{p}}}(b,\varepsilon)$ satisfies
	$$\mathtt{\omega}_{\tiny{\textnormal{p}}}(b,\varepsilon)\underset{\varepsilon\rightarrow 0}{\longrightarrow}\big(-\Omega_{j}(b)\big)_{j\in\mathbb{S}},\qquad \Omega_{j}(b)\triangleq\frac{1}{2}\Big(j-1+b^{2j}\Big).$$ 
	The perturbation $\mathbf{p}:\mathbb{T}^{|\mathbb{S}|+1}\to\mathbb{R}$ is an even function satisfying for some large index of regularity $s=s(\mathbb{S})>0.$
	$$\| \mathbf{p}\|_{H^{{s}}(\mathbb{T}^{|\mathbb{S}|+1},\mathbb{R})}\underset{\varepsilon\rightarrow 0}{=}o(\varepsilon).$$
	In addition, the solutions are linearly stable.
\end{theo}
\begin{remark}
    Let us make the following comments.
    \begin{enumerate}
        \item The Theorem \ref{thm:QPS_in_nonradial_domains} constitutes the first construction of quasi-periodic vortex patch structures in domains which are non-radial (non $\mathbb{R}^2$ and non $\mathbb{D}$). Note that not necessarily the boundary of the ambient domain must be real-analytic.
        \item Combined with the Theorem \ref{thm:inverse-pb}, this shows that, in the special periodic case $|\mathbb{S}|=1$ and for real-analytic boundary, the solutions that are constructed in Theorem \ref{thm:QPS_in_nonradial_domains} do NOT correspond to V-states, namely to a purely rotating domain. Oscillations of the patch boundary must occur during the time evolution.
        \item The assumption $$\overline{f_{\varepsilon}(x)}=f_{\varepsilon}(\overline{x})\qquad\textnormal{i.e.}\qquad\overline{\boldsymbol{\Phi}(x)}=\boldsymbol{\Phi}(\overline{x})$$ corresponds to a symmetry of the ambient domain $\mathbf{D}$, namely it possesses an axis of symmetry (1-fold). This geometric condition is essential in order ensure the reversibility of the underlying Hamitonian dynamics, an essential fact to run KAM reducibility techniques. 
        \item The link between the amplitude $\varepsilon$ of the solutions and the scale deformation $\varepsilon^2$ of the domain allows to perform a pertubated proof compared to \cite{HR21-1}, the equilibrium spectrum corresponding again exactly to the bifurcation points of V-states from the Rankine vortices. It is not clear how to run a proof without this coherent scaling.
        \item Here, linear stability refers to the fact that the Floquet exponents associated with the linearized dynamics around the quasi-periodic solution \eqref{QP:sol:thm} are all purely imaginary. This property is not established separately, but rather emerges naturally from the construction of the solution. More precisely, the reducibility argument developed in Section \ref{sec:reductionKAM}, combined with the Berti--Bolle reduction in action-angle-normal variables \cite{BB15}, transforms the linearized system into a constant-coefficient system whose spectrum is purely imaginary. We refer to \cite[Thm. 1.2]{BM18} for a detailed discussion of this argument.
        \item Assume that the ambient domain $\mathbf{D}$ is $\mathbf{m}$-fold symmetric for a given $\mathbf{m}\geqslant2$, that is
        $$f_{\varepsilon}(x)=\sum_{n=0}^{\infty}\underbrace{a_{n}(\varepsilon)}_{\in\mathbb{R}}x^{\mathbf{m}n+1},\qquad\textnormal{i.e.}\qquad\boldsymbol{\Phi}\left(e^{\frac{2\ii\pi}{\mathbf{m}}}x\right)=e^{\frac{2\ii\pi}{\mathbf{m}}}\boldsymbol{\Phi}(x).$$
        Then, one can show that the Hamiltonian governing the contour dynamics equation in Lemma \ref{lem eq EDP r} is invariant under $\tfrac{2\pi}{\mathbf{m}}$-space translations, see Remark \ref{rem:mfold}. Therefore, similarly to \cite{HHR23}, one can run the KAM/Nash-Moser scheme within the Sobolev $\mathbf{m}$-fold function scale and ensure that the final vortex patch solution is $\mathbf{m}$-fold symmetric. In order to avoid useless complexity, we do not perform it here, but just mention that it can be done. 
    \end{enumerate}
\end{remark}
\subsection{Ideas of proofs}
We conclude this introductive section by briefly explaining the key ideas and main steps in the proofs of Theorem \ref{thm:inverse-pb}, Theorem \ref{thm:rigidityPV} and Theorem \ref{thm:QPS_in_nonradial_domains}.\\

Let us start with Theorem \ref{thm:inverse-pb}. It is a proof by contradiction assuming that the ambient domain is not a disc.
We compute in two different ways the circulation of the initial velocity field 
$$u_0\triangleq\mathbf{v}(0,\cdot)=\nabla^{\bot}\phi_0,\qquad\phi_0\triangleq\boldsymbol{\Psi}(0,\cdot)$$
along a Jordan curve $\Gamma$ surrounding the initial patch $D_0.$ Direct computation based on Stokes' Theorem shows that
$$\oint_{\Gamma} u_0\cdot dp=\int_{\operatorname{int}(\Gamma)}\operatorname{curl}(u_0)(x)dx=\int_{D_0}1\,dx=|D_0|>0.$$
The rest of the work consists in calculating it in another way and prove that it is zero, therefore arriving at a contradiction. First, observe that a non-circular domain must possess a transversal direction, namely there exists an element (actually a continuum) $p\in\partial\mathbf{D}$ such that $p^{\bot}$ is not tangent to $\partial\mathbf{D}.$ This allows, in Lemma \ref{lem:arc} to construct a whole family of arcs $\theta\mapsto I_{\theta}$, satisfying 
$$I_{\theta}\subset\partial\mathbf{D}_{\theta}\cap\mathbf{D},\qquad\mathbf{D}_{\theta}\triangleq\mathcal{R}_{-\theta}\mathbf{D},\qquad|\theta|\ll1,$$
obtained by small rotations of open arc surrounding $p.$ On each arc we show the existence of a constant $a_{\theta}\in\mathbb{R}$ such that 
$$\phi_0\equiv a_{\theta},\qquad\textnormal{on }I_{\theta}.$$
In Lemma \ref{lem:phi0_constant}, we prove that $\phi_0$ extends to a whole neighborhood of $\partial\mathbf{D}$ and therefore to any boundary $\partial\mathbf{D}_{\theta}$, for small $\theta$. Using that $\partial\mathbf{D}$ is real-analytic and that $\phi_0$ takes the constant value $a_{\theta}$ on the open arc $I_{\theta}\subset\partial\mathbf{D}_{\theta}$, then 
$$\phi_{0}\equiv a_{\theta},\qquad\textnormal{on }\partial\mathbf{D}_{\theta}.$$
Now, $\partial\mathbf{D}$ and $\partial\mathbf{D}_{\theta}$ being two Jordan curves enclosing two domains with same area and with a common point (the center of rotation), both curves must intersect at least once. Since $\phi_0$ vanishes on $\partial\mathbf{D}$ by the Dirichlet reformulation \eqref{Dirichlet_stream} of the impermeability condition, then $a_{\theta}=0$ and
$$\phi_0\equiv0,\qquad\textnormal{on }\partial\mathbf{D}_{\theta}.$$
The family of arcs $\theta\mapsto I_{\theta}$ actually draws, via local inversion, a small open subset of the connected set $\mathbf{D}\setminus\overline{D}_0$ where the function $\phi_0$ vanishes. The initial stream function $\phi_0$ being also harmonic in $\mathbf{D}\setminus\overline{D}_0,$ then, by unique continuation,
$$\phi_{0}\equiv0,\qquad\textnormal{on }\mathbf{D}\setminus\overline{D}_0.$$
The Jordan curve $\Gamma$ being contained in $\mathbf{D}\setminus\overline{D}_0$ and $u_0=\nabla^{\bot}\phi_0\equiv0$ there, we deduce that
$$\oint_{\Gamma}u_0\cdot dp=0.$$
Contradiction.\\

Next, we provide the key steps in the proof of Theorem \ref{thm:rigidityPV}. Assume that there is a circular point vortex orbit $\partial(r\mathbb{D})$ with $r>0.$ By the Hamiltonian nature of the point vortex dynamics, this orbit corresponds to an energy level of the Robin function
$$\mathcal{R}_{\mathbf{D}}=\lambda,\quad\textnormal{on }\partial(r\mathbf{D}).$$
The Robin function solves the Liouville equation
$$\Delta\mathcal{R}_{\mathbf{D}}=4e^{2\mathcal{R}_{\mathbf{D}}},\quad\textnormal{in }\mathbf{D}.$$
By means of maximum principle, we show that $\mathcal{R}_{\mathbf{D}}$ must be radial on $r\mathbb{D}.$ Then, real-analyticity of the Robin function allows to extend this radial property to the whole domain $\mathbf{D}$. Finally, combining this last property with the boundary behavior of the Robin function, we conclude that the domain must be circular.\\

Let us now give some elements of the proof of Theorem \ref{thm:QPS_in_nonradial_domains}. The proof is perturbative with respect to the undeformed circular domain \cite{HR21-1} and, similarly, is based on a Nash-Moser iteration coupled with KAM reducibility arguments. The assumption
$$\boldsymbol{\Phi}(x)=x+\varepsilon^2f_{\varepsilon}(x)$$ allows to write, at least far from the boundary, the stream function of the domain $\mathbf{D}$ as a perturbation of the one in the unit disc through the following Green kernel identity
$$G_{\mathbf{D}}(x,y)=G_{\mathbb{D}}(x,y)+\varepsilon^2 G_p^{\varepsilon}(x,y),$$
where the perturbed Green function $G_{p}^{\varepsilon}$ is smooth. We look for vortex patch solutions near the Rankine patch of radius $b\in(b_0,b_1)\Subset(0,1)$ in the form
$$\boldsymbol{\omega}(t,\cdot)=\mathbf{1}_{D_t^{\varepsilon}},\qquad D_t^{\varepsilon}\triangleq\Big\{\ell e^{\ii\theta},\quad 0\leqslant \theta\leqslant 2\pi,\quad 0\leqslant \ell< \sqrt{b^2+2\varepsilon r(t,\theta)}\Big\}.$$
The deformation function $r$ solves the Hamiltonian equation, related to the pseudo-kinetic energy $H$,
$$\partial_t r=\frac{1}{\varepsilon^2}\partial_{\theta}\nabla H(r),\qquad H(r)(t)\triangleq-\frac{1}{4\pi^2}\int_{D_t}\int_{D_t}G_{\mathbf{D}}(x,y)dxdy.$$
The symmetry assumption $\overline{\boldsymbol{\Phi}(x)}=\boldsymbol{\Phi}(\overline{x})$ implies that the Hamiltonian $H$ is reversible, meaning that if $(t,\theta)\mapsto r(t,\theta)$ is a solution, then so is $(t,\theta)\mapsto r(-t,-\theta)$. Note that since $\varepsilon^2G_p^{\varepsilon}$ is smooth and small, its contribution to the dynamics is only pertubative with respect to \cite{HR21-1}. The parameter $\varepsilon$ encodes both the geometry deformation of the domain and the amplitude of the solutions. Hence, the linearization at $\varepsilon=0$ corresponds to the linearization at the Rankine patch $b\cdot\mathbb{D}$ within the circular domain. According to \cite[Lem. 3.2 and Prop. 3.1]{HR21-1}, this latter possesses, for almost every $b$, a family of reversible quasi-periodic solutions
\begin{equation}\label{EQU:SOL}
    h(t,\theta)=\sum_{j\in\mathbb{S}}\mathtt{a}_j\cos\big(j\theta-\Omega_{j}(b)t\big),\qquad\mathbb{S}\subset\mathbb{N}^*\textnormal{ finite},\qquad\Omega_{j}(b)\triangleq\frac{j}{2|j|}\big(|j|-1+b^{2|j|}\big).
\end{equation}
Fixing a finite set $\mathbb{S}$, with $d\triangleq|\mathbb{S}|\in\mathbb{N}^*,$ of Fourier modes as above allows to split the phase space into tangential and normal parts. On the tangential part, we introduce actions ($I\in\mathbb{R}^d$) and angles ($\vartheta\in\mathbb{T}^d$) variables as symplectic polar change of variable on the Fourier coefficients
$$r=\mathbf{A}(\vartheta,I,z)\triangleq\sum_{j\in\mathbb{S}}\sqrt{\mathtt{a}_j^2+jI_j}e^{\ii\vartheta_j}+z,\qquad z=\textnormal{ normal variable},$$
bringing the problem to the search for zeros, embedded tori carrying a quasi-periodic trajectory with frequency vector $\omega$, of a nonlinear functional
$$\mathscr{F}(i,\omega)=0,\qquad\mathscr{F}(i,\omega)\triangleq\omega\cdot\partial_{\varphi}i-X_{H\circ\mathbf{A}}(i),\qquad i:\varphi\in\mathbb{T}^d\mapsto\triangleq\big(\vartheta(\varphi),I(\varphi),z(\varphi)\big),$$
where $X_H$ is the vector field of the underlying Hamiltonian system.
We find the desired solution by implementing a Nash-Moser iteration starting with the linear reversible quasi-periodic solution \eqref{EQU:SOL} as first approximate solution, seen in the action-angle-normal variables as the trivial embedding $\varphi\mapsto(\varphi,0,0)$. The Nash-Moser scheme requires, at each step, the construction of an approximate right inverse to the linearized operator at the current approximate solutions. It is performed by implementing the classical Berti-Bolle approach \cite{BB15} coupled with a KAM reducibility of the linearized operator in the normal directions. 
When linearizing, all the pertubative terms, coming from $\varepsilon^2G_p^{\varepsilon}$ correspond to either small size smoothing integral operators or small size local perturbation in the variable coefficient transport term. Therefore, from there the rest of the proof is  very similar to the unperturbed domain case \cite{HR21-1}. The interested reader is referred to this reference for more details.\\

\noindent\textbf{Acknowledgements :}
\begin{enumerate}[label=\textbullet]
    \item Supported by the ERC STARTING GRANT 2021 “Hamiltonian Dynamics, Normal Forms and Water Waves” (HamDyWWa), Project Number: 101039762.\\
    The views and opinions expressed are however those of the authors only and do not necessarily reflect those of the European Union or the European Research Council. Neither the European Union nor the granting authority can be held responsible for them.
    \item Also supported by INdAM - GNAMPA Project “(In)-stability in Fluid Mechanics”, CUP E53C25002010001.
    \item The author also would like to thank Taoufik Hmidi for useful comments and suggestions during the preparation of the manuscript.
\end{enumerate} 

\section{Rigidity results}
Here we provide the proofs of both Theorem \ref{thm:inverse-pb} and Theorem \ref{thm:rigidityPV}.
\subsection{The inverse problem for vortex patches}
This subsection is devoted to the proof of Theorem \ref{thm:inverse-pb}.\\

Let us consider $\mathbf{D}\subset\mathbb{R}^2$ an open bounded simply-connected domain of $\mathbb{R}^2$ with real-analytic boundary. Assume that there exists a non-trivial V-state rotating with an angular velocity $\Omega\in\mathbb{R}^*$ inside of $\mathbf{D}$, namely
\begin{equation}\label{V-state:ansatz}
    \boldsymbol{\omega}(t,\cdot)=\mathbf{1}_{D_t},\qquad D_t=\mathcal{R}_{\Omega t}D_0,\qquad\mathcal{R}_{\theta}\triangleq\begin{bmatrix}
    \cos(\theta) & -\sin(\theta)\\
    \sin(\theta) & \cos(\theta)
\end{bmatrix},
\end{equation}
where $D_0$ is a nonempty simply-connected open subset strict of $\mathbf{D}.$
\begin{lem}\label{lem:0inD}
    We have necessarily $0\in\mathbf{D}.$
\end{lem}
\begin{proof}
    Take $x\in D_0$. Since
$\mathcal{R}_{\Omega t}x\in\mathbf{D}$ for every $t\in\mathbb{R}$, the whole circle
$$\mathcal{C}_{x}\triangleq\{\mathcal{R}_\theta x,\,\,\theta\in\mathbb{R}\}$$
is contained in $\mathbf{D}$. Since $\mathbf{D}$ is simply-connected, the disc $\operatorname{int}(\mathcal{C}_{x})$ (centered in $0$) formed of the interior of the circle $\mathcal{C}_x$ is
contained in $\mathbf{D}$. This provides the desired result.
\end{proof}
In a similar way to \cite{HMV13}, integrating the contour dynamics provides the following stream function reformulation of the problem
\begin{equation}\label{stream_ref_Vstates}
    \forall t\in\mathbb{R},\quad\forall x\in\partial D_0,\quad \boldsymbol{\Psi}\big(t,\mathcal{R}_{\Omega t}x\big)=\frac{\Omega}{2}|x|^2+c(t),\qquad c(t)\in\mathbb{R}.
\end{equation}
Fix $\theta\in\mathbb{R}$. Define the rotated domain 
$$\mathbf{D}_{\theta}\triangleq\mathcal{R}_{-\theta}\mathbf{D}$$ 
and the following function on $\mathbf{D}_{\theta}$
$$\forall x\in\mathbf{D}_{\theta},\quad\phi_\theta(x)\triangleq\boldsymbol{\Psi}\left(\frac{\theta}{\Omega},\mathcal{R}_\theta x\right).$$
\begin{lem}\label{lem:htheta}
Denote $\mathbf{U}_{\theta}$ the connected component of $\mathbf{D}\cap \mathbf{D}_{\theta}$ containing $D_0$.
    Consider the function $h_{\theta},$ defined on $\mathbf{U}_{\theta}$ by
$$h_\theta\triangleq\phi_0-\phi_\theta.$$
Then, there exists a constant $a_{\theta}\in\mathbb{R}$ such that
$$h_{\theta}\equiv a_{\theta},\qquad\textnormal{in }\mathbf{U}_{\theta}.$$
\end{lem}
\begin{proof}
Recall the following classical property of the Laplacian: for any orthogonal transformation $A\in O_2(\mathbb{R}),$ one has
$$\Delta\big(f\circ A\big)=(\Delta f)\circ A.$$
Using that $\mathcal{R}_{\theta}$ is an orthogonal transformation, \eqref{elliptic_eq:stream-intro} and \eqref{V-state:ansatz}, we find
\begin{equation}\label{Laplace:phitheta}
    \begin{aligned}
    \forall x\in\mathbf{D}_{\theta},\quad\Delta\phi_{\theta}(x)&=\Delta\left[\boldsymbol{\Psi}\left(\tfrac{\theta}{\Omega},\mathcal{R}_{\theta}x\right)\right]\\
    &=(\Delta\boldsymbol{\Psi})\left(\tfrac{\theta}{\Omega},\mathcal{R}_{\theta}x\right)\\
    &=\mathbf{1}_{\mathcal{R}_{\theta}D_0}(\mathcal{R}_{\theta}x)\\
    &=\mathbf{1}_{D_0}(x).
\end{aligned}
\end{equation}
In addition, since $\partial\mathbf{D}_{\theta}=\mathcal{R}_{-\theta}\partial\mathbf{D},$ then in view of \eqref{Dirichlet_stream}, one has
\begin{equation}\label{Diri:phith}
    \phi_{\theta}=0,\qquad\textnormal{on }\partial\mathbf{D}_{\theta}.
\end{equation}
Moreover, evaluating \eqref{stream_ref_Vstates} at time $t=\tfrac{\theta}{\Omega},$ we find
$$\forall x\in\partial D_0,\quad\phi_\theta(x)-\frac{\Omega}{2}|x|^2=c_\theta,\qquad c_{\theta}\triangleq c\left(\tfrac{\theta}{\Omega}\right).$$
This being true for any $\theta\in\mathbb{R},$ substracting the time zero relation gives
\begin{equation}\label{htheta:bdry}
    h_{\theta}=\phi_0-\phi_\theta=c_0-c_\theta\triangleq a_{\theta},
    \qquad\textnormal{on }\partial D_0.
\end{equation}
Remark that $\mathbf{U}_{\theta}$ is nonempty thanks to Lemma \ref{lem:0inD}. 
According to \eqref{Laplace:phitheta}, we have
$$\Delta h_{\theta}=\Delta\phi_0-\Delta\phi_{\theta}=\mathbf{1}_{D_0}-\mathbf{1}_{D_0}=0,\qquad\textnormal{in }\mathbf{U}_{\theta}.$$
That is, the function $h_\theta$ is harmonic in $\mathbf{U}_{\theta}$ and in particular in $D_0$. The boundary condition \eqref{htheta:bdry} together with the maximum principle therefore imply
$$h_\theta=a_{\theta},\qquad\textnormal{in }D_0.$$
Introducing
$$g_{\theta}\triangleq h_{\theta}-a_{\theta},$$
then
$$\begin{cases}
    \Delta g_{\theta}=\Delta h_{\theta}=0, & \textnormal{in }\mathbf{U}_{\theta},\\
    g_{\theta}=0, & \textnormal{in }D_0.
\end{cases}$$
The function $g_{\theta}$ being harmonic, it is analytic \cite[Thm. 10, pp. 31]{Evans} in $\mathbf{U}_{\theta}.$ Moreover, it vanishes on the open nonempty set $D_0.$ Hence, by the uniqueness extension theorem
$$g_{\theta}\equiv0, \quad\textnormal{in }\mathbf{U}_{\theta},\qquad\textnormal{i.e.}\qquad h_{\theta}=a_{\theta}, \quad\textnormal{in }\mathbf{U}_{\theta}.$$
This concludes the proof of Lemma \ref{lem:htheta}.
\end{proof}
\begin{lem}\label{lem:arc}
    Assume that $\mathbf{D}$ is not a disc. Then, there exists a small angle $\theta_{\textnormal{max}}>0$ such that for any $\theta\in(-\theta_{\textnormal{max}},\theta_{\textnormal{max}})\setminus\{0\}$, there exists an open arc 
    $$I_{\theta}\subset\partial\mathbf{D}_{\theta}\cap\mathbf{D}.$$
    Moreover,
    $$\phi_0=a_{\theta}, \qquad\textnormal{on }I_{\theta}.$$
\end{lem}
\begin{proof}
$\blacktriangleright$ \textbf{Existence of a transversal direction :} Assume, in view of a contradiction, that
$$\forall p\in\partial\mathbf{D},\quad p^{\bot}\cdot\mathbf{n}_{\partial\mathbf{D}}(p)=0.$$
This means that $p^{\bot}$ is tangent to $\partial\mathbf{D}$. Let $\gamma:\mathbb{T}\to\partial\mathbf{D}$ be a parametrization of the boundary. Then, by hypothesis
$$\forall s\in\mathbb{T},\quad\gamma'(s)=\lambda(s)\gamma^{\bot}(s),\qquad\lambda(s)\in\mathbb{R}.$$
Hence,
$$\frac{d}{ds}|\gamma(s)|^2=2\gamma(s)\cdot\gamma'(s)=2\lambda(s)\gamma(s)\cdot\gamma^{\bot}(s)=0.$$
Therefore, $|\gamma|$ is constant, i.e. $\partial\mathbf{D}$ is a circle. This means that $\mathbf{D}$ is a disc. Contradiction. So there exists $p_0\in\partial\mathbf{D}$ such that
\begin{equation}\label{transverse:p0}
    p_0^{\bot}\cdot\mathbf{n}_{\partial\mathbf{D}}(p_0)\neq0.
\end{equation}
Actually, we can say more. The nornal vector can express in terms of the parametrization as
$$\forall s\in\mathbb{T},\quad\mathbf{n}_{\partial\mathbf{D}}\big(\gamma(s)\big)=-\frac{\big(\gamma'(s)\big)^{\bot}}{|\gamma'(s)|}\cdot$$
Therefore, for any $s\in\mathbb{T},$
$$\big(\gamma(s)\big)^{\bot}\cdot\mathbf{n}_{\partial\mathbf{D}}\big(\gamma(s)\big)=-\frac{\big(\gamma(s)\big)^{\bot}\cdot\big(\gamma'(s)\big)^{\bot}}{|\gamma'(s)|}=-\frac{\gamma(s)\cdot\gamma'(s)}{|\gamma'(s)|}=-\frac{1}{2|\gamma'(s)|}\frac{d}{ds}|\gamma(s)|^2.$$
Thus,
\begin{equation}\label{integral:transverse}
    \int_{\partial\mathbf{D}}p^{\bot}\cdot\mathbf{n}_{\partial\mathbf{D}}(p)dp=-\frac{1}{2}\int_{\mathbb{T}}\frac{d}{ds}|\gamma(s)|^2ds=0.
\end{equation}
The function $T:p\mapsto p^{\bot}\cdot\mathbf{n}_{\partial\mathbf{D}}(p)$ being continuous and non identically zero because of \eqref{transverse:p0}, we deduce from \eqref{integral:transverse} that the "transverse" function $T$ takes strictly positive and strictly negative values.

\medskip

$\blacktriangleright$ \textbf{Construction of the arc :} Let us treat the case
$$p_0^{\bot}\cdot\mathbf{n}_{\partial\mathbf{D}}(p_0)>0.$$
Denote $\mathtt{d}_{\partial\mathbf{D}}$ the signed distance function to $\partial\mathbf{D}$
$$\mathtt{d}_{\partial\mathbf{D}}(x)\triangleq\left(\mathbf{1}_{\mathbf{D}}(x)-\mathbf{1}_{\mathbb{R}^2\setminus\mathbf{D}}(x)\right)\inf_{p\in\partial\mathbf{D}}|x-p|.$$
Since $\partial\mathbf{D}$ is real-analytic, then the application $\mathtt{d}_{\partial\mathbf{D}}$ is smooth in a tubular neighborhood $\mathcal{V}$ of $\partial\mathbf{D}$, see \cite[Thm. 1]{D23}, and satisfies the following properties 
$$\begin{cases}
  \mathtt{d}_{\partial\mathbf{D}}>0,&\textnormal{in }\mathbf{D}\cap\mathcal{V},\\
  \mathtt{d}_{\partial\mathbf{D}}=0,&\textnormal{on }\partial\mathbf{D},\\
  \mathtt{d}_{\partial\mathbf{D}}<0,&\textnormal{in }(\mathbb{R}^2\setminus\mathbf{D})\cap\mathcal{V}
\end{cases}$$
and, the boundary $\partial\mathbf{D}$ corresponding to the zero level set of $\mathtt{d}_{\partial\mathbf{D}}$, we have
\begin{equation}\label{grad_dist}
    \forall p\in\partial\mathbf{D},\quad\nabla\mathtt{d}_{\partial\mathbf{D}}(p)=-\mu(p)\mathbf{n}_{\partial\mathbf{D}}(p),\qquad\mu(p)>0.
\end{equation}
By continuity of the family of matrices $\theta\mapsto\mathcal{R}_{\theta}$ and compactness of $\partial\mathbf{D}$ there exists $\theta_{\textnormal{max}}>0$ such that
$$\forall\theta\in(-\theta_{\textnormal{max}},\theta_{\textnormal{max}}),\quad\partial\mathbf{D}_{\theta}\subset\mathcal{V}.$$
Therefore, the following function $F$ is well-defined and smooth on $(-\theta_{\textnormal{max}},\theta_{\textnormal{max}})\times\partial\mathbf{D}$ 
$$F(\theta,p)\triangleq\mathtt{d}_{\partial\mathbf{D}}(\mathcal{R}_{-\theta}p).$$
In particular, we have the equivalence
\begin{equation}\label{inD_cond}
    F(\theta,p)>0\quad\Longleftrightarrow\quad\mathcal{R}_{-\theta}p\in\mathbf{D}.
\end{equation}
Note that at $(0,p_0)$, we have
$$F(0,p_0)=0$$
and, by \eqref{grad_dist},
\begin{align*} 
\partial_\theta F(0,p_0)&=-\nabla\mathtt{d}_{\partial\mathbf{D}}(p_0)\cdot p_0^\bot\\
&=\mu(p_0)\mathbf{n}_{\partial\mathbf{D}}(p_0)\cdot p_0^{\bot}\\
&>0. 
\end{align*}
Therefore, by Taylor formula, we deduce that there exists an open arc $A\subset\partial\mathbf{D}$ with $p_0\in A$ such that, up to reducing the value of $\theta_{\textnormal{max}},$
$$\forall(\theta,p)\in(0,\theta_{\textnormal{max}})\times A,\quad F(\theta,p)>0.$$
Equivalently, by \eqref{inD_cond},
$$\forall(\theta,p)\in(0,\theta_{\textnormal{max}})\times A,\quad\mathcal{R}_{-\theta} p\in\mathbf{D}.$$
But,
$$\mathcal{R}_{-\theta}A\subset\mathcal{R}_{-\theta}(\partial\mathbf{D}) =\partial(\mathcal{R}_{-\theta}\mathbf{D})=\partial\mathbf{D}_{\theta}.$$
Hence, setting
$$I_{\theta}\triangleq \mathcal{R}_{-\theta}A,$$
we obtain that $I_{\theta}$ is an open arc as a rotation of the open arc $A$ and 
$$\forall\theta\in(0,\theta_{\textnormal{max}}),\quad I_{\theta}\subset\partial\mathbf{D}_{\theta}\cap\mathbf{D}.$$
The case 
$$p_0^{\bot}\cdot\mathbf{n}_{\partial\mathbf{D}}(p_0)<0$$
reaches the negative values $(-\theta_{\textnormal{max}},0).$

\medskip

    $\blacktriangleright$ \textbf{Value of the initial stream function on the arc :} Finally, by virtue of \eqref{Diri:phith} and Lemma \ref{lem:htheta}, we obtain $$\phi_0=a_{\theta},\qquad\textnormal{on }I_\theta.$$
    This completes the proof of Lemma \ref{lem:arc}.
\end{proof}
 In the next lemma, we show that actually, $\phi_0$ is constant on the whole rotated boundaries. 

\begin{lem}\label{lem:phi0_constant}
     Assume that $\mathbf{D}$ is not a disc. There exists a neighborhood $\mathcal{N}(\partial\mathbf{D})$ of the boundary $\partial\mathbf{D}$ and $\theta_{\textnormal{max}}>0$ such that 
     $$\forall\theta\in(-\theta_{\textnormal{max}},\theta_{\textnormal{max}}),\quad\partial\mathbf{D}_{\theta}\subset\mathcal{N}(\partial\mathbf{D}).$$
     The function $\phi_0$ admits a unique extension (still denoted $\phi_0$) on $\mathcal{N}(\partial\mathbf{D})$ and we have
     $$\phi_0\equiv a_{\theta}
    \qquad\textnormal{on }\partial\mathbf{D}_{\theta}.$$
\end{lem}
\begin{proof}
$\blacktriangleright$ \textbf{Extension of $\phi_0$ :} By \eqref{Laplace:phitheta}, we have
$$\Delta\phi_0=\mathbf 1_{D_0}.$$
Since $D_0\Subset\mathbf{D}$, there is a positive distance
$$d\triangleq\operatorname{dist}(D_0,\partial\mathbf{D})>0,\qquad\operatorname{dist}(A,B)\triangleq\inf_{a\in A\atop b\in B}|a-b|.$$
Hence, the collar
$$\mathbf{W}_d\triangleq\{x\in\mathbf{D}\quad\textnormal{s.t.}\quad\operatorname{dist}(x,\partial\mathbf{D})<\tfrac{d}{2}\}$$
does not intersect $D_0$
$$\mathbf{W}_d\cap D_0=\varnothing.$$
Therefore,
$$\Delta\phi_0=0,\qquad\textnormal{in }\mathbf{W}_d.$$
Thus, $\phi_0$ is harmonic in a full neighborhood of the outer boundary, on its interior side. Moreover, by our normalization,
$$\phi_0=0,\qquad\textnormal{on }\partial\mathbf{D}.$$
Since $\partial\mathbf{D}$ is real-analytic one can apply the harmonic extension theorem across real-analytic arcs \cite[Thm. 4.10, pp. 20]{C12} and obtain in particular the existence of a tubular neighborhood $\mathcal{N}(\partial\mathbf{D})$ of the boundary $\partial\mathbf{D}$ and a harmonic extension still denoted $\phi_0$ to $\mathcal{N}(\partial\mathbf{D}).$

\medskip

$\blacktriangleright$ \textbf{$\phi_0$ is constant on each rotated boundary :} Using the continuity of the family $\theta\mapsto\mathcal{R}_{\theta}$ of linear transformations, we get the existence of $\theta_{\textnormal{max}}>0$ such that
$$\forall\theta\in(-\theta_{\textnormal{max}},\theta_{\textnormal{max}}),\quad\partial\mathbf{D}_{\theta}\subset\mathcal{N}(\partial\mathbf{D}).$$
So the already-constructed extension of $\phi_0$ is defined on the entire family of slightly rotated boundaries. Now, fix $\theta\in(-\theta_{\textnormal{max}},\theta_{\textnormal{max}})$ and let $\gamma:\mathbb{T}\to\partial\mathbf{D}$ be a real-analytic parametrization of the boundary $\partial\mathbf{D}$. Then, the application $\gamma_{\theta}:\mathbb{T}\to\partial\mathbf{D}_{\theta}$ defined by
$$\gamma_\theta(s)\triangleq R_{-\theta}\gamma(s)$$
provides a real-analytic parametrization of $\partial\mathbf{D}_{\theta}$. On the other hand, the extended $\phi_0$ is harmonic in a neighborhood of $\partial\mathbf{D}_{\theta}$ and harmonic functions are analytic \cite[Thm. 10, pp. 31]{Evans}. Hence,
$$f_\theta(s)\triangleq\phi_0\big(\gamma_\theta(s)\big)=\phi_0\big(\mathcal{R}_{-\theta}\gamma(s)\big)$$
is a real-analytic function of the variable $s$. In addition, according to Lemma \ref{lem:arc}, the function $f_{\theta}$ takes the constant value $a_{\theta}$ on the open interval $\gamma_{\theta}^{-1}(I_{\theta}).$ By the uniqueness extension theorem for real-analytic functions, we deduce that
$$\phi_0\equiv a_{\theta},\qquad\textnormal{on }\partial\mathbf{D}_{\theta}.$$
This ends the proof of Lemma \ref{lem:phi0_constant}.
\end{proof}
The next step is to show that actually, the constant value of $\phi_0$ on the rotated boundaries is zero.
\begin{lem}\label{lem:phi0_vanish}
    Assume that $\mathbf{D}$ is not a disc. Then, for any $\theta\in(-\theta_{\textnormal{max}},\theta_{\textnormal{max}}),$
    $$\phi_0\equiv 0,
    \qquad\textnormal{on }\partial\mathbf{D}_{\theta}.$$
\end{lem}
\begin{proof}
 According to Lemma \ref{lem:arc}, in particular, we have (if $\theta_{\textnormal{max}}$ is small enough)
 $$\forall\theta\in(-\theta_{\textnormal{max}},\theta_{\textnormal{max}})\setminus\{0\},\quad\mathbf{D}_{\theta}\neq\mathbf{D}.$$ 
 By virtue of Lemma \ref{lem:0inD}, both $\partial\mathbf{D}$ and $\partial\mathbf{D}_{\theta}$ are Jordan curves enclosing the origin. They cannot be
disjoint, otherwise one of their interiors would be strictly contained in the other, which is impossible because rotations preserve area. Hence,
$$\partial\mathbf{D}\cap\partial\mathbf{D}_{\theta}\neq\varnothing.$$
Let $p\in\partial\mathbf{D}\cap\partial\mathbf{D}_{\theta}.$ By \eqref{Dirichlet_stream} and Lemma \ref{lem:phi0_constant}, we have 
$$0=\phi_0(p)=a_{\theta}.$$ 
As a consequence, using one more time Lemma \ref{lem:phi0_constant},
$$\phi_0=0,\qquad\textnormal{on }\partial\mathbf{D}_{\theta}$$
for all sufficiently small $\theta$ (the case $\theta=0$ being covered by \eqref{Dirichlet_stream}). This achieves the proof of Lemma \ref{lem:phi0_vanish}.
\end{proof}

\begin{proof}[Proof of Theorem \ref{thm:inverse-pb}]
Suppose, in view of a contradiction, that $\mathbf{D}$ is not a disc centered at the origin.

\medskip

$\blacktriangleright$ \textbf{$\phi_0$ vanishes on an open set inside $\mathbf{D}$ :} Following the proof of Lemma \ref{lem:arc}, see \eqref{transverse:p0}, there exists $p_0\in\partial\mathbf{D}$ such that $p_0^{\bot}$ is not tangent to
$\partial\mathbf{D}$. Let $\gamma:\mathbb{T}\to\partial\mathbf{D}$ be a smooth parametrization of $\partial\mathbf{D}$
with $\gamma(s_0)=p_0$ for some $s_0\in\mathbb{T}$, and define
$$G(s,\theta)\triangleq\mathcal{R}_{-\theta}\gamma(s).$$
At $(s_0,0)$,
$$G(s_0,0)=p_0,\qquad\partial_sG(s_0,0)=\gamma'(s_0),
    \qquad
    \partial_\theta G(s_0,0)=-p_0^{\bot}.$$
By construction, these last two vectors are linearly independent. Therefore,
$$\det\big(D_{(s,\theta)}G(s_0,0)\big)=\det\left(\begin{bmatrix}
    \gamma'(s_0) & -p_0^{\bot}
\end{bmatrix}\right)\neq0.$$
Hence, the local inverse function
theorem implies in particular the existence of a nonempty open set $\mathbf{V}$ satisfying, for some $\theta_{\textnormal{max}}>0,$ the inclusion 
$$\mathbf{V}\subset\mathbf{D}\cap\bigcup_{|\theta|<\theta_{\textnormal{max}}}\mathcal{R}_{-\theta}\partial\mathbf{D}.$$
For $\theta_{\textnormal{max}}$ sufficiently small, $\mathbf{V}$ is contained in
$\mathbf{D}\setminus\overline{D_0}$. Since $\phi_0=0$ on every
$\partial\mathbf{D}_{\theta}=\mathcal{R}_{-\theta}\partial\mathbf{D}$ with $|\theta|<\theta_{\textnormal{max}}$, we have
$$\phi_0=0,\qquad\textnormal{on }\mathbf{V}.$$

\medskip

$\blacktriangleright$ \textbf{Obtaining a contradiction :} Denote $\mathbf{U}$ the connected component of $\mathbf{D}\setminus\overline{D_0}$ containing $\mathbf{V}$. Since
$$\Delta\phi_0=0,\qquad\textnormal{in }\mathbf{U},$$
and vanishes on the open subset $\mathbf{V}\subset\mathbf{U},$ the unique continuation implies
$$\phi_0\equiv0,\qquad\textnormal{in }\mathbf{U}.$$
Consequently,
$$u_0\triangleq\nabla^\bot\phi_0\equiv0,\qquad\textnormal{in }\mathbf{U}.$$
Choose a smooth Jordan curve $\Gamma\subset\mathbf{U}$ enclosing $D_0$. Then, we have
\begin{equation}\label{circul_zero}
    \oint_{\Gamma} u_0\cdot dp=0.
\end{equation}
On the other hand, by Stokes' theorem,
$$\oint_{\Gamma} u_0\cdot dp=\int_{\operatorname{int}(\Gamma)}\operatorname{curl}(u_0)(x)dx.$$
Since
$$\operatorname{curl}(u_0)=\Delta\phi_0=\mathbf{1}_{D_0}\qquad\textnormal{and}\qquad D_0\subset\operatorname{int}(\Gamma),$$
we obtain
$$\oint_{\Gamma}u_0\cdot dp=\int_{D_0}1\,dx=|D_0|>0,$$
which is a contradiction with \eqref{circul_zero}. Therefore $\mathbf{D}$ must be a disc centered at the origin. This ends the proof of Theorem \ref{thm:inverse-pb}.
\end{proof}

\subsection{The inverse problem for point vortices}
This subsection is devoted to the proof of Theorem \ref{thm:rigidityPV}.\\

Let us consider $\mathbf{D}\subset\mathbb{R}^2$ an open bounded simply-connected domain of $\mathbb{R}^2$ with $C^1$ boundary. Assume that there exists a circular nontrivial (i.e. not a stationary point) point vortex orbit
$$\{t\mapsto\xi(t)\}=\partial(r\mathbb{D})\subset\mathbf{D},\qquad r>0.$$

\begin{proof} $\blacktriangleright$ \textbf{The Robin function is radial :}
Recall from \eqref{Hamiltonian_system_1PV} that the dynamics of a single point vortex is Hamiltonian, with Hamiltonian given, up to an irrelevant multiplicative constant, by the Robin function $\mathcal{R}_{\mathbf{D}}$. In particular, $\mathcal{R}_{\mathbf{D}}$ is constant along every vortex trajectory. Hence there exists an energy level $\lambda\in\mathbb{R}$ such that
\begin{equation}\label{energy_level}
    \forall x\in\partial(r\mathbb{D}),\quad \mathcal{R}_{\mathbf{D}}(x)=\lambda.
\end{equation}
Since $\partial(r\mathbb{D})\subset\mathbf{D}$ and $\mathbf{D}$ is simply-connected, then the whole disc $r\mathbb{D}$ is contained in $\mathbf{D}$. Recall from \eqref{Liouville intro} that the Robin function satisfies the Liouville equation
\begin{equation}\label{Liouville_Robin}
    \Delta\mathcal{R}_{\mathbf{D}}=4e^{2\mathcal{R}_{\mathbf{D}}},\qquad\textnormal{in }\mathbf{D}.
\end{equation}
Consider the radial solution $u$ of
\begin{equation}\label{pb:radial_u}
    \begin{cases}
    \Delta u=4e^{2u}, & \textnormal{in }r\mathbb{D},\\
    u=\lambda, & \textnormal{on }\partial(r\mathbb{D}).
\end{cases}
\end{equation}
Note that the function $u$ is explicitly given by
$$u(x)=\log\left(\frac{\mathtt{a}}{1-\mathtt{a}^2|x|^2}\right),\qquad\mathtt{a}\triangleq\frac{\sqrt{1+4e^{\lambda}r^2}-1}{2e^{\lambda}r^2}\cdot$$
Set
$$h\triangleq\mathcal{R}_{\mathbf{D}}-u.$$
Then, from \eqref{energy_level} and the second equation in \eqref{pb:radial_u}, we have
\begin{equation}\label{vanish_diff}
    h=0,\qquad\textnormal{on }\partial(r\mathbb{D}).
\end{equation} 
Besides, \eqref{Liouville_Robin} and the first equation in \eqref{pb:radial_u} give
\begin{equation}\label{Laplace_h}
    \Delta h=4\left(e^{2\mathcal{R}_{\mathbf{D}}}-e^{2u}\right).
\end{equation}
We shall now prove that $h$ identically vanishes by two different ways:
\begin{enumerate}[label=\textbullet]
    \item \textit{Proof 1:} For any $x\in r\mathbb{D},$ applying the mean-value theorem to the function $t\mapsto e^{2t}$ between $\mathcal{R}_{\mathbf{D}}(x)$ and $u(x),$ we obtain
$$e^{2\mathcal{R}_{\mathbf{D}}(x)}-e^{2u(x)}=2\big(\mathcal{R}_{\mathbf{D}}(x)-u(x)\big)\int_0^1e^{2u(x)+2t\big(\mathcal{R}_{\mathbf{D}}(x)-u(x)\big)}dt.$$
We deduce from \eqref{Laplace_h} that $h$ is solution to the equation
$$\Delta h-\mathtt{c}h=0,\qquad\mathtt{c}(x)\triangleq 8\int_0^1e^{2u(x)+2t\big(\mathcal{R}_{\mathbf{D}}(x)-u(x)\big)}dt>0.$$
The maximum principle \cite[Thm. 2, pp. 329]{Evans}, together with \eqref{vanish_diff}, yields
$$h\equiv0,\quad\textnormal{in }r\mathbb{D},\qquad\textnormal{i.e.}\qquad\mathcal{R}_{\mathbf{D}}\equiv u,\quad\textnormal{in }r\mathbb{D}$$
\item \textit{Proof 2:} Assume in view of a contradiction that $h$ is not identically zero on $r\mathbb{D}$. According to \eqref{vanish_diff}, $h$ takes strictly positive or strictly negative values inside $r\mathbb{D}.$ Let us treat the strictly positive case, the other one being similar. So, by compactness and continuity, there exists $x_0\in r\mathbb{D}$ such that
\begin{equation}\label{h_pos_max}
    \mathcal{R}_{\mathbf{D}}(x_0)-u(x_0)=h(x_0)>0\qquad\textnormal{and}\qquad h(x_0)=\max_{x\in r\mathbb{D}}h(x).
\end{equation}
Since $h$ realizes a maximum at $x_0$, we have
$$\Delta h(x_0)\leqslant0.$$
But, the real exponential being strictly increasing, we have from the first condition in \eqref{h_pos_max} and \eqref{Laplace_h},
$$\Delta h(x_0)=4\left(e^{2\mathcal{R}_{\mathbf{D}}(x_0)}-e^{2u(x_0)}\right)>0.$$
Contradiction. Hence,
$$h\equiv0,\quad\textnormal{in }r\mathbb{D},\qquad\textnormal{i.e.}\qquad\mathcal{R}_{\mathbf{D}}\equiv u,\quad\textnormal{in }r\mathbb{D}$$
\end{enumerate}
So $\mathcal{R}_{\mathbf{D}}$ is radial around $0$ on a nonempty open subset of $\mathbf{D}$. Since the Robin function is real-analytic in $\mathbf{D}$, the function
$$x\longmapsto x^{\bot}\cdot\nabla\mathcal{R}_{\mathbf{D}}(x)$$
is real-analytic in $\mathbf{D}$. It vanishes identically in $r\mathbb{D}$ because $\mathcal{R}_{\mathbf{D}}$ is radial there. Since $\mathbf{D}$ is connected, the identity principle for real-analytic functions gives
$$\forall x\in\mathbf{D},\quad x^{\bot}\cdot\nabla\mathcal{R}_{\mathbf{D}}(x)=0.$$
Therefore $\mathcal{R}_{\mathbf{D}}$ is radial on $\mathbf{D}$.

\medskip

$\blacktriangleright$ \textbf{A radial Robin function implies a radial domain :} Assume, in view of a contradiction that $\mathbf{D}$ is not a disc centered at $0$. Then, according to the proof of Lemma \ref{lem:arc} (actually, according to the given reference \cite[Thm. 1]{D23}, it is sufficient here to assume that $\partial\mathbf{D}$ is $C^{1}$), there exists $p_0\in\partial\mathbf{D}$ such that $p_{0}^{\bot}\cdot\mathbf{n}_{\partial\mathbf{D}}(p_0)>0$ and there exists $\theta_{\textnormal{max}}>0$ such that
$$\forall\theta\in(0,\theta_{\textnormal{max}}),\quad\mathcal{R}_{-\theta}p_0\in\mathbf{D}.$$
By the behavior of the Robin function at the boundary \eqref{Robin bndry}, we deduce that
$$\lim_{\theta\to0}\mathcal{R}_{\mathbf{D}}\big(\mathcal{R}_{-\theta}p_0\big)=+\infty.$$
But, since $\mathcal{R}_{\mathbf{D}}$ is radial, we have 
$$\forall\theta\in(0,\theta_{\textnormal{max}}),\quad\mathcal{R}_{\mathbf{D}}\big(\mathcal{R}_{-\theta}p_0\big)=\mathcal{R}_{\mathbf{D}}\big(\mathcal{R}_{-\frac{\theta_{\textnormal{max}}}{2}}p_0\big)\in\mathbb{R}.$$
So the continuous function $$\theta\mapsto\mathcal{R}_{\mathbf{D}}\big(\mathcal{R}_{-\theta}p_0\big)$$
is constant on the interval $(0,\theta_{\textnormal{max}})$ and tends to $+\infty$ at $0.$ Contradiction. Therefore $\mathbf{D}$ is a disc centered at $0$. This concludes the proof of Theorem \ref{thm:rigidityPV}.
























\end{proof}

\section{Quasi-periodic patches in non-radial domains}\label{sec:QP}
This section is devoted to the proof of Theorem \ref{thm:QPS_in_nonradial_domains}.

\subsection{The equation and its linearization}
Our first purpose is to set up the problem by writting down a contour dynamics equation for a vortex patch evolution in domains that are a pertubation of the unit disc. We show that the corresponding dynamics is Hamiltonian. Then, we study its linearization. By coupling the amplitude of the solution with the domain deformation, the equilibrium state corresponds to Rankine circular vortices within the unit disc. We recall that the linearization there already admits quasi-periodic solutions.
\subsubsection{Contours dynamics in nearly circular domains}
In the whole Section \ref{sec:QP}, we shall focus on the case where the simply-connected bounded domain $\mathbf{D}$ is a small perturbation of the unit disc $\mathbb{D}$. More precisely, we shall assume that the conformal mapping $\boldsymbol{\Phi}:\mathbf{D}\to\mathbb{D}$ writes
\begin{equation}\label{conf map}
	\boldsymbol{\Phi}(x)=x+\varepsilon^2f_{\varepsilon}(x),
\end{equation}
with $\varepsilon\in(0,1)$ a small parameter. We assume that the family $(f_{\varepsilon})_{\varepsilon\in(0,1)}$ satisfies
\begin{equation}\label{assumptions:feps}
    f_{\varepsilon}\textnormal{ is analytic on }\mathbf{D},\qquad f_{\varepsilon}(0)=0,\qquad f_{\varepsilon}(\overline{x})=\overline{f_{\varepsilon}(x)},\qquad\varepsilon\mapsto f_{\varepsilon} \textnormal{ is continuous}.
\end{equation}
The continuity assumption $\varepsilon\mapsto f_{\varepsilon}$ is understood as from a small symmetric interval $(-\varepsilon_0,\varepsilon_0)$ onto the Fréchet space of holomorphic functions in $\mathbf{D}.$ Typically (but not necessarily), we can choose
$$f_{\varepsilon}(x)=\sum_{n=0}^{\infty}a_n\varepsilon^{n}x^{n+1},\qquad a_n\in\mathbb{R}.$$
As a consequence of the symmetry (third) condition in \eqref{assumptions:feps}, 
\begin{equation}\label{symm:conform}
   \boldsymbol{\Phi}(\overline{x})=\overline{\boldsymbol{\Phi}(x)}, 
\end{equation}
which means that the domain $\mathbf{D}$ is $1$-fold, that is admits an axis of symmetry which can be chosen as the $(Ox)$ axis. This assumption is required later to check the reversibility of the Hamiltonian describing the dynamics of the patch. Now, remarking from \eqref{def:Green} that
\begin{equation}\label{sym G0}
	G_{\mathbb{D}}(\overline{z},\overline{\xi})=G_{\mathbb{D}}(z,\xi),
\end{equation}
we deduce from \eqref{symm:conform} that
\begin{equation}\label{sym G}
	G_{\mathbf{D}}(\overline{x},\overline{y})=G_{\mathbf{D}}(x,y).
\end{equation}
Observe that combining \eqref{def:Green} and \eqref{conf map}, we can write (at least formally)
\begin{align*}
	G_{\mathbf{D}}(x,y)&=\log\left|\frac{x-y+\varepsilon^2\big(f_{\varepsilon}(x)-f_{\varepsilon}(y)\big)}{1-\overline{y}x-\varepsilon^2\big(xf_{\varepsilon}(\overline{y})+\overline{y}f_{\varepsilon}(x)+\varepsilon^2f_{\varepsilon}(x)f_{\varepsilon}(\overline{y})\big)}\right|\\
	&=G_\mathbb{D}(x,y)+\log\left|1+\varepsilon^2\frac{f_{\varepsilon}(x)-f_{\varepsilon}(y)}{x-y}\right|-\log\left|1-\varepsilon^2\frac{xf_{\varepsilon}(\overline{y})+\overline{y}f_{\varepsilon}(x)+\varepsilon^2f_{\varepsilon}(x)f_{\varepsilon}(\overline{y})}{1-\overline{y}x}\right|\\
	&\triangleq G_\mathbb{D}(x,y)+\log\left|1+\varepsilon^2g_1(\varepsilon,x,y)\right|-\log\left|1+\varepsilon^2g_2(\varepsilon,x,y)\right|.
\end{align*}
Now, we shall give a meaning to the previous decomposition. First observe that $G_{\mathbb{D}}$ is only well-defined on $\mathbb{D}\times\mathbb{D},$ therefore, the previous spliting can only make sense there. Since $f_{\varepsilon}$ is in particular differentiable on $\mathbf{D},$ then $g_1(\varepsilon,\cdot,\cdot)$ is well-defined on $\mathbf{D}\times\mathbf{D}.$ However, $g_2(\varepsilon,\cdot,\cdot)$ is a priori well-defined only on $\mathbf{D}\times\mathbf{D}\setminus\big\{(x,y)\in\mathbf{D}\times\mathbf{D}\quad\textnormal{s.t.}\quad x\overline{y}=1\big\}.$ In particular, if we fix $a\in(0,1)$, such that
$$\overline{D(0,a)}\subset\mathbb{D}\cap\mathbf{D},\qquad D(0,a)\triangleq\big\{(u,v)\in\mathbb{R}^2\quad\textnormal{s.t.}\quad u^2+v^2<a^2\big\},$$
then $g_2(\varepsilon,\cdot,\cdot)$ is smooth on $\overline{D(0,a)}\times\overline{D(0,a)}.$ Hence, for $\varepsilon$ small enough, we can expand the $\log$ in power series and write for $(x,y)\in \overline{D(0,a)}\times\overline{D(0,a)},$
\begin{align*}
	\log\left|1+\varepsilon^2g_1(\varepsilon,x,y)\right|-\log\left|1+\varepsilon^2g_2(\varepsilon,x,y)\right|&=\varepsilon^2\sum_{n=1}^{\infty}\tfrac{(-\varepsilon^2)^{n-1}}{n}\textnormal{Re}\left\{\big(g_{1}(\varepsilon,x,y)\big)^n-\big(g_2(\varepsilon ,x,y)\big)^n\right\}\\
	&\triangleq\varepsilon^2 G_p^\varepsilon(x,y).
\end{align*}
Hence, we have the following expression
\begin{equation}\label{split G}
	\forall (x,y)\in\overline{D(0,a)}\times\overline{D(0,a)},\quad G(x,y)=G_{\mathbb{D}}(x,y)+\varepsilon^2 G_{p}^{\varepsilon}(x,y),
\end{equation}
with $G_{p}^{\varepsilon}:\overline{D(0,a)}\times\overline{D(0,a)}\to\mathbb{R}$ being a smooth function and $\varepsilon\mapsto G_{p}^{\varepsilon}$ being continuous. As a consequence, according to \eqref{Psi EDP}, choosing a vorticity $\boldsymbol{\omega}$ with (at least for a short time $T>0$)
$$\forall t\in[0,T],\quad\textnormal{supp}\big(\boldsymbol{\omega}(t,\cdot)\big)\subset\overline{D(0,a)},$$
the potential velocity can be split in the following way
\begin{equation}\label{split Psi}
	\forall (t,x)\in[0,T]\times\overline{D(0,a)},\quad\boldsymbol{\Psi}(t,x)=\boldsymbol{\Psi}_{\mathbb{D}}(t,x)+\varepsilon^2\boldsymbol{\Psi}_p^{\varepsilon}(t,x),
\end{equation}
with
\begin{equation}\label{PsiD Psip}
	\boldsymbol{\Psi}_{\mathbb{D}}(t,x)\triangleq\frac{1}{2\pi}\int_{\mathbf{D}}G_{\mathbb{D}}(x,y)\boldsymbol{\omega}(t,y)dy,\qquad\boldsymbol{\Psi}_{p}^{\varepsilon}(t,x)\triangleq\frac{1}{2\pi}\int_{\mathbf{D}}G_{p}^{\varepsilon}(x,y)\boldsymbol{\omega}(t,y)dy.
\end{equation}

Now, we fix  
\begin{equation}\label{Rankine_range}
    0<b_0<b_1<1.
\end{equation}
For a given $b\in (b_0,b_1),$ we consider a vortex patch $\boldsymbol{\omega}:t\mapsto \mathbf{1}_{D_t^{\varepsilon}}$ which is close to the Rankine vortex $\mathbf{1}_{b\cdot\mathbb{D}}.$ Recall that the latter is a stationary solution for Euler equations in the unit disc, namely the unperturbed domain.
Assume that $D_t^{\varepsilon}$ can be parametrized in polar coordinates in the following way
\begin{equation}\label{def Re}
	D_t^{\varepsilon}=\Big\{\ell e^{\ii\theta},\quad 0\leqslant \theta\leqslant 2\pi,\quad 0\leqslant \ell< R_{\varepsilon}(b,t,\theta) \Big\},\qquad R_{\varepsilon}(b,t,\theta)\triangleq \sqrt{b^2+2\varepsilon r(t,\theta)}.
\end{equation}
The parameter $\varepsilon$ therefore corresponds to the amplitude of the solutions and $r$ is the graph perturbation unknown. Notice that such parametrization is valid at least for short time since we are looking for a solution close to the stationary one in the unperturbed domain. In addition, for $\varepsilon$ and $t$ small enough, we can ensure
$$\textnormal{supp}\big(\boldsymbol{\omega}(t)\big)=D_t^{\varepsilon}\subset \overline{D\big(0,\tfrac{1+b_1}{2}\big)}\subset\mathbb{D}\cap\mathbf{D}.$$
According to the above discussion, this allows to guarantee the splittings \eqref{split G} and \eqref{split Psi}-\eqref{PsiD Psip} with $G_{p}^{\varepsilon}$ smooth on $\overline{D_t^{\varepsilon}}\times\overline{D_t^{\varepsilon}}.$ The next proposition states that the radial deformation $r$ satisfies a Hamiltonian equation related to the stream function.
\begin{lem}\label{lem eq EDP r}
Recall that $\boldsymbol{\Psi}$, $\boldsymbol{\Psi}_{\mathbb{D}}$, $\boldsymbol{\Psi}_p$ and $R_\varepsilon$ are defined in \eqref{Psi EDP}, \eqref{PsiD Psip} and \eqref{def Re}, respectively. Then, the following holds.
\begin{enumerate}[label=(\roman*)]
    \item For a small enough time $T$, the radial deformation $r$, defined through \eqref{def Re}, satisfies the following nonlinear and nonlocal transport PDE
	\begin{equation}\label{EDP eq r}
		\forall (t,\theta)\in[0,T]\times\mathbb{T},\quad\partial_{t}r(t,\theta)+\tfrac{1}{\varepsilon}\partial_{\theta}\Big(\boldsymbol{\Psi}\big(t,R_{\varepsilon}(b,t,\theta)e^{\ii\theta}\big)\Big)=0.
	\end{equation}
    \item The equation \eqref{EDP eq r} is Hamiltonian and can be written as 
\begin{equation}\label{HAM EDP}
	\partial_t r=\partial_\theta\nabla H(r) , \qquad H(r)\triangleq -\tfrac{1}{2\varepsilon^2}E(r),\qquad E(r)\triangleq E_{\mathbb{D}}(r)+\varepsilon^2E_p^{\varepsilon}(r),
\end{equation}
with
\begin{equation}\label{ED Ep}
	E_{\mathbb{D}}(r)(t)\triangleq \frac{1}{2\pi}\int_{D_t}\Psi_{\mathbb{D}} (t,x)dx,\qquad E_{p}^{\varepsilon}(r)(t)\triangleq \frac{1}{2\pi}\int_{D_t}\Psi_{p}^{\varepsilon}(t,x)dx
\end{equation}
and where $\nabla$ is the $L^{2}(\mathbb{T})$-gradient associated with the $L^{2}(\mathbb{T})$ normalized inner product
$$\big\langle r_{1},r_{2}\big\rangle_{L^{2}(\mathbb{T})}\triangleq\int_{\mathbb{T}}r_{1}(\theta)r_{2}(\theta)d\theta,\qquad\int_{\mathbb{T}}f(\eta)d\eta\triangleq \frac{1}{2\pi}\int_{0}^{2\pi}f(\eta)d\eta.$$
\item The equation \eqref{HAM EDP} is reversible in the sense that if $(t,\theta)\mapsto r(t,\theta)$ is a solution, then so is $(t,\theta)\mapsto r(-t,-\theta)$. In other words, introducing the involution
	\begin{equation}\label{def:invol_rev}
		(\mathscr{S}r)(\theta)\triangleq r(-\theta),
	\end{equation}
    then we have the invariance
    $$H\circ\mathscr{S}=H.$$
	Consequently, in what follows, we reduce our research of quasi-periodic solutions to the one satisfying the following reversibility condition
	$$r(-t,-\theta)=r(t,\theta).$$
\end{enumerate}

\end{lem}

\begin{proof}
	$(i)$ A natural parametrization of $\partial D_t$ is given by
	$$z_{\varepsilon}(t,\theta)=R_{\varepsilon}(b,t,\theta)e^{\ii\theta}.$$
	Direct computations give
	$$\mbox{Im}\left(\partial_{t}z_{\varepsilon}(t,\theta)\overline{\partial_{\theta}z_{\varepsilon}(t,\theta)}\right)=-\varepsilon\partial_{t}r(t,\theta).$$
	Inserting this into \eqref{CDE} implies immediately \eqref{EDP eq r}.\\
$(ii)$ Using the splitting \eqref{split Psi}-\eqref{PsiD Psip} together with the polar coordinates, we can write
\begin{align*}
	\boldsymbol{\Psi}\big(t,R_{\varepsilon}(b,t,\theta)e^{\ii\theta}\big)&= \boldsymbol{\Psi}_{\mathbb{D}}\big(t,R_{\varepsilon}(b,t,\theta)e^{\ii\theta}\big)+\varepsilon^2\boldsymbol{\Psi}_p^{\varepsilon}\big(t,R_{\varepsilon}(b,t,\theta)e^{\ii\theta}\big)\\
	&=\int_{\mathbb{T}}\left(\int_{0}^{R_{\varepsilon}(b,t,\eta)}G_{\mathbb{D}}\big(R_{\varepsilon}(b,t,\theta)e^{\ii\theta},\ell_2 e^{\ii\eta}\big)\ell_2 \,d\ell_2\right) d\eta\\
	&\quad+\varepsilon^2\int_{\mathbb{T}}\left(\int_{0}^{R_{\varepsilon}(b,t,\eta)}G_{p}^{\varepsilon}\big(R_{\varepsilon}(b,t,\theta)e^{\ii\theta},\ell_2 e^{\ii\eta}\big)\ell_2 \,d\ell_2\right) d\eta.
\end{align*}
Besides, still using the polar coordinates, we have from \eqref{ED Ep}
\begin{equation}\label{polar:EdEp}
    \begin{aligned}
	E_{\mathbb{D}}(r)(t)&=\int_{\mathbb{T}}\int_{\mathbb{T}}\int_{0}^{R_{\varepsilon}(b,t,\theta)}\left(\int_0^{R_{\varepsilon}(b,t,\eta)}G_{\mathbb{D}}\big(\ell_1e^{\ii\theta},\ell_2 e^{\ii\eta}\big)\ell_2 \,d\ell_2\right) \ell_1 d\ell_1 d\eta d\theta,\\
	E_{p}(r)(t)&=\int_{\mathbb{T}}\int_{\mathbb{T}}\int_{0}^{R_{\varepsilon}(b,t,\theta)}\left(\int_0^{R_{\varepsilon}(b,t,\eta)}G_p^{\varepsilon}\big(\ell_1e^{\ii\theta},\ell_2 e^{\ii\eta}\big)\ell_2 \,d\ell_2\right) \ell_1 d\ell_1 d\eta d\theta.
\end{aligned}
\end{equation}
Differentiating in $r$ and using the symmetries $G_{\mathbb{D}}(x,y)=G_{\mathbb{D}}(y,x)$ and $G_{p}^{\varepsilon}(x,y)=G_{p}^{\varepsilon}(y,x)$ yields
\begin{align*}
	d_rE_{\mathbb{D}}(r)[h](t)&=2\varepsilon\int_{\mathbb{T}}h(t,\theta)\left(\int_{\mathbb{T}}\left(\int_{0}^{R_{\varepsilon}(b,t,\eta)}G_{\mathbb{D}}\big(R_{\varepsilon}(b,t,\theta)e^{\ii\theta},\ell_2 e^{\ii\eta}\big)\ell_2 \,d\ell_2\right)  d\eta \right)d\theta\\
	d_rE_p^{\varepsilon}(r)[h](t)&=2\varepsilon\int_{\mathbb{T}}h(t,\theta)\left(\int_{\mathbb{T}}\left(\int_{0}^{R_{\varepsilon}(b,t,\eta)}G_p^{\varepsilon}\big(R_{\varepsilon}(b,t,\theta)e^{\ii\theta},\ell_2 e^{\ii\eta}\big)\ell_2 \,d\ell_2\right)  d\eta \right)d\theta.
\end{align*}
Hence,
\begin{align*}
	d_rE_{\mathbb{D}}(r)[\rho]&=\big\langle \nabla E_{\mathbb{D}}(r),\rho\big\rangle_{L^2(\mathbb{T})},\quad  \nabla E_{\mathbb{D}}(r)(t,\theta)=2\varepsilon\boldsymbol{\Psi}_{\mathbb{D}}\big(t,R_{\varepsilon}(b,t,\theta)e^{\ii\theta}\big),\\
	d_rE_p^{\varepsilon}(r)[\rho]&=\big\langle \nabla E_p^{\varepsilon}(r),\rho\big\rangle_{L^2(\mathbb{T})},\quad  \nabla E_p^{\varepsilon}(r)(t,\theta)=2\varepsilon\boldsymbol{\Psi}_p^{\varepsilon}\big(t,R_{\varepsilon}(b,t,\theta)e^{\ii\theta}\big).
\end{align*}
$(iii)$ Follows from the change of variables $\eta\mapsto-\eta,$ \eqref{polar:EdEp}, \eqref{HAM EDP} and \eqref{sym G}. This achieves the proof of Lemma \ref{lem eq EDP r}.
\end{proof}

\begin{remark}
	The structure of the equation \eqref{HAM EDP} ($\partial_{\theta}$ Poisson tensor) implies the conservation of the space average during the dynamics. Therefore, the natural phase space associated with this Hamiltonian equation is the space of square integrable functions with zero average denoted $L_0^2(\mathbb{T})$ and defined by
	\begin{equation}\label{def:phase-space}
	    L_0^{2}(\mathbb{T})\triangleq \left\lbrace r=\sum_{j\in\mathbb{Z}^*}r_{j}\mathbf{e}_j\quad\textnormal{s.t.}\quad r_{-j}=\overline{r_j},\qquad\sum_{j\in\mathbb{Z}^*}|r_j|^2<+\infty\right\rbrace,\qquad \mathbf{e}_{j}(\theta)\triangleq e^{\ii j\theta}.
	\end{equation}
\end{remark}
\begin{remark}\label{rem:mfold}
    Assume that the domain $\mathbf{D}$ is $\mathbf{m}$-fold symmetric for a given $\mathbf{m}\geqslant2$, namely invariant under $\tfrac{2\pi}{\mathbf{m}}$-rotations. At the level of conformal mapping, this means
    $$f_{\varepsilon}(x)=\sum_{n=0}^{\infty}\underbrace{a_{n}(\varepsilon)}_{\in\mathbb{R}}x^{\mathbf{m}n+1},\qquad\textnormal{i.e.}\qquad\boldsymbol{\Phi}\left(e^{\frac{2\ii\pi}{\mathbf{m}}}x\right)=e^{\frac{2\ii\pi}{\mathbf{m}}}\boldsymbol{\Phi}(x).$$
    In view of \eqref{def:Green}, this implies
    \begin{equation}\label{mfold:Green}
        G_{\mathbf{D}}\left(e^{\frac{2\ii\pi}{\mathbf{m}}}x,e^{\frac{2\ii\pi}{\mathbf{m}}}y\right)=G_{\mathbf{D}}(x,y).
    \end{equation}
    Therefore, performing the change of variable $\eta\mapsto\eta+\tfrac{2\pi}{\mathbf{m}}$ inside \eqref{polar:EdEp} and using \eqref{HAM EDP} and \eqref{mfold:Green}, we find the following invariance of the Hamiltonian
$$H\circ\mathcal{T}_{\mathbf{m}}=\mathcal{T}_{\mathbf{m}}\circ H,\qquad\mathcal{T}_{\mathbf{m}}(r)(\theta)\triangleq r\left(\theta+\tfrac{2\pi}{\mathbf{m}}\right).$$
Consequently, under this symmetry assumption on the domain, it is possible to run the forthcoming analysis (Nash-Moser construction) within the $\mathbf{m}$-fold Sobolev scale. It just adds extra computations, for instance checking that the various operators $A$ appearing in the reducibility procedure preserve the class of $\mathbf{m}$-fold symmetric functions, namely
$$r\left(t,\theta+\tfrac{2\pi}{\mathbf{m}}\right)=r(t,\theta)\qquad\Rightarrow\qquad(Ar)\left(t,\theta+\tfrac{2\pi}{\mathbf{m}}\right)=(Ar)(t,\theta).$$
In the present manuscript, this is just additional information and does not solve any substantial problem. The interested reader is invited to consult the work \cite{HHR23} where this analysis is performed. There, the $\mathbf{m}$-fold symmetry was essential to be tracked since low frequencies for annular patches correspond to (linear) instability.
\end{remark}
\subsubsection{Linearized operator and equilibrium quasi-periodic solutions}
Here, we study the linearized equation of \eqref{HAM EDP} at a given small state $\varepsilon r.$ The associated linearized operator can be decomposed into the sum of a variable coefficients transport operator, a nonlocal operator of order $0$ and a smoothing nonlocal operator in the variable $\theta.$ The equilibrium case $\varepsilon=0$ corresponds to linearized Euler equation in the unit disc at the Rankine vortices and was studied in \cite{HR21-1}. In this case, the linear operator is a Fourier multiplier with polynomial linear frequencies with respect to the radius $b$ of the Rankine patch $b\cdot\mathbb{D}$.
\begin{lem}\label{lem:lin_origin}
	The linearized Hamiltonian equation of \eqref{HAM EDP} at a state $r$ is the time-dependant Hamiltonian system 
	$$\partial_{t}h+\partial_{\theta}\Big(V_{\varepsilon r}h+\mathbf{L}_{\varepsilon r}[h]+\mathbf{S}_{\varepsilon r}[h]\Big)=0,$$
	where 
    \begin{enumerate}[label=\textbullet]
        \item $V_{\varepsilon r}$ is a function defined by
	\begin{equation}\label{Vr_EDP}
	    \begin{aligned}
		V_{\varepsilon r}(b,t,\theta)&\triangleq\tfrac{\varepsilon^2}{R_{\varepsilon}(b,t,\theta)}\int_{\mathbb{T}}\left(\int_{0}^{R_{\varepsilon}(b,t,\eta)}(\nabla_xG_{p}^{\varepsilon})\big(R_{\varepsilon}(b,t,\theta)e^{\ii\theta},\ell_2 e^{\ii\eta}\big)e^{\ii\theta}\ell_2 \,d\ell_2\right) d\eta\\
		&\quad-\tfrac{1}{2}\int_{\mathbb{T}}\tfrac{R_{\varepsilon}^{2}(b,t,\eta)}{R_{\varepsilon}^{2}(b,t,\theta)}d\eta
		-\tfrac{1}{R_{\varepsilon}(b,t,\theta)}\int_{\mathbb{T}}\log\big(A_{\varepsilon r}(b,t,\theta,\eta)\big)\partial_{\eta}\big(R_{\varepsilon}(b,t,\eta)\sin(\eta-\theta)\big)d\eta\\
		&\quad-\tfrac{1}{R_{\varepsilon}^{3}(b,t,\theta)}\int_{\mathbb{T}}\log\big(B_{\varepsilon r}(b,t,\theta,\eta)\big)\partial_{\eta}\big(R_{\varepsilon}(b,t,\eta)\sin(\eta-\theta)\big)d\eta.
	\end{aligned} 
	\end{equation}
    Moreover, if $r(-t,-\theta)=r(t,\theta)$, then
	\begin{equation}\label{symVr}
		V_{\varepsilon r}(b,-t,-\theta)=V_{\varepsilon r}(b,t,\theta).
	\end{equation}
	\item $\mathbf{L}_{\varepsilon r}$ is a nonlocal operator in the form
	\begin{equation}\label{mathbfLr}
		\mathbf{L}_{\varepsilon r}[h](b,t,\theta)\triangleq\int_{\mathbb{T}}h(t,\eta)\log\big(A_{\varepsilon r}(b,t,\theta,\eta)\big)d\eta,
	\end{equation}
    where
    \begin{equation}\label{Ar}
        A_{\varepsilon r}(b,t,\theta,\eta)\triangleq \left|R_{\varepsilon}(b,t,\theta)e^{\ii\theta}-R_{\varepsilon}(b,t,\eta)e^{\ii\eta}\right|,
    \end{equation}
    with $R_{\varepsilon}$ as in \eqref{def Re}.
	\item $\mathbf{S}_{\varepsilon r}$ is a smoothing nonlocal operator in the form
	\begin{equation}\label{bfSr_EDP}
		\mathbf{S}_{\varepsilon r}[h](b,t,\theta)\triangleq\int_{\mathbb{T}}h(t,\eta)\Big[\varepsilon^2 G_{p}^{\varepsilon}\big(R_{\varepsilon}(b,t,\theta)e^{\ii\theta},R_{\varepsilon}(b,t,\eta)e^{\ii\eta}\big)-\log\big(B_{\varepsilon r}(b,t,\theta,\eta)\big)\Big]d\eta,
	\end{equation}
    where $G_{p}^{\varepsilon}$ is defined in \eqref{split G} and
    \begin{equation}\label{Br}
        B_{\varepsilon r}(b,t,\theta,\eta)\triangleq \left|1-R_{\varepsilon}(b,t,\theta)R_{\varepsilon}(b,t,\eta)e^{\ii(\eta-\theta)}\right|.
    \end{equation}
    \end{enumerate}
\end{lem}
\begin{proof}
	According to \eqref{split G}, we can write
	\begin{align}\label{Psi to PsiD}
		\boldsymbol{\Psi}\big(t,R_{\varepsilon}(b,t,\theta)e^{\ii\theta}\big)&=\int_{\mathbb{T}}\left(\int_{0}^{R_{\varepsilon}(b,t,\eta)}G_{\mathbb{D}}\big(R_{\varepsilon}(b,t,\theta)e^{\ii\theta},\ell_2 e^{\ii\eta}\big)\ell_2 \,d\ell_2\right) d\eta\nonumber\\
		&\quad+\varepsilon^2 \int_{\mathbb{T}}\left(\int_{0}^{R_{\varepsilon}(b,t,\eta)}G_{p}^{\varepsilon}\big(R_{\varepsilon}(b,t,\theta)e^{\ii\theta},\ell_2 e^{\ii\eta}\big)\ell_2 \,d\ell_2\right) d\eta\nonumber\\
		&=\boldsymbol{\Psi}_{\mathbb{D}}\big(t,R_{\varepsilon}(b,t,\theta)e^{\ii\theta}\big)+\varepsilon^2 \int_{\mathbb{T}}\left(\int_{0}^{R_{\varepsilon}(b,t,\eta)}G_{p}^{\varepsilon}\big(R_{\varepsilon}(b,t,\theta)e^{\ii\theta},\ell_2 e^{\ii\eta}\big)\ell_2 \,d\ell_2\right) d\eta.
	\end{align}
	Hence linearizing $\boldsymbol{\Psi}$ yields
	\begin{align*}
		\partial_r\boldsymbol{\Psi}\big(t,R_{\varepsilon}(b,t,\theta)e^{\ii\theta}\big)h=&\partial_r\boldsymbol{\Psi}_{\mathbb{D}}\big(t,R_{\varepsilon}(b,t,\theta)e^{\ii\theta}\big)h+\varepsilon^3\int_{\mathbb{T}}h(t,\eta)G_{p}^{\varepsilon}\big(R_{\varepsilon}(b,t,\theta)e^{\ii\theta},R_{\varepsilon}(b,t,\eta) e^{\ii\eta}\big)d\eta\\
		+&\varepsilon^3\frac{h(t,\theta)}{R_{\varepsilon}(b,t,\theta)}\int_{\mathbb{T}}\left(\int_{0}^{R_{\varepsilon}(b,t,\eta)}(\nabla_xG_{p}^{\varepsilon})\big(R_{\varepsilon}(b,t,\theta)e^{\ii\theta},\ell_2 e^{\ii\eta}\big)e^{\ii\theta}\ell_2 \,d\ell_2\right)d\eta.
	\end{align*}
	By setting
	\begin{align}
		K_{\varepsilon r}^p(b,t,\theta,\eta)&\triangleq G_{p}^{\varepsilon}\big(R_{\varepsilon}(b,t,\theta)e^{\ii\theta},R_{\varepsilon}(b,t,\eta) e^{\ii\eta}\big),\label{Kerp}\\
		V_{\varepsilon r}^p(b,t,\theta)&\triangleq \frac{1}{R_{\varepsilon}(b,t,\theta)}\int_{\mathbb{T}}\left(\int_{0}^{R_{\varepsilon}(b,t,\eta)}(\nabla_xG_{p}^{\varepsilon})\big(R_{\varepsilon}(b,t,\theta)e^{\ii\theta},\ell_2 e^{\ii\eta}\big)e^{\ii\theta}\ell_2 \,d\ell_2\right) d\eta,\label{Verp}
	\end{align}
	we may write
	\begin{align*}
		\partial_r\boldsymbol{\Psi}\big(t,R_{\varepsilon}(b,t,\theta)e^{\ii\theta}\big)h=&\partial_r\Psi_{\mathbb{D}}\big(R_{\varepsilon}(b,t,\theta)e^{\ii\theta}\big)h+\varepsilon^3\int_{\mathbb{T}}h(t,\eta)K_{\varepsilon r}^p(b,t,\theta,\eta)d\eta+\varepsilon^3 V_{\varepsilon r}^p(b,t,\theta)h(t,\theta)\\
		\triangleq &\partial_r\boldsymbol{\Psi}_{\mathbb{D}}\big(t,R_{\varepsilon}(b,t,\theta)e^{\ii\theta}\big)h(t,\theta)+\varepsilon^3\mathbf{S}_{\varepsilon r}^p[h](b,t,\theta)+\varepsilon^3 V_{\varepsilon r}^p(b,t,\theta)h(t,\theta).
	\end{align*}
	Notice that $K_{\varepsilon r}^p$ and $V_{\varepsilon r}^p$ are smooth. From the computations done in \cite[Lem. 3.1]{HR21-1} (being careful to the change of sign in the definition of $\mathbf{S}_{\varepsilon r}^{\mathbb{D}}$ and to the rescaling in $\varepsilon$), we get 
	\begin{align*}
		\partial_r\boldsymbol{\Psi}_{\mathbb{D}}\big(t,R_{\varepsilon}(b,t,\theta)e^{\ii\theta}\big)h=&\varepsilon\Big(V_{\varepsilon r}^{\mathbb{D}}h+\mathbf{L}_{\varepsilon r}^{\mathbb{D}}h+\mathbf{S}_{\varepsilon r}^{\mathbb{D}}h\Big),
	\end{align*}
	where
	\begin{enumerate}[label=\textbullet]
	    \item $V_{\varepsilon r}^{\mathbb{D}}$ is a function defined by
        \begin{equation}\label{VerD}
        \begin{aligned}
		V_{\varepsilon r}^{\mathbb{D}}(b,t,\theta)&\triangleq-\tfrac{1}{2}\int_{\mathbb{T}}\tfrac{R_{\varepsilon}^{2}(b,t,\eta)}{R_{\varepsilon}^{2}(b,t,\theta)}d\eta-\tfrac{1}{R_{\varepsilon}(b,t,\theta)}\int_{\mathbb{T}}\log\big(A_{\varepsilon r}(b,t,\theta,\eta)\big)\partial_{\eta}\big(R_{\varepsilon}(b,t,\eta)\sin(\eta-\theta)\big)d\eta\\
		&\quad-\tfrac{1}{R_{\varepsilon}^{3}(b,t,\theta)}\int_{\mathbb{T}}\log\big(B_{\varepsilon r}(b,t,\theta,\eta)\big)\partial_{\eta}\big(R_{\varepsilon}(b,t,\eta)\sin(\eta-\theta)\big)d\eta,
	\end{aligned}
    \end{equation}
	\item $\mathbf{L}_{\varepsilon r}^{\mathbb{D}}$ is a nonlocal operator in the form
	$$\mathbf{L}_{\varepsilon r}^{\mathbb{D}}[h](b,t,\theta)\triangleq\int_{\mathbb{T}}h(t,\eta)\log\big(A_{\varepsilon r}(b,t,\theta,\eta)\big)d\eta.$$
	\item $\mathbf{S}_{\varepsilon r}^{\mathbb{D}}$ is a smoothing nonlocal operator in the form
	$$\mathbf{S}_{\varepsilon r}^{\mathbb{D}}[h](b,t,\theta)\triangleq-\int_{\mathbb{T}}h(t,\eta)\log\big(B_{\varepsilon r}(b,t,\theta,\eta)\big)d\eta.$$
	\end{enumerate}
	Finally we get the structure
	
	\begin{align*}
		\partial_r\boldsymbol{\Psi}\big(t,R_{\varepsilon}(b,t,\theta)e^{\ii\theta}\big)h=
		&\varepsilon\Big(\big(V_{\varepsilon r}^{\mathbb{D}}+\varepsilon^2 V_{\varepsilon r}^p\big)h+\mathbf{L}_{\varepsilon r}^{\mathbb{D}}h+\big(\mathbf{S}_{\varepsilon r}^{\mathbb{D}}+\varepsilon^2\mathbf{S}_{\varepsilon r}^p\big)h\Big)
	\end{align*}
	and, after simplification by $\varepsilon$, the linearized equation at any state $r$ takes the form
	\begin{equation}\label{lin eq EDP}
		\partial_t h+\partial_\theta\Big[V_{\varepsilon r}h+\mathbf{L}_{\varepsilon r}h+\mathbf{S}_{\varepsilon r}h\Big]=0,
	\end{equation}
	with
	\begin{equation}\label{dec VLS}
		V_{\varepsilon r}\triangleq V_{\varepsilon r}^{\mathbb{D}}+\varepsilon^2 V_{\varepsilon r}^p,\qquad \mathbf{L}_{\varepsilon r}\triangleq \mathbf{L}_{\varepsilon r}^{\mathbb{D}},\qquad\mathbf{S}_{\varepsilon r}\triangleq \mathbf{S}_{\varepsilon r}^{\mathbb{D}}+\varepsilon^2\mathbf{S}_{\varepsilon r}^p.
	\end{equation}
The symmetry property \eqref{symVr} is easily obtained from \eqref{def Re}, \eqref{VerD}, \eqref{Verp} and \eqref{dec VLS}.
\end{proof}
Our next step is to provide the expression of the linearized equation at the equilibrium state $\varepsilon=0$. By virtue of \eqref{dec VLS} this corresponds to the linearized equations at the Rankine vortex $b\cdot\mathbb{D}$ within the unperturbed domain $\mathbb{D}.$ Therefore, the analysis has already been performed in \cite[Lem. 3.2 and Prop. 3.1]{HR21-1} and we recall here the results.
\begin{lem}\label{lem:equilibre_HR21}
	The following results hold true.
	\begin{enumerate}[label=(\roman*)]
		\item The linearized equation of \eqref{HAM EDP} at the equilibrium state $(\varepsilon=0)$ writes 
		\begin{equation}\label{Ham eq 0}
			\partial_{t}h=\partial_{\theta}\mathrm{L}(b)[h]=\partial_{\theta}\nabla H_{\mathrm{L}}(h),
		\end{equation}
		where $\mathrm{L}(b)$ is the self-adjoint operator defined by 
		\begin{equation*}
			\mathrm{L}(b)[h]\triangleq -\frac{h}{2}-(\mathcal{K}_1-\mathcal{K}_{b})\ast h=-\sum_{j\in\mathbb{Z}^*}\tfrac{\Omega_{j}(b)}{j}h_{j}\mathbf{e}_{j},
		\end{equation*} 
	with
		\begin{equation}\label{def:calKb}
			\mathcal{K}_{\mathtt{p}}\triangleq \log\left(\left|1-\mathtt{p}^2e^{\ii\theta}\right|\right)
		\end{equation}
	and
	\begin{equation}\label{equilibrium_spectrum}
	    \Omega_{j}(b)\triangleq\frac{j}{2|j|}\Big(|j|-1+b^{2|j|}\Big).
	\end{equation}
		The Hamiltonian $H_{\mathrm{L}}$ is quadratic and expresses as follows
		\begin{equation}\label{defLHL_EDP}
			H_{\mathrm{L}}(h)\triangleq \frac{1}{2}\big\langle\mathrm{L}(b)h,h\big\rangle_{L^{2}(\mathbb{T})}=-\sum_{j\in\mathbb{Z}^{*}}\tfrac{\Omega_{j}(b)}{2j}|h_{j}|^{2}.
		\end{equation}
		\item Let $0<b_0<b_1<1$ and $\mathbb{S}\subset\mathbb{N}^*$ finite. Then, there exists a Cantor-like set 
		$$\mathscr{C}_{\textnormal{Eq}}\subset(b_0,b_1),\qquad|\mathscr{C}_{\textnormal{Eq}}|=b_1-b_0$$
		such that for all $b\in\mathscr{C}_{\textnormal{Eq}}$, every function in the form
		\begin{equation}\label{rev sol eq:EDP}
			h(t,\theta)=\sum_{j\in\mathbb{S}}h_{j}\cos\big(j\theta-\Omega_{j}(b)t\big),\quad h_{j}\in\mathbb{R}^*
		\end{equation} 
		is a time quasi-periodic reversible solution to \eqref{Ham eq 0} with vector frequency $-\omega_{\textnormal{Eq}}(b),$ where
		\begin{equation}\label{def:equil_freq_HHHM}
		    \omega_{\textnormal{Eq}}(b)\triangleq\big(\Omega_{j}(b)\big)_{j\in\mathbb{S}}.
		\end{equation}
	\end{enumerate}
\end{lem}

\subsection{Functional of interest and regularity aspects}
We begin by recasting \eqref{HAM EDP} within the Hamiltonian framework underlying the subsequent KAM analysis. The linearization around the equilibrium provides the reference integrable dynamics, while the nonlinear terms are treated as a perturbation of this system. We then distinguish a finite collection of tangential modes from the remaining normal directions and introduce action-angle coordinates for the former. In these variables, the construction of quasi-periodic solutions amounts to finding an invariant torus with prescribed tangential frequencies. This is expressed as the vanishing of a nonlinear functional $\mathscr{F}$, for which we shall implement a Nash--Moser iteration. We conclude the section by establishing the tame regularity properties of the Hamiltonian vector field entering the functional equation, which will be essential for the estimates required throughout the Nash--Moser procedure.
\\

Observe that, according to Lemma \ref{lem:equilibre_HR21}, the equation \eqref{HAM EDP} writes
\begin{equation}\label{perturb_EDP and XP}\partial_{t}r=\partial_{\theta}\mathrm{L}(b)[r]+\varepsilon X_{P_{\varepsilon}}(r),\qquad X_{P_{\varepsilon}}(r)\triangleq X_{P_{\varepsilon}^{\mathbb{D}}}(r)+X_{P_{\varepsilon}^{p}}(r),
\end{equation}
with
\begin{equation}\label{XPD XPp}
	X_{P_{\varepsilon}^{\mathbb{D}}}(r)\triangleq \tfrac{1}{\varepsilon}\Big(\tfrac{1}{2}\partial_{\theta}r+\partial_{\theta}\mathcal{K}_{b}\ast r-\tfrac{1}{2\varepsilon^2}\partial_{\theta}\nabla E_{\mathbb{D}}(r)\Big),\qquad X_{P_{\varepsilon}^{p}}(r)\triangleq -\tfrac{1}{2\varepsilon}\partial_{\theta}\nabla E_p^{\varepsilon}(r),
\end{equation}
where the components $E_{\mathbb{D}}$ and $E_{p}^{\varepsilon}$ of the kinetic energy are introduced in  \eqref{ED Ep} and the convolution kernel $\mathcal{K}_b$ is given by \eqref{def:calKb}. Observe that \eqref{perturb_EDP and XP} is the Hamiltonian system generated by Hamiltonian  
\begin{align}\label{Hamil:perturbe}
	\mathcal{H}_{\varepsilon}(r)&\triangleq H_{\mathrm{L}}(r)+\varepsilon P_{\varepsilon}(r),\qquad P_{\varepsilon}(r)\triangleq P_{\varepsilon}^{\mathbb{D}}(r)+P_{\varepsilon}^{p}(r),
\end{align}
with $H_{\mathrm{L}}$ the quadratic Hamiltonian defined in \eqref{defLHL_EDP}, 
$\varepsilon P_{\varepsilon}^{\mathbb{D}}(r)$ containing terms
of higher order more than cubic and $\varepsilon P_{\varepsilon}^p(r)$ affine but containing smallness.
\subsubsection{Reformulation with  the action-angle and normal variables}\label{subsec act-angl}
We provide here a suitable reformulation of the problem in terms of action-angle-mormal variables well adapted to the implementation of KAM techniques.

\medskip

$\blacktriangleright$ \textbf{Phase-space decomposition :}
We first isolate the finite-dimensional set of modes around which the quasi-periodic dynamics will be constructed. Let
$$\mathbb S\subset\mathbb N^*,
\qquad
d\triangleq|\mathbb{S}|<\infty,$$
be a fixed finite collection of tangential sites. Since the phase space consists of real-valued functions, it is convenient to work with the symmetrized set
$$\overline{\mathbb S}\triangleq\mathbb S\cup(-\mathbb S),
\qquad
\mathbb S_0\triangleq\overline{\mathbb S}\cup\{0\}.$$
We then split the phase space $L_0^2(\mathbb{T}),$ introduced in \eqref{def:phase-space}, into the modes selected above and their complement
\begin{equation*}
L^2_0(\mathbb T)=L_{\overline{\mathbb S}}
\overset{\bot}{\oplus}
L^2_\bot,
\qquad
L_{\overline{\mathbb S}}
\triangleq
\operatorname{span}\{\mathbf{e}_j,\,j\in\overline{\mathbb S}\},\qquad L^2_\bot
\triangleq
\overline{\operatorname{span}\{\mathbf{e}_j,\,j\in\mathbb{Z}\setminus\mathbb{S}_0\}}.
\end{equation*}
Accordingly, every $r\in L^2_0(\mathbb T)$ admits a unique decomposition
\begin{equation*}
r=v+z,
\qquad
v=\Pi_{\overline{\mathbb S}}r
\triangleq\sum_{j\in\overline{\mathbb S}}r_j\mathbf{e}_j,
\qquad
z=\Pi^\bot_{\mathbb S_0}r\triangleq
\sum_{j\in\mathbb Z\setminus\mathbb S_0}r_j\mathbf{e}_j.
\end{equation*}
The variables $v$ and $z$ will respectively describe the \textit{tangential} and \textit{normal directions}.\\

$\blacktriangleright$ \textbf{Introduction of action-angle variables :} The tangential component is then parametrized in canonical coordinates. Fix small positive amplitudes
$$
(\mathtt a_j)_{j\in\overline{\mathbb S}}
\in(\mathbb R_+^*)^d,
\qquad
\mathtt a_{-j}=\mathtt a_j,
$$
and introduce, for $j\in\overline{\mathbb S}$, the polar transformation
\begin{equation*}
r_j=\sqrt{\mathtt a_j^2+|j|I_j}e^{\ii\vartheta_j}.
\end{equation*}
The reality condition is encoded by
\begin{equation*}
I_{-j}=I_j,
\qquad
\vartheta_{-j}=-\vartheta_j,
\qquad
I_j\in\mathbb R,\quad \vartheta_j\in\mathbb T.
\end{equation*}
Hence the complete change of variables can be written as
\begin{equation}\label{def:bfA}
r=\mathbf{A}(\vartheta,I,z)
\triangleq\sum_{j\in\overline{\mathbb S}}
\sqrt{\mathtt a_j^2+|j|I_j}\,
e^{\ii\vartheta_j}\mathbf{e}_j+z.
\end{equation}
The symplectic form associated with \eqref{HAM EDP} on $L^2_0(\mathbb T)$ is
$$\mathcal W(r_1,r_2)
\triangleq
\int_{\mathbb T}\partial_\theta^{-1}r_1(\theta)r_2(\theta)\,d\theta,
\qquad
\partial_\theta^{-1}r
\triangleq
\sum_{j\in\mathbb Z^*}\frac{r_j}{\ii j}\mathbf{e}_j$$
and that the corresponding Hamiltonian vector field is
$$X_H(r)=\partial_\theta\nabla H(r).$$ This parametrization is chosen so that the symplectic structure separates into canonical action-angle variables and the symplectic structure of the normal directions. More precisely,
\begin{equation*}
\mathcal W=\sum_{j\in\mathbb S}d\vartheta_j\wedge dI_j
+
\sum_{j\in\mathbb Z\setminus\mathbb S_0}
\frac{1}{2\ii j},dr_j\wedge dr_{-j}=
\left(\sum_{j\in\mathbb S}d\vartheta_j\wedge dI_j\right)
\oplus
\mathcal W_{|L^2_\bot}.
\end{equation*}
In particular, $(\vartheta,I)$ form canonical coordinates on the tangential directions, while $z$ contains all the remaining modes. Moreover, the linear dynamics around the equilibrium is compatible with the above splitting. Since $\mathrm L(b)$ preserves both $L_{\overline{\mathbb S}}$ and $L^2_\bot$, the quadratic Hamiltonian becomes, up to an irrelevant constant,
\begin{equation*}
H_{\mathrm L}\circ\mathbf{A}=-\sum_{j\in\mathbb S}\Omega_j(b)I_j
+
\frac12\langle \mathrm L(b)z,z\rangle_{L^2(\mathbb T)}=-\omega_{\mathrm{Eq}}(b)\cdot I
+
\frac12\langle \mathrm L(b)z,z\rangle_{L^2(\mathbb T)},
\end{equation*}
where $\omega_{\mathrm{Eq}}(b)$ is the vector of tangential frequencies defined in \eqref{def:equil_freq_HHHM}. Notice in particular that the trajectory
$v(-\omega_{\mathrm{Eq}}(b)t,0)$ reproduces the corresponding solution of the linearized equation \eqref{Ham eq 0}, cf. \eqref{rev sol eq:EDP}. Consequently, after composing the original Hamiltonian with the transformation $\mathbf{A}$, the perturbed problem takes the standard integrable-plus-perturbation form
\begin{equation}\label{derniere_def}
H_\varepsilon=\mathcal N+\varepsilon\mathcal P_\varepsilon,
\qquad
\mathcal N
\triangleq
-\omega_{\mathrm{Eq}}(b)\cdot I
+
\frac12\langle\mathrm L(b)z,z\rangle_{L^2(\mathbb T)},
\qquad
\mathcal P_\varepsilon
\triangleq
P_\varepsilon\circ\mathbf{A}.
\end{equation}
Thus, the linearized dynamics provides the integrable Hamiltonian around which the KAM construction will be performed.

\medskip
$\blacktriangleright$ \textbf{Invariant tori and the nonlinear functional :}
The objective is now to find an invariant $d$-dimensional torus carrying quasi-periodic motions with a prescribed frequency vector $\omega\in\mathbb R^d$. We therefore consider embeddings
\begin{equation*}
i:
\begin{array}[t]{rcl}
\mathbb T^d&\longrightarrow&
\textnormal{PS}\\
\varphi&\longmapsto&
i(\varphi)=\big(\vartheta(\varphi),I(\varphi),z(\varphi)\big),
\end{array}\qquad\textnormal{PS}\triangleq\mathbb T^d\times\mathbb R^d\times L^2_\bot.
\end{equation*}
Such an embedding parametrizes an invariant torus precisely when its pullback by the linear flow $\varphi\mapsto\varphi+\omega t$ agrees with the Hamiltonian dynamics. Equivalently,
\begin{equation*}
\omega\cdot\partial_\varphi i(\varphi)=X_{H_\varepsilon}\big(i(\varphi)\big),
\end{equation*}
where
\begin{equation}\label{def_ultime}
X_{H_\varepsilon}=\left(
\partial_IH_\varepsilon,
-\partial_\vartheta H_\varepsilon,
\Pi_{\mathbb S_0}^\bot
\partial_\theta\nabla_zH_\varepsilon
\right).
\end{equation}
At $\varepsilon=0$, this equation has an explicit family of invariant tori. Indeed, for every $b\in(b_0,b_1)$,
\begin{equation}\label{def iflat}
i_{\mathrm{flat}}(\varphi)=(\varphi,0,0)
\end{equation}
is invariant for the unperturbed Hamiltonian system when the tangential frequency is $\omega_{\mathrm{Eq}}(b)$. The perturbative construction amounts to continuing these trivial tori for $\varepsilon\neq0$ while allowing their frequency to vary. To implement the Nash--Moser scheme, we introduce an auxiliary frequency parameter $\alpha\in\mathbb R^d$ and consider
\begin{equation*}
H_\varepsilon^\alpha=\mathcal N_\alpha+\varepsilon\mathcal P_\varepsilon,
\qquad
\mathcal N_\alpha
\triangleq
\alpha\cdot I
+
\frac12\langle\mathrm L(b)z,z\rangle_{L^2(\mathbb T)}.
\end{equation*}
The original Hamiltonian is recovered for
$$\alpha=-\omega_{\mathrm{Eq}}(b).$$
The freedom to vary $\alpha$ is useful in solving the finite-dimensional compatibility conditions that arise when constructing an approximate inverse of the linearized operator. We can now formulate the entire problem as a single functional equation. Define
\begin{equation*}
\mathscr F(i,\alpha,(b,\omega),\varepsilon)
\triangleq
\omega\cdot\partial_\varphi i
-X_{H_\varepsilon^\alpha}(i).
\end{equation*}
In components,
$$\mathscr F\big((\vartheta,I,z),\alpha,(b,\omega),\varepsilon\big)=
\begin{bmatrix}
    \omega\cdot\partial_\varphi\vartheta
-\alpha
-\varepsilon\partial_I\mathcal P_\varepsilon(i)\\
\omega\cdot\partial_\varphi I
+\varepsilon\partial_\vartheta\mathcal P_\varepsilon(i)\\
\omega\cdot\partial_\varphi z
-\partial_\theta
\left[
\mathrm L(b)z
+
\varepsilon\nabla_z\mathcal P_\varepsilon(i)
\right]
\end{bmatrix}.$$
Thus, the construction of quasi-periodic solutions is reduced to finding zeros of $\mathscr F$. The involution $\mathscr{S}$ in \eqref{def:invol_rev} associated with reversibility also takes a particularly simple form in the action-angle-normal variables
\begin{equation}\label{reversiblility_op:action-angle}
\mathfrak S:(\vartheta,I,z)
\longmapsto
(-\vartheta,I,\mathscr Sz).
\end{equation} 
In these coordinates, the reversibility inherited from the original equation reads: for every $\alpha\in\mathbb R^d$,
$$
H_\varepsilon^\alpha\circ\mathfrak S
=
H_\varepsilon^\alpha.
$$
We consequently restrict the functional equation to the search of reversible embeddings, namely those satisfying
\begin{equation*}
\mathfrak{S}(i)(\varphi)=i(-\varphi),\qquad\textnormal{i.e.}\qquad\vartheta(-\varphi)=-\vartheta(\varphi),
\qquad
I(-\varphi)=I(\varphi),
\qquad
z(-\varphi)=(\mathscr Sz)(\varphi).
\end{equation*}
This symmetry is useful in the subsequent analysis since it is compatible with both the Hamiltonian structure and the linearized equations.

\medskip
$\blacktriangleright$ \textbf{Perturbation of the trivial torus.}
It remains to choose coordinates adapted to the perturbative nature of the problem. Since the unperturbed torus is given by \eqref{def iflat}, we measure a general embedding relative to this reference configuration rather than working directly with $i$ itself. We write
$$
i(\varphi)
=
(\varphi,0,0)
+
\mathfrak I(\varphi),
$$
where
\begin{equation*}
\mathfrak I(\varphi)
\triangleq
i(\varphi)-(\varphi,0,0)=\big(\Theta(\varphi),I(\varphi),z(\varphi)\big),
\qquad
\Theta(\varphi)
\triangleq
\vartheta(\varphi)-\varphi.
\end{equation*}
The unknown $\mathfrak I$ therefore represents the deformation of the flat invariant torus: $\Theta$ describes the correction to the angular component, while $I$ and $z$ measure the tangential and normal displacements, respectively.

\subsubsection{Regularity of the perturbed Hamiltonian vector field}
Here start the functional analysis of the problem. We shall obtain the required tame estimates on the Hamiltonian vector field governing the dynamics in order to run the Nash-Moser scheme. The Nash-Moser iteration creates dependence of the approximate solutions with respect to the parameters $(b,\omega).$ Therefore, we shall work within the weighted Sobolev scale following this parameter dependence. We refer the reader to the Appendix \ref{appendix:topo} for their presentation. We shall fix a parameter set
\begin{equation}\label{def:global_param_set}
		\mathscr{O}\triangleq (b_{0},b_{1})\times B_{\mathbb{R}^d}(0,R_0)\subset\mathbb{R}^{d+1}, \qquad R_0>0,
	\end{equation}
    where $b_0,b_1$ are the lower and upper Rankine radii chosen in \eqref{Rankine_range} and $B_{\mathbb{R}^d}(0,R_0)$ is an open ball of $\mathbb{R}^d$ containing the continuous curve of equilibrium frequency vector  $b\in[b_0,b_1]\mapsto\omega_{\textnormal{Eq}}(b);$ see \eqref{def:equil_freq_HHHM}. We consider an integer
    $$q\in\mathbb{N}^*$$
    tracking the regularity of our functions with respect to the parameters inside $\mathscr{O}.$
    We also introduce a parameter
    $$\gamma>0$$
    that will appear naturally in the non-resonance condition during the KAM reducibility to construct an approximate inverse. Finally, we fix a range of Sobolev regularity indices in the time-space variables $(\varphi,\theta)$, where the Nash-Moser scheme holds
	$$\tfrac{d+1}{2}+q+2\leqslant s_0\leqslant s\leqslant S.$$
\begin{lem}\label{lem:tame_XP_EDP}
	There exists $\varepsilon_{0}\in(0,1]$ such that if
		$$\varepsilon\leqslant\varepsilon_0\quad\textnormal{and}\quad\| r\|_{q,s_{0}+2}^{\gamma,\mathscr{O}}\leqslant1,$$
		then the vector field $X_{P_{\varepsilon}}$, defined in \eqref{perturb_EDP and XP} satisfies the following estimates
		\begin{enumerate}[label=(\roman*)]
			\item $\| X_{P_{\varepsilon}}(r)\|_{q,s}^{\gamma,\mathscr{O}}\lesssim 1+\| r\|_{q,s+1}^{\gamma,\mathscr{O}}.$
			\item $\| d_{r}X_{P_{\varepsilon}}(r)[h]\|_{q,s}^{\gamma,\mathscr{O}}\lesssim\|h\|_{q,s+2}^{\gamma,\mathscr{O}}+\| r\|_{q,s+2}^{\gamma,\mathscr{O}}\|h\|_{q,s_{0}+1}^{\gamma,\mathscr{O}}.$
			\item 
			$\| d_r^{2}X_{P_{\varepsilon}}(r)[h_{1},h_{2}]\|_{q,s}^{\gamma,\mathscr{O}}\lesssim\|h_{1}\|_{q,s_{0}+1}^{\gamma,\mathscr{O}}\|h_{2}\|_{q,s+2}^{\gamma,\mathscr{O}}+\big(\|h_{1}\|_{q,s+2}^{\gamma,\mathscr{O}}+\| r\|_{q,s+2}^{\gamma,\mathscr{O}}\|h_{1}\|_{q,s_{0}+1}^{\gamma,\mathscr{O}}\big)\|h_{2}\|_{q,s_{0}+1}^{\gamma,\mathscr{O}}$.
	\end{enumerate}
\end{lem}
\begin{proof} Recall from \eqref{perturb_EDP and XP} the decomposition
	$$X_{P_{\varepsilon}}(r)=X_{P_{\varepsilon}^{\mathbb{D}}}(r)+X_{P_{\varepsilon}^{p}}(r).$$
	$(i)$ It has been proved in \cite[Lem. 5.1-(i)]{HR21-1} that
	\begin{equation}\label{e-XPD}
		\|X_{P_{\varepsilon}^{\mathbb{D}}}(r)\|_{q,s}^{\gamma,\mathscr{O}}\lesssim\|r\|_{q,s+2}^{\gamma,\mathscr{O}}\|r\|_{q,s_0+1}^{\gamma,\mathscr{O}}.
	\end{equation}
	In particular, there is no singularity in $\varepsilon.$ On the other hand, we remind the identity
	$$\nabla E_p^{\varepsilon}(r)(\varphi,\theta)=2\varepsilon\boldsymbol{\Psi}_p^{\varepsilon}\big(\varphi,R_{\varepsilon}(b,\varphi,\theta)e^{\ii\theta}\big).$$
	Therefore,
	\begin{align*}
		X_{P_{\varepsilon}^p}(r)&\triangleq  \partial_{\theta}\left(\int_{\mathbb{T}}\left(\int_{0}^{R_{\varepsilon}(b,\varphi,\eta)}G_{p}^{\varepsilon}\big(R_{\varepsilon}(b,\varphi,\theta)e^{\ii\theta},\ell_2 e^{\ii\eta}\big)\ell_2 \,d\ell_2\right) d\eta\right)\\
		&=\Big[\ii R_{\varepsilon}(b,\varphi,\theta)+\frac{\varepsilon\partial_{\theta}r(\varphi,\theta)}{R_{\varepsilon}(b,\varphi,\theta)}\Big]e^{\ii\theta}\int_{\mathbb{T}}\left(\int_{0}^{R_{\varepsilon}(b,\varphi,\eta)}(\nabla_xG_{p}^{\varepsilon})\big(R_{\varepsilon}(b,\varphi,\theta)e^{\ii\theta},\ell_2 e^{\ii\eta}\big)\ell_2 \,d\ell_2\right) d\eta.
	\end{align*}
By smoothness of $G_p$, one immediately obtains after the change of variables $\eta\mapsto\eta+\theta$ and using Lemma \ref{lem:funct_topo}
\begin{equation}\label{e-XPp}
	\|X_{P_{\varepsilon}^p}(r)\|_{q,s}^{\gamma,\mathscr{O}}\lesssim 1+\|r\|_{q,s+1}^{\gamma,\mathscr{O}}.
\end{equation}
Combining \eqref{e-XPD} and \eqref{e-XPp}, one gets the desired result.\\
$(ii)$ It has been proved in \cite[Lem. 5.1-(ii)]{HR21-1} that
$$\| d_{r}X_{P_{\varepsilon}^{\mathbb{D}}}(r)[h]\|_{q,s}^{\gamma,\mathscr{O}}\lesssim\|h\|_{q,s+2}^{\gamma,\mathscr{O}}+\| r\|_{q,s+2}^{\gamma,\mathscr{O}}\|h\|_{q,s_{0}+1}^{\gamma,\mathscr{O}}.$$
On the other hand, one has from \eqref{lin eq EDP}, \eqref{Kerp} and \eqref{Verp}
\begin{equation}\label{dXPp}
	d_rX_{P_{\varepsilon}^p}(r)[h]=\varepsilon^{2}\partial_{\theta}\big(V_{\varepsilon r}^ph+\mathbf{S}_{\varepsilon r}^p[h]\big),
\end{equation}
with
$$V_{\varepsilon r}^p(b,\varphi,\theta)=\frac{1}{R_{\varepsilon}(b,\varphi,\theta)}\int_{\mathbb{T}}\left(\int_{0}^{R_{\varepsilon}(b,\varphi,\eta)}(\nabla_xG_{p}^{\varepsilon})\big(R_{\varepsilon}(b,\varphi,\theta)e^{\ii\theta},\ell_2 e^{\ii\eta}\big)e^{\ii\theta}\ell_2 \,d\ell_2\right) d\eta$$
and
$$\mathbf{S}_{\varepsilon r}^p(h)(b,\varphi,\theta)=\int_{\mathbb{T}}h(\varphi,\eta)K_{\varepsilon r}^p(b,\varphi,\theta,\eta)d\eta,\quad K_{\varepsilon r}^p(b,\varphi,\theta,\eta)=G_p^{\varepsilon}\big(R_{\varepsilon}(b,\varphi,\theta)e^{\ii\theta},R_{\varepsilon}(b,\varphi,\eta)e^{\ii\eta}\big).$$
By smoothness of $G_p^{\varepsilon}$, one immediately obtains after the change of variables $\eta\mapsto\eta+\theta$ and using Lemma \ref{lem:funct_topo}
\begin{equation}\label{e-Verp}
	\|V_{\varepsilon r}^{p}\|_{q,s}^{\gamma,\mathscr{O}}\lesssim 1+\|r\|_{q,s}^{\gamma,\mathscr{O}}.
\end{equation}
One also gets
\begin{equation}\label{e-Kerp}
	\|K_{\varepsilon r}^{p}\|_{q,H_{\varphi,\theta,\eta}^s}^{\gamma,\mathscr{O}}\lesssim 1+\|r\|_{q,s}^{\gamma,\mathscr{O}},
\end{equation}
which implies in turn by Sobolev embeddings
$$\|\partial_{\theta}\mathbf{S}_{\varepsilon r}^p(h)\|_{q,s}^{\gamma,\mathscr{O}}\lesssim\|h\|_{q,s}^{\gamma,\mathscr{O}}+\|r\|_{q,s+1}^{\gamma,\mathscr{O}}\|h\|_{q,s_0}^{\gamma,\mathscr{O}}.$$
Gathering the previous computations leads to
$$\|d_rX_{P_{\varepsilon}^p}(r)[h]\|_{q,s}^{\gamma,\mathscr{O}}\lesssim\|h\|_{q,s+1}^{\gamma,\mathscr{O}}+\|r\|_{q,s+1}^{\gamma,\mathscr{O}}\|h\|_{q,s_0+1}^{\gamma,\mathscr{O}}.$$
$(iii)$ It has been proved in \cite[Lem. 5.1-(iii)]{HR21-1} that
$$\| d_r^{2}X_{P_{\varepsilon}^{\mathbb{D}}}(r)[h_{1},h_{2}]\|_{q,s}^{\gamma,\mathscr{O}}\lesssim\|h_{1}\|_{q,s_{0}+1}^{\gamma,\mathscr{O}}\|h_{2}\|_{q,s+2}^{\gamma,\mathscr{O}}+\big(\|h_{1}\|_{q,s+2}^{\gamma,\mathscr{O}}+\| r\|_{q,s+2}^{\gamma,\mathscr{O}}\|h_{1}\|_{q,s_{0}+1}^{\gamma,\mathscr{O}}\big)\|h_{2}\|_{q,s_{0}+1}^{\gamma,\mathscr{O}}.$$
Differentiating \eqref{dXPp} with respect to $r$ in the direction $h_2$ yields
$$d_r^2X_{P_{\varepsilon}^p}(r)[h_1,h_2]=\varepsilon^{2}\partial_{\theta}\big(d_rV_{\varepsilon r}^p(r)[h_2]h_1+d_r\mathbf{S}_{\varepsilon r}^p(r)[h_2][h_1]\big).$$
We first study the term associated to $V_{\varepsilon r}^p$. Straightforward computations gives
\begin{align*}
	d_{r}V_{\varepsilon r}^{p}(r)[h_2](b,\varphi,\theta)&=\frac{-\varepsilon h_{2}(\varphi,\theta)}{R_{\varepsilon}^2(b,\varphi,\theta)}\int_{\mathbb{T}}\left(\int_{0}^{R_{\varepsilon}(b,\varphi,\eta)}(\nabla_xG_{p}^{\varepsilon})\big(R_{\varepsilon}(b,\varphi,\theta)e^{\ii\theta},\ell_2 e^{\ii\eta}\big)\cdot e^{\ii\theta}\ell_2 \,d\ell_2\right) d\eta\\
	&\quad+\frac{\varepsilon}{R_{\varepsilon}(b,\varphi,\theta)}\int_{\mathbb{T}}h_2(\varphi,\eta)(\nabla_xG_{p}^{\varepsilon})\big(R_{\varepsilon}(b,\varphi,\theta)e^{\ii\theta},R_{\varepsilon}(b,\varphi,\eta) e^{\ii\eta}\big)\cdot e^{\ii\theta}d\eta\\
	&\quad+\frac{\varepsilon h_2(\varphi,\theta)e^{\ii\theta}}{R_{\varepsilon}^2(b,\varphi,\theta)}\int_{\mathbb{T}}\left(\int_{0}^{R_{\varepsilon}(b,\varphi,\eta)}(\nabla_x^2G_{p}^{\varepsilon})\big(R_{\varepsilon}(b,\varphi,\theta)e^{\ii\theta},\ell_2 e^{\ii\eta}\big)e^{2\ii\theta}\ell_2 \,d\ell_2\right) d\eta.
\end{align*}
Using the same techniques as in the previous points, one gets
$$\|d_{r}V_{\varepsilon r}^{p}(r)[h_2]\|_{q,s}^{\gamma,\mathscr{O}}\lesssim\|h_2\|_{q,s}^{\gamma,\mathscr{O}}+\|r\|_{q,s}^{\gamma,\mathscr{O}}\|h_2\|_{q,s_0}^{\gamma,\mathscr{O}}.$$
Therefore,
\begin{align*}
	\|\partial_{\theta}\big(d_{r}V_{\varepsilon r}^{p}(r)[h_2]h_{1}\big)\|_{q,s}^{\gamma,\mathscr{O}}&\leqslant\|d_{r}V_{\varepsilon r}^{p}(r)[h_2]h_{1}\|_{q,s+1}^{\gamma,\mathscr{O}}\\
	&\lesssim\|d_{r}V_{\varepsilon r}^{p}(r)[h_2]\|_{q,s+1}^{\gamma,\mathscr{O}}\|h_1\|_{q,s_0}^{\gamma,\mathscr{O}}+\|d_{r}V_{\varepsilon r}^{p}(r)[h_2]\|_{q,s_0}^{\gamma,\mathscr{O}}\|h_1\|_{q,s+1}^{\gamma,\mathscr{O}}\\
	&\lesssim\|h_{1}\|_{q,s_0}^{\gamma,\mathscr{O}}\|h_2\|_{q,s+1}^{\gamma,\mathscr{O}}+\|r\|_{q,s+1}^{\gamma,\mathscr{O}}\|h_1\|_{q,s_0}^{\gamma,\mathscr{O}}\|h_2\|_{q,s}^{\gamma,\mathscr{O}}+\|h_1\|_{q,s+1}^{\gamma,\mathscr{O}}\|h_2\|_{q,s_0}^{\gamma,\mathscr{O}}.
\end{align*}
Besides, we can write
$$d_r\mathbf{S}_{\varepsilon r}^p(r)[h_2][h_1](b,\varphi,\theta)=\int_{\mathbb{T}}h_1(\varphi,\eta+\theta)d_rK_{\varepsilon r}^p(r)[h_2](b,\varphi,\theta,\eta+\theta)d\eta,$$
with
\begin{align*}
	d_rK_{\varepsilon r}^p(r)[h_2](b,\varphi,\theta,\eta)&=\frac{\varepsilon h_2(\varphi,\theta)e^{\ii\theta}}{R_{\varepsilon}(b,\varphi,\theta)}(\nabla_xG_p^{\varepsilon})\big(R_{\varepsilon}(b,\varphi,\theta)e^{\ii\theta},R_{\varepsilon}(b,\varphi,\eta)e^{\ii\eta}\big)\\
	&\quad+\frac{\varepsilon h_2(\varphi,\eta)e^{\ii\eta}}{R_{\varepsilon}(b,\varphi,\eta)}(\nabla_yG_p^{\varepsilon})\big(R_{\varepsilon}(b,\varphi,\theta)e^{\ii\theta},R_{\varepsilon}(b,\varphi,\eta)e^{\ii\eta}\big).
\end{align*}
Thus
$$\|d_rK_{\varepsilon r}^p(r)[h_2]\|_{q,H_{\varphi,\theta,\eta}^s}^{\gamma,\mathscr{O}}\lesssim\|h_2\|_{q,s}^{\gamma,\mathscr{O}}+\|r\|_{q,s}^{\gamma,\mathscr{O}}\|h_2\|_{q,s_0}^{\gamma,\mathscr{O}}.$$
It follows that
\begin{align*}
	\|\partial_{\theta}d_r\mathbf{S}_{\varepsilon r}^p(r)[h_2]h_1\|_{q,s}^{\gamma,\mathscr{O}}&\lesssim\|d_r\mathbf{S}_{\varepsilon r}^p(r)[h_2]h_1\|_{q,s+1}^{\gamma,\mathscr{O}}\\
	&\lesssim\|h_1\|_{q,s+1}^{\gamma,\mathscr{O}}\|h_{2}\|_{q,s_0+1}^{\gamma,\mathscr{O}}+\|h_1\|_{q,s_0}^{\gamma,\mathscr{O}}\|h_{2}\|_{q,s+1}^{\gamma,\mathscr{O}}+\|r\|_{q,s+1}^{\gamma,\mathscr{O}}\|h_1\|_{q,s_0}^{\gamma,\mathscr{O}}\|h_{2}\|_{q,s_0+1}^{\gamma,\mathscr{O}}.
\end{align*}
Putting together the foregoing computations ends the proof of Lemma \ref{lem:tame_XP_EDP}.
\end{proof}

With Lemma \ref{lem:tame_XP_EDP} in hand, we establish tame estimates for the Hamiltonian vector field 
$$X_{\mathcal{P}_{\varepsilon}}=(\partial_{I}\mathcal{P}_{\varepsilon},-\partial_{\vartheta}\mathcal{P}_{\varepsilon},\Pi_{\mathbb{S}}^{\bot}\partial_{\theta}\nabla_{z}\mathcal{P}_{\varepsilon})$$
defined by \eqref{derniere_def} and \eqref{def_ultime}. The torus displacement $\mathfrak I$ is equipped with the norm
$$
\|\mathfrak I\|_{q,s}^{\gamma,\mathscr{O}}
\triangleq
\|\Theta\|_{q,s}^{\gamma,\mathscr{O}}
+
\|I\|_{q,s}^{\gamma,\mathscr{O}}
+
\|z\|_{q,s}^{\gamma,\mathscr{O}}.
$$
\begin{lem}\label{lem;tame_calXP}
		There exists $\varepsilon_0\in(0,1)$ such that if 
		\begin{equation}\label{petit:init}
		    \varepsilon\leqslant\varepsilon_0\quad\textnormal{and}\quad\|\mathfrak{I}\|_{q,s_{0}+2}^{\gamma,\mathscr{O}}\leqslant 1,
		\end{equation} 
		then  the  perturbed Hamiltonian vector field $X_{\mathcal{P}_{\varepsilon}}$ satisfies the following tame estimates for $s\geqslant s_0,$
		\begin{enumerate}[label=(\roman*)]
			\item $\| X_{\mathcal{P}_{\varepsilon}}(i)\|_{q,s}^{\gamma,\mathscr{O}}\lesssim 1+\|\mathfrak{I}\|_{q,s+2}^{\gamma,\mathscr{O}}.$
			\item $\big\| d_{i}X_{\mathcal{P}_{\varepsilon}}(i)[\,\widehat{i}\,]\big\|_{q,s}^{\gamma,\mathscr{O}}\lesssim \|\,\widehat{i}\,\|_{q,s+2}^{\gamma,\mathscr{O}}+\|\mathfrak{I}\|_{q,s+2}^{\gamma,\mathscr{O}}\|\,\widehat{i}\,\|_{q,s_{0}+1}^{\gamma,\mathscr{O}}.$
			\item $\big\| d_{i}^{2}X_{\mathcal{P}_{\varepsilon}}(i)[\,\widehat{i},\widehat{i}\,]\big\|_{q,s}^{\gamma,\mathscr{O}}\lesssim \|\,\widehat{i}\,\|_{q,s+2}^{\gamma,\mathscr{O}}\|\,\widehat{i}\,\|_{q,s_{0}+1}^{\gamma,\mathscr{O}}+\|\mathfrak{I}\|_{q,s+2}^{\gamma,\mathscr{O}}\big(\|\,\widehat{i}\,\|_{q,s_{0}+1}^{\gamma,\mathscr{O}}\big)^{2}.$
	\end{enumerate}
\end{lem}
\begin{proof}
    Follows from Lemma \ref{lem:tame_XP_EDP} similarly to \cite[Lem. 5.2]{HR21}.
\end{proof}
\begin{remark}
    From there, the rest of the paper is similar to \cite{HR21} and \cite{HR21-1} and doesn't bring any novelty. It is just here for the completeness of the proof and therefore the experts in the field can stop here. We only provide the main steps (sometimes simplifying the argument by omitting some technical details like extensions of functions and operators). The interested reader is referred to the given references and should consider what follows as a detailed roadmap in the application of KAM reducibility and Nash-Moser iteration. 
\end{remark}
\subsection{Obtention of an approximate right inverse}\label{sec:approx_inv}
In this subsection we fix $\alpha_0\in\mathbb{R}^d$ and a reversible torus $i_0=(\vartheta_0,I_0,z_0)$
$$\vartheta_0(-\varphi)=-\vartheta_0(\varphi),\qquad I_0(-\varphi)=I_0(\varphi),\qquad z_0(-\varphi)=(\mathscr{S}z_0)(\varphi)$$
that is meant to be a given approximate solution during the Nash-Moser iterative scheme and therefore satisfying the smallnesse condition \eqref{petit:init}. The corresponding patch deformation function is obtained through the transformation $\mathbf{A}$ in \eqref{def:bfA}
$$r=\mathbf{A}(i_0).$$
The torus $i_0$ being reversible, so is $r$, namely
\begin{equation}\label{symr}
    r(-\varphi,-\theta)=r(\varphi,\theta).
\end{equation}
In addition,
\begin{align}
	\|r\|_{q,s}^{\gamma,\mathscr{O}}&\lesssim 1+\|\mathfrak{I}_{0}\|_{q,s}^{\gamma,\mathscr{O}},\label{stima:r}\\
    \|\Delta_{12}r\|_{q,s}^{\gamma,\mathscr{O}}&\lesssim\|\Delta_{12}i\|_{q,s}^{\gamma,\mathscr{O}}+\|\Delta_{12}i\|_{q,s_0}^{\gamma,\mathscr{O}}\max_{k\in\{1,2\}}\|\mathfrak{I}_k\|_{q,s}^{\gamma,\mathscr{O}}. \label{stima:r12}
		\end{align}
here and in the sequel, the notation $\Delta_{12}$ denotes
$$\Delta_{12}f\triangleq f(i_1)-f(i_2),$$
for two given tori $i_1,i_2$ satisfying the smallness condition \eqref{petit:init}.
Our goal here is to construct an approximate inverse to the linearized operator 
$$d_{(i,\alpha)}\mathscr{F}(i_0,\alpha_0)[\hat{\imath},\hat{\alpha}]=\omega\cdot\partial_{\varphi}\hat{\imath}-d_{i}X_{H_\varepsilon^{\alpha_0}}\big(i_{0}(\varphi)\big)[\hat{\imath}]-(\hat{\alpha},0,0).$$
at this state which enjoys tame estimates. The final result is stated in Proposition \ref{prop:approx_full}.
\subsubsection{Triangular system in action-angle-normal variables}
The first step is to implement the Berti-Bolle approach \cite{BB15}, or actually its modified version \cite[Sec. 6]{HHM21}: a suitable choice of diffeomorphism of the phase-space $\textnormal{PS}$ conjugates the linear operator to a triangular system in action-angle-normal variables up to error terms vanishing at an exact solution, thus bringing the problem to the construction of an approximate right inverse in the normal directions only.\\

Assume that $\vartheta_0:\mathbb{T}^d\to\mathbb{T}^d$ is a diffeomorphism and consider the transformation $G_0:\textnormal{PS}\to\textnormal{PS}$ given by
$$\forall(\phi,\mathtt{y},\mathtt{w})\in\textnormal{PS},\quad G_0(\mathtt{\phi},\mathtt{y},\mathtt{w})\triangleq\begin{bmatrix}
    \theta_0(\phi)\\
    I_0(\phi)+\mathtt{L}_1(\phi)\mathtt{y}+\mathtt{L}_2(\phi)\mathtt{w}\\
    z_0(\phi)+\mathtt{w}
\end{bmatrix},$$
where
$$\mathtt{L}_1(\phi)\triangleq\big(\partial_{\varphi}\vartheta_0(\varphi)\big)^{-\top},\qquad\mathtt{L}_2(\varphi)\triangleq\Big(\partial_{\vartheta}\bar{z}_0\big(\vartheta_0(\varphi)\big)\Big)^{\top}\partial_{\theta}^{-1},\qquad \bar{z}_0(\vartheta)\triangleq z_0\big(\vartheta_0(\vartheta)\big).$$
Its linearization at the trivial torus gives
$$DG_0\big(i_{\textnormal{\tiny{triv}}}(\varphi)\big)=\begin{bmatrix}
    \partial_{\varphi}\vartheta_0(\varphi) & 0 & 0\\
    \partial_{\varphi}I_0(\varphi) & \mathtt{L}_1(\varphi) & \mathtt{L}_2(\varphi)\\
    \partial_{\varphi}z_0(\varphi) & 0 & \textnormal{Id}_{L^2_{\bot}} \\
\end{bmatrix}.$$
The torus $i_0$ being reversible it is easy to check the invariance
$$G_0\circ\mathfrak{S}=\mathfrak{S}\circ G_0$$
and consequently, the linear transformation $DG_0$ is reversibility preserving in the sense that it sends reversible tori onto reversible tori.
Then, denoting
$$\bar{G}_0(i,\alpha)\triangleq\big(G_0(i),\alpha\big),$$
we have, according to \cite[Sec. 6]{HHM21}, the following conjugation formula
\begin{equation}\label{conjugaison:Berti-Bolle}
    DG_0(i_{\textnormal{\tiny{triv}}})^{-1}d_{(i,\alpha)}\mathscr{F}(i_0,\alpha_0)D\bar{G}_0(i_{\textnormal{\tiny{triv}}})=\mathcal{T}+\mathcal{E},
\end{equation}
where
\begin{enumerate}[label=\textbullet]
    \item the transformation $\mathcal{T}$ is triangular in the variables $\hat{\phi},\hat{y},\hat{w}$ and given by
\begin{equation}\label{triangular_system}
    \mathcal{T}[\hat{\phi},\hat{\mathtt{y}},\hat{\mathtt{w}},\hat{\alpha}]\triangleq\begin{bmatrix}
        \omega\cdot\partial_{\varphi}\hat{\phi}-\big[H_{20}(\varphi)\hat{\mathtt{y}}+H_{11}^{\top}(\varphi)\hat{\mathtt{w}}+\mathtt{L}_1^{\top}(\varphi)\hat{\alpha}\big]\\
        \omega\cdot\partial_{\varphi}\mathtt{y}+\mathtt{P}(\varphi)\hat{\alpha}\\
        \mathcal{L}_{\bot}\hat{\mathtt{w}}-\partial_{\theta}\big[H_{11}(\varphi)\hat{\mathtt{y}}-\mathtt{L}_{2}^{\top}\hat{\alpha}\big]
    \end{bmatrix},
\end{equation}
with
\begin{align*}
\mathtt{P}(\varphi)&\triangleq \big(\partial_{\varphi}\vartheta_0(\varphi)\big)^{\top}\partial_{\varphi}I_0(\varphi)\mathtt{L}_1^{\top}(\varphi)+\big(\partial_{\varphi}z_0(\varphi)\big)^{\top}\mathtt{L}_2^{\top}(\varphi):\mathbb{R}^d\to\mathbb{R}^d,\\
    H_{20}(\varphi)&\triangleq\varepsilon\mathtt{L}_1^{\top}(\varphi)(\partial_{II}\mathcal{P}_{\varepsilon})\big(i_0(\varphi)\big)\mathtt{L}_1(\varphi):\mathbb{R}^d\to\mathbb{R}^d,\\
    H_{11}(\varphi)&\triangleq\varepsilon\mathtt{L}_2^{\top}(\varphi)(\partial_{II}\mathcal{P}_{\varepsilon})\big(i_0(\varphi)\big)\mathtt{L}_1(\varphi)+\varepsilon(\partial_{I}\nabla_z\mathcal{P}_{\varepsilon})\big(i_0(\varphi)\big)\mathtt{L}_1(\varphi):\mathbb{R}^d\to L^2_\bot.
\end{align*}
The operator $\mathcal{L}_{\bot}$ is the linearized operator in the normal directions defined by
\begin{equation}\label{lin_zdir}
	\begin{aligned}
	    \mathcal{L}_{\bot}&\triangleq \Pi_{\mathbb{S}_0}^{\bot}\Big(\omega\cdot\partial_{\varphi}-\partial_{\theta}( \partial_{z}\nabla_z H_\varepsilon^{\alpha_0}) \big(i_0(\varphi)\big)-\varepsilon\partial_{\theta}\mathcal{R}(\varphi)\Big)\Pi_{\mathbb{S}_0}^{\bot}\\
    &=\Pi_{\mathbb{S}_0}^{\bot}\Big(\mathcal{L}_{\varepsilon r}-\varepsilon\partial_{\theta}\mathcal{R}\Big)\Pi_{\mathbb{S}_0}^{\bot},
	\end{aligned}
\end{equation}
where $\mathcal{L}_{\varepsilon r}$ is the full linearized operator in the original variables obtained in Lemma \ref{lem:lin_origin} and
\begin{align}\label{finite_rank}
	\mathcal{R}(\varphi)\triangleq L_{2}^{\top}(\varphi)\partial_{I}\nabla_I\mathcal{P}_{\varepsilon}\big(i_{0}(\varphi)\big)L_{2}(\varphi)+L_{2}^{\top}(\varphi)\partial_{z}\nabla_{I}\mathcal{P}_{\varepsilon}\big(i_{0}(\varphi)\big)+\partial_{I}\nabla_{z}\mathcal{P}_{\varepsilon}\big(i_{0}(\varphi)\big)L_{2}(\varphi).
\end{align}
Note that, similarly to \cite[Prop. 6.1]{HR21-1}, we can show that the operator $\mathcal{R}(\varphi)$ is reversibility preserving, finite rank, and admits the integral form
$$\mathcal{R}(\varphi)[h](\varphi,\theta)=\int_{\mathbb{T}}h(\varphi,\eta)K(\varphi,\theta,\eta)d\eta,$$
where $K$ is a smooth kernel satisfying the reversibility preserving condition
\begin{equation}\label{sym:K}
    K(-\varphi,-\theta,-\eta)=K(\varphi,\theta,\eta)
\end{equation}
and the estimates
\begin{align}
    \|K\|_{q,s}&\lesssim 1+\|\mathfrak{I}_0\|_{q,s}^{\gamma,\mathscr{O}}\label{stima:K}\\
    \|\Delta_{12}K\|_{q,s}^{\gamma,\mathscr{O}}&\lesssim\|\Delta_{12}i\|_{q,s+3}^{\gamma,\mathscr{O}}+\|\Delta_{12}i\|_{q,s_0+3}^{\gamma,\mathscr{O}}\max_{k\in\{1,2\}}\|\mathfrak{I}_k\|_{q,s+3}^{\gamma,\mathscr{O}}.\label{stima:K12}
\end{align}
\item the operator $\mathcal{E}$ is an error term vanishing at an exact solution and taking the form
\begin{align*}
    \mathcal{E}[\hat{\phi},\hat{\mathtt{y}},\hat{\mathtt{w}}]&\triangleq DG_0(i_{\textnormal{\tiny{triv}}})^{-1}\partial_{\varphi}\mathtt{Z}(\varphi)\hat{\phi}\\
    &\quad+\begin{bmatrix}
0\\
\mathtt{Q}(\varphi)\big[\big(\partial_{\theta}( \partial_{z}\nabla_z H_\varepsilon^{\alpha_0}) \big(i_0(\varphi)\big)+\mathcal{R}(\varphi)\big)\hat{\mathtt{y}}+H_{11}^{\top}(\varphi)\hat{\mathtt{w}}\big]-R_{10}(\varphi)\hat{\mathtt{y}}-R_{01}(\varphi)\hat{\mathtt{w}}\\
0
\end{bmatrix},
\end{align*}
where
\begin{align*}
    \mathtt{Q}(\varphi)&\triangleq\big(\partial_{\varphi} \vartheta_0(\varphi)\big)^{\top}\partial_{\varphi} I_0(\varphi)-\big(\partial_\varphi I_0(\varphi)\big)^{\top}\partial_\varphi \vartheta_0(\varphi)  -\big(\partial_\varphi z_0(\varphi)\big)^{\top} \partial_{\theta}^{-1} \partial_\varphi z_0(\varphi):\mathbb{R}^d\to\mathbb{R}^d,\\
    R_{10}(\varphi)&\triangleq\big(\partial_{\varphi}\mathtt{Z}_1(\varphi)\big)^{\top}\mathtt{L}_1(\varphi):\mathbb{R}^d\to\mathbb{R}^d,\\
    R_{01}(\varphi)&\triangleq\big(\partial_{\varphi}\mathtt{Z}_1(\varphi)\big)^{\top}\mathtt{L}_2(\varphi)\partial_{\theta}^{-1}-\big(\partial_{\varphi}\mathtt{Z}_3(\varphi)\big)^{\top}\partial_{\theta}^{-1}:L^2_{\bot}\to\mathbb{R}^d,
\end{align*}
with
$$\mathtt{Z}=(\mathtt{Z}_1,\mathtt{Z}_2,\mathtt{Z}_3)\triangleq\mathscr{F}(i_0,\alpha_0).$$
\end{enumerate}
The triangular system \eqref{triangular_system} is solved via the inversion of the operators $\omega\cdot\partial_{\varphi}$ and $\mathcal{L}_{\bot}.$ The following subsubsection focuses on this last one. 
\subsubsection{Diagonalization and inversion in the normal directions}\label{sec:reductionKAM}
We shall now construct the approximate right inverse for the linearized operator restricted to the normal directions $\mathcal{L}_{\bot}$ introduced in \eqref{lin_zdir}. It is obtained via a reducibility procedure, conjugating $\mathcal{L}_{\bot}$ to a Fourier multiplier.\\

Define the scale
$$N_n\triangleq N_0^{\left(\frac{3}{2}\right)^n},\qquad N_0\gg1.$$
Throughout this section, we may consider the following parameters
\begin{equation}\label{prmtrrr}
	\begin{array}{ll} 
		s_{l}\triangleq s_0+\tau_1q+\tau_1+2,& \overline{\mu}\triangleq 4\tau_1q+6\tau_1+3, \\
		\overline{s}_{l}\triangleq s_l+\tau_2q+\tau_2, & \overline{s}_{h}\triangleq \frac{3}{2}\overline{\mu}+s_{l}+1\\
        \mu\geqslant \overline{\mu}+2\tau_2+\tau_2,&s_{h}\geqslant\max\left(\frac{3}{2}\mu+\overline{s}_{l}+1,\overline{s}_{h}+\mathtt{p}\right)\\
        \mathtt{p}\geqslant0.
	\end{array}
\end{equation}

Assume the smallness condition: there exists $\varepsilon_0>0$ small enough and $\sigma_{max}>0$ such that
\begin{equation}\label{petitesse}
    \varepsilon\gamma^{-2-q}N_0^{\mu}\leqslant\varepsilon_0,\qquad\|\mathfrak{I}_0\|_{q,s_h+\sigma_{max}}^{\gamma,\mathscr{O}}\leqslant1.
\end{equation}
In what follows, the Lipschitz estimates for $\Delta_{12}$ quantities assume that both tori $i_1$ and $i_2$ satisfy \eqref{petitesse}.\\

Let us start with the transport reduction to constant coefficients for the full linear operator. Classically, this is done via a symplectic change of variables \cite[Prop. 6.2]{HR21}. The next proposition explains this straightening and provides the effect of such transformation on lower order terms of the linearized operator.
\begin{prop}\label{prop:transport}
Let $\upsilon\in\left(0,\frac{1}{q+2}\right].$ 
	There exists $\sigma\triangleq\sigma(q,d,\tau_1)>0$ such that under the smallness condition \eqref{petitesse} with $\sigma\leqslant\sigma_{max}$, there exist
    \begin{enumerate}[label=\textbullet]
        \item an invertible transformation $\mathscr{B}$ in the form
    $$\mathscr{B}=(1+\partial_{\theta}\beta)\mathcal{B},\qquad\mathcal{B}h(b,\omega,\varphi,\theta)\triangleq h\big(b,\omega,\varphi,\theta+\beta(b,\omega,\varphi,\theta)\big),$$
    with 
    \begin{equation}\label{oddness_beta}
				\beta\in\bigcap_{s\in[s_{0},S]}W^{q,\infty,\gamma }(\mathscr{O},H^{s})\textnormal{ odd in }(\varphi,\theta)
			\end{equation}
            satisfying for all $s\in[s_0,S]$ and $\mathtt{p}\geqslant0,$
    \begin{align}
				\|\mathscr{B}^{\pm 1}h\|_{q,s}^{\gamma ,\mathscr{O}}&\lesssim\|h\|_{q,s}^{\gamma ,\mathscr{O}}+\varepsilon\gamma^{-1}\| \mathfrak{I}_{0}\|_{q,s+\sigma}^{\gamma ,\mathscr{O}}\|h\|_{q,s_{0}}^{\gamma ,\mathscr{O}},\label{stima:CVAR}\\
				\|\beta\|_{q,s}^{\gamma ,\mathscr{O}}&\lesssim \varepsilon\gamma ^{-1}\left(1+\| \mathfrak{I}_{0}\|_{q,s+\sigma}^{\gamma ,\mathscr{O}}\right),\label{stima:beta}\\
                \|\Delta_{12}\beta\|_{q,\overline{s}_{h}+\mathtt{p}}^{\gamma,\mathscr{O}}&\lesssim\varepsilon\gamma^{-1}\|\Delta_{12}i\|_{q,\overline{s}_{h}+\mathtt{p}+\sigma}^{\gamma,\mathscr{O}};\label{Lipschitz_beta}
			\end{align}
    \item a coefficient $V_{i_0}^{\infty}\in W^{q,\infty,\gamma }(\mathscr{O},\mathbb{C})$ satisfying
    \begin{align}
		\|V_{i_{0}}^{\infty}-\tfrac{1}{2}\|_{q}^{\gamma ,\mathscr{O}}&\lesssim \varepsilon,\label{stima:Vinfty}\\
        \|\Delta_{12}V_{i}^{\infty}\|_{q}^{\gamma,\mathscr{O}}&\lesssim\varepsilon\| \Delta_{12}i\|_{q,\overline{s}_{h}+2}^{\gamma,\mathscr{O}} \label{Lipschitz V}
    \end{align}
    \end{enumerate}
such that for any $n\in\mathbb{N}$, in restriction to the Cantor set
			\begin{equation}\label{def Cantor trans}
				\mathscr{O}_{n}^{(\mathtt{TR})}(\gamma,\tau_{1},i_{0})\triangleq\bigcap_{(l,j)\in\mathbb{Z}^{d}\times\mathbb{Z}\setminus\{(0,0)\}\atop|l|\leqslant N_{n}}\left\lbrace(b,\omega)\in \mathscr{O}\quad\textnormal{s.t.}\quad\big|\omega\cdot l+jV_{i_{0}}^{\infty}(b,\omega)\big|>\tfrac{4\gamma^{\upsilon}\langle j\rangle}{\langle l\rangle^{\tau_{1}}}\right\rbrace,
			\end{equation}
			the following decomposition holds
			$$\mathscr{B}^{-1}\mathcal{L}_{\varepsilon r}\mathscr{B}=\omega\cdot\partial_{\varphi}+V_{i_{0}}^{\infty}\partial_{\theta}+\tfrac{1}{2}\mathcal{H}-\partial_{\theta}\mathcal{K}_{b}\ast\cdot+\mathfrak{R}_{\varepsilon r}+\mathtt{E}_{n}^{(\mathtt{TR})},$$
			where
            \begin{enumerate}[label=\textbullet]
                \item $\mathcal{H}$ is the $2\pi$-periodic Hilbert transform
    $$\mathcal{H}[h](\theta)\triangleq\int_{\mathbb{T}}\cot\left(\tfrac{\theta-\eta}{2}\right)h(\eta)d\eta;$$
    \item $\mathcal{K}_{b}$ is the kernel defined in Lemma \ref{lem:equilibre_HR21};
    \item $\mathfrak{R}_{\varepsilon r}$ is a real and reversibility preserving integral operator satisfying
	\begin{align}
		\forall s\in[s_0,S],\quad \max_{k\in\{0,1\}}\|\partial_{\theta}^{k}\mathfrak{R}_{\varepsilon r}\|_{\textnormal{\tiny{O-d}},q,s}^{\gamma,\mathscr{O}}&\lesssim\varepsilon\gamma^{-1}\left(1+\|\mathfrak{I}_{0}\|_{q,s+\sigma}^{\gamma,\mathscr{O}}\right).\label{stima:Rfrac}\\
        \|\Delta_{12}\mathfrak{R}_{\varepsilon r}\|_{\textnormal{\tiny{O-d}},q,\overline{s}_{h}+\mathtt{p}}^{\gamma,\mathscr{O}}&\lesssim\varepsilon\gamma^{-1}\|\Delta_{12}i\|_{q,\overline{s}_{h}+\mathtt{p}+\sigma}^{\gamma,\mathscr{O}};\label{stima:Rfrac12}
	\end{align}
     \item $\mathtt{E}_{n}^{(\mathtt{TR})}=\mathtt{E}_{n}^{(\mathtt{TR})}(b,\omega,i_{0})$ is a linear operator satisfying
	   \begin{equation}\label{stima:EnTR}
       \|\mathtt{E}_{n}^{(\mathtt{TR})}h\|_{q,s_{0}}^{\gamma,\mathscr{O}}\lesssim\varepsilon N_{0}^{\mu}N_{n+1}^{-\mu}\|h\|_{q,s_{0}+2}^{\gamma,\mathscr{O}}.
	   \end{equation}
\end{enumerate}
\end{prop}
\begin{proof}
According to Lemma \ref{lem:lin_origin}, we have
\begin{equation}\label{rappel:L}
    \mathcal{L}_{\varepsilon r}=\omega\cdot\partial_{\varphi}+\partial_{\theta}(V_{\varepsilon r}\cdot)+\partial_{\theta}\mathbf{L}_{\varepsilon r}+\partial_{\theta}\mathbf{S}_{\varepsilon r}.
\end{equation}
	 Recall from \eqref{dec VLS} that
	$$V_{\varepsilon r}=V_{\varepsilon r}^{\mathbb{D}}+\varepsilon^2V_{\varepsilon r}^p,$$
	with $V_{\varepsilon r}^{\mathbb{D}}$ and $V_{\varepsilon r}^p$ defined by \eqref{VerD} and \eqref{Verp}, respectively. The following estimate can be found in \cite{HR21-1}, see also \eqref{stima:r},
    $$\|V_{\varepsilon r}^{\mathbb{D}}-\tfrac{1}{2}\|_{q,s}^{\gamma,\mathscr{O}}\lesssim\varepsilon\|r\|_{q,s+1}^{\gamma,\mathscr{O}}\lesssim\varepsilon\left(1+\|\mathfrak{I}_{0}\|_{q,s+1}^{\gamma,\mathscr{O}}\right).$$
Also recall the reversibility property \eqref{symVr} of the transport coefficient.
	These properties allow to implement the KAM reducibility scheme devised in \cite[Prop. 6.2]{HR21} providing $\beta$ and $V_{i_0}^{\infty}$ with their corresponding estimates such that, for any $n\in\mathbb{N},$ in restriction to the Cantor set $\mathscr{O}_{n}^{(\mathtt{TR})}(\gamma,\tau_{1},i_{0})$ in \eqref{def Cantor trans}, the following identity holds
	\begin{equation}\label{reduc lin trans op}
		\mathscr{B}^{-1}\big(\omega\cdot\partial_{\varphi}+\partial_{\theta}\big(V_{\varepsilon r}\cdot\big)\big)\mathscr{B}=\omega\cdot\partial_{\varphi}+V_{i_{0}}^{\infty}\partial_{\theta}+\mathtt{E}_{n}^{(\mathtt{TR})},
	\end{equation}
	where $\mathtt{E}_{n}^{(\mathtt{TR})}$ is an operator enjoying the decay property stated in \eqref{stima:EnTR}. In view of \eqref{mathbfLr} and \eqref{bfSr_EDP}, we have the following decompositions
	\begin{equation}\label{dec ops}
		\partial_{\theta}\mathbf{L}_{\varepsilon r}=\partial_{\theta}\mathcal{K}_{1}\ast\cdot+\partial_{\theta}\mathcal{I}_{\mathbb{K}_{\varepsilon r}^{\mathbf{L}}},\qquad\partial_{\theta}\mathbf{S}_{\varepsilon r}=-\partial_{\theta}\mathcal{K}_{b}\ast\cdot+\partial_{\theta}\mathcal{I}_{\mathbb{K}_{\varepsilon r}^{\mathbf{S}}},
	\end{equation}
	where $\mathcal{I}_{\mathbb{K}_{\varepsilon r}^{\mathbf{L}}}$ and $\mathcal{I}_{\mathbb{K}_{\varepsilon r}^{\mathbf{S}}}$ are integral operators in the sense \eqref{def:integr} with kernels
    \begin{align*}
        \mathbb{K}_{\varepsilon r}^{\mathbf{L}}&\triangleq\log(1+v_{\varepsilon r}), \qquad v_{\varepsilon r}(\varphi,\theta,\eta)\triangleq\left(R_{\varepsilon}(\varphi,\theta)R_{\varepsilon}(\varphi,\eta)+\left(\frac{R_{\varepsilon}(\varphi,\theta)-R_{\varepsilon}(\varphi,\eta)}{2\sin\left(\frac{\theta-\eta}{2}\right)}\right)^2\right)^{\frac{1}{2}}
    \end{align*}
    and
    \begin{align*}
        \mathbb{K}_{\varepsilon r}^{\mathbf{S}}(\varphi,\theta,\eta)&\triangleq\varepsilon^2 G_{p}^{\varepsilon}\big(R_{\varepsilon}(\varphi,\theta)e^{\ii\theta},R_{\varepsilon}(\varphi,\eta)e^{\ii\eta}\big)-\log\left(\frac{B_{\varepsilon r}(\varphi,\theta,\eta)}{B_{0}(\varphi,\theta,\eta)}\right).
    \end{align*}
    From \eqref{symr} and \eqref{sym G}, we have
    \begin{equation}\label{sym:bbK-LS}
        \mathbb{K}_{\varepsilon r}^{\mathbf{L}}(-\varphi,-\theta,-\eta)=\mathbb{K}_{\varepsilon r}^{\mathbf{L}}(\varphi,\theta,\eta),\qquad \mathbb{K}_{\varepsilon r}^{\mathbf{S}}(-\varphi,-\theta,-\eta)=\mathbb{K}_{\varepsilon r}^{\mathbf{S}}(\varphi,\theta,\eta).
    \end{equation}
    Making appeal to Lemma \ref{lem:funct_topo}-$(v)$ and \cite[Lem. 4.2]{HR21} and using \eqref{stima:r}-\eqref{stima:r12}, we get
    \begin{equation}\label{stima:kernelsLS}
        \begin{aligned}
        \|\mathbb{K}_{\varepsilon r}^{\mathbf{L}}\|_{q,s}^{q,\mathscr{O}}+\|\mathbb{K}_{\varepsilon r}^{\mathbf{S}}\|_{q,s}^{q,\mathscr{O}}&\lesssim\varepsilon\|r\|_{q,s+1}^{\gamma,\mathscr{O}}\lesssim\varepsilon\left(1+\|\mathfrak{I}_{0}\|_{q,s+1}^{\gamma,\mathscr{O}}\right),\\
        \|\Delta_{12}\mathbb{K}_{\varepsilon r}^{\mathbf{L}}\|_{q,s}^{q,\mathscr{O}}+\|\Delta_{12}\mathbb{K}_{\varepsilon r}^{\mathbf{S}}\|_{q,s}^{q,\mathscr{O}}&\lesssim\varepsilon\|\Delta_{12}r\|_{q,s+1}^{\gamma,\mathscr{O}}\lesssim\varepsilon\left(\|\Delta_{12}i\|_{q,s+1}^{\gamma,\mathscr{O}}+\|\Delta_{12}i\|_{q,s_0}^{\gamma,\mathscr{O}}\max_{k\in\{1,2\}}\|\mathfrak{I}_k\|_{q,s+1}^{\gamma,\mathscr{O}}\right).
    \end{aligned}
    \end{equation}
    Also note the identity,
    $$\partial_{\theta}\mathcal{K}_{1}\ast\cdot=\tfrac{1}{2}\mathcal{H}.$$
    Therefore, using \eqref{rappel:L}, and \eqref{reduc lin trans op}, one obtains in the Cantor set $\mathscr{O}_{n}^{(\mathtt{TR})}(\gamma,\tau_{1},i_{0})$ the following decomposition
    \begin{align*}
		\mathscr{B}^{-1}\mathcal{L}_{\varepsilon r}\mathscr{B}&=\mathscr{B}^{-1}\big(\omega\cdot\partial_{\varphi}+\partial_{\theta}\left(V_{\varepsilon r}\cdot\right)\big)\mathscr{B}+\mathscr{B}^{-1}\partial_{\theta}\mathbf{L}_{\varepsilon r}\mathscr{B}+\mathscr{B}^{-1}\partial_{\theta}\mathbf{S}_{\varepsilon r}\mathscr{B}\\
		&=\omega\cdot\partial_{\varphi}+V_{i_{0}}^{\infty}\partial_{\theta}+\tfrac{1}{2}\mathcal{H}-\partial_{\theta}\mathcal{K}_{b}\ast\cdot+\mathfrak{R}_{\varepsilon r}+\mathtt{E}_{n}^{(\mathtt{TR})},
	\end{align*}
    where
    \begin{align*}
        \mathfrak{R}_{\varepsilon r}&\triangleq\tfrac{1}{2}(\mathscr{B}^{-1}\mathcal{H}\mathscr{B}-\mathcal{H})-\Big(\mathscr{B}^{-1}\partial_{\theta}\big(\mathcal{K}_{b}\ast\cdot\big)\mathscr{B}-\partial_{\theta}\mathcal{K}_{b}\ast\cdot\Big)+\mathscr{B}^{-1}\mathcal{I}_{\mathbb{K}_{\varepsilon r}^{\mathbf{L}}}\mathscr{B}+\mathscr{B}^{-1}\mathcal{I}_{\mathbb{K}_{\varepsilon r}^{\mathbf{S}}}\mathscr{B}.
    \end{align*}
    Combining \eqref{sym:bbK-LS}, \eqref{def:calKb}, \eqref{oddness_beta} with Lemma \ref{lem:conj_intgral}, Lemma \ref{lem:Hilbert}, we deduce that $\mathfrak{R}_{\varepsilon r}$ is a reversible Toeplitz-in-time integral operator satisfying by virtue of \eqref{stima:conj_int}, \eqref{stima:conj_int_bis}, \eqref{stima:D12_conj}, \eqref{stime:Hilbert}, \eqref{stima:beta}, \eqref{Lipschitz_beta}, \eqref{stima:kernelsLS} and \eqref{petitesse} the estimates \eqref{stima:Rfrac}-\eqref{stima:Rfrac12}, up to increase the value of $\sigma$ if necessary. This achieves the proof of Proposition \ref{prop:transport}.
\end{proof}
Next, we study the effects of the localization in the normal directions for the reduction of the transport part. For that purpose, we consider the localized quasi-periodic symplectic change of coordinates defined by
\begin{equation}\label{def:Bbot}
    \mathscr{B}_{\bot}\triangleq\Pi_{\mathbb{S}_0}^{\bot}\mathscr{B}\Pi_{\mathbb{S}_{0}}^{\bot}.
\end{equation}
It satisfies the estimates for $s\in[s_0,S],$
\begin{equation}\label{stima:Bbot}
	\|\mathscr{B}_{\bot}^{\pm 1}h\|_{q,s}^{\gamma ,\mathscr{O}}\lesssim\|h\|_{q,s}^{\gamma ,\mathscr{O}}+\varepsilon\gamma ^{-1}\| \mathfrak{I}_{0}\|_{q,s+\sigma}^{\gamma ,\mathscr{O}}\|h\|_{q,s_{0}}^{\gamma ,\mathscr{O}}.
\end{equation}
Then, we have the following result.
\begin{prop}\label{prop:normal_projection}
	There exist $\sigma\triangleq\sigma(\tau_{1},q,d,s_{0})>0$ such that under the smallness condition \eqref{petitesse} with $\sigma\leqslant\sigma_{max}$, for any $n\in\mathbb{N},$ the following identity holds in restriction to the Cantor set $\mathscr{O}_{n}^{(\mathtt{TR})}(\gamma,\tau_{1},i_{0})$ in \eqref{def Cantor trans},
    $$\mathscr{B}_{\bot}^{-1}\mathcal{L}_{\bot}\mathscr{B}_{\bot}= \omega\cdot\partial_{\varphi}\Pi_{\mathbb{S}_0}^{\bot}+\mathscr{D}_{0}+\mathscr{R}_{0}+\mathtt{E}_{n}^{(\mathtt{Nor})},$$
        where
	\begin{enumerate}[label=\textbullet]
		 \item $\mathcal{L}_{\bot}$ and $\mathscr{B}_{\bot}$ are as in \eqref{lin_zdir} and \eqref{def:Bbot}.
		\item $\mathscr{D}_{0}=\Pi_{\mathbb{S}_0}^\bot \mathscr{D}_{0}\Pi_{\mathbb{S}_0}^\bot\triangleq\big(V_{i_{0}}^{\infty}\partial_{\theta}+\tfrac{1}{2}\mathcal{H}-\partial_{\theta}\mathcal{K}_{b}\ast\cdot\big)\Pi_{\mathbb{S}_0}^{\bot}$ is a reversible Fourier multiplier operator given by 
		\begin{equation}\label{initial diago}
		    \forall (l,j)\in \mathbb{Z}^{d}\times\mathbb{S}_{0}^{c},\quad  \mathscr{D}_{0}\mathbf{e}_{l,j}=\ii\,\mathtt{d}_{j}^{0}\,\mathbf{e}_{l,j},\qquad \mathtt{d}_{j}^{0}(b,\omega,i_{0})\triangleq\Omega_{j}(b)+j\big(V_{i_{0}}^{\infty}(b,\omega)-\tfrac{1}{2}\big).
		\end{equation}
        \item $\mathscr{R}_{0}=\Pi_{\mathbb{S}_0}^\bot \mathscr{R}_{0}\Pi_{\mathbb{S}_0}^\bot$  is a real and reversible Toeplitz in time operator satisfying 
        \begin{align}
            \forall s\in [s_{0},S],\quad \max_{k\in\{0,1\}}\|\partial_{\theta}^{k}\mathscr{R}_{0}\|_{\textnormal{\tiny{O-d}},q,s}^{\gamma,\mathscr{O}}&\lesssim\varepsilon\gamma^{-1}\left(1+\| \mathfrak{I}_{0}\|_{q,s+\sigma}^{\gamma,\mathscr{O}}\right),\label{stima:R0}\\
            \|\Delta_{12}\mathscr{R}_{0}\|_{\textnormal{\tiny{O-d}},q,\overline{s}_{h}+\mathtt{p}}^{\gamma,\mathscr{O}}&\lesssim\varepsilon\gamma^{-1}\| \Delta_{12}i\|_{q,\overline{s}_{h}+\mathtt{p}+\sigma}^{\gamma,\mathscr{O}}.\label{LipschitzR0}
        \end{align}
		\item The operator $\mathtt{E}_{n}^{(\mathtt{Nor})}=\Pi_{\mathbb{S}_0}^\bot \mathtt{E}_{n}^{(\mathtt{Nor})}\Pi_{\mathbb{S}_0}^\bot$ satisfies the following estimate
		\begin{equation}\label{stima:ttEn_Nor}
			\|\mathtt{E}_{n}^{(\mathtt{Nor})}h\|_{q,s_{0}}^{\gamma,\mathscr{O}}\lesssim \varepsilon N_{0}^{\mu}N_{n+1}^{-\mu}\|h\|_{q,s_{0}+2}^{\gamma,\mathscr{O}}.
		\end{equation}
	\end{enumerate}
\end{prop}
\begin{proof}
Recall from Proposition \ref{prop:transport} that, in restriction to the Cantor set $\mathcal{O}_{n}^{(\mathtt{TR})}(\gamma,\tau_1,i_0)$, the following identity holds
$$\mathcal{L}_{\varepsilon r}\mathscr{B}=\mathscr{B}\big(\omega\cdot\partial_{\varphi}+V_{i_{0}}^{\infty}\partial_{\theta}+\tfrac{1}{2}\mathcal{H}-\partial_{\theta}\mathcal{K}_{b}\ast\cdot+\mathfrak{R}_{\varepsilon r}+\mathtt{E}_{n}^{(\mathtt{TR})}\big).$$
Also observe that
$$\textnormal{Id}=\Pi_{\mathbb{S}_0}+\Pi_{\mathbb{S}_0}^\bot,\qquad (\Pi_{\mathbb{S}_0}^\bot)^2=\Pi_{\mathbb{S}_0}^\bot,\qquad\Pi_{\mathbb{S}_0}\Pi_{\mathbb{S}_0}^\bot=\Pi_{\mathbb{S}_0}^\bot\Pi_{\mathbb{S}_0}=0.$$
Therefore, denoting 
\begin{equation}\label{def:ttEn_nor}
		\mathtt{E}_{n}^{(\mathtt{Nor})}\triangleq \mathscr{B}_{\bot}^{-1}\Pi_{\mathbb{S}_{0}}^{\bot}\mathscr{B}\mathtt{E}_{n}^{(\mathtt{TR})}\Pi_{\mathbb{S}_{0}}^{\bot},
	\end{equation}
we obtain, in restriction to the Cantor set $\mathcal{O}_{n}^{(\mathtt{TR})}(\gamma,\tau_1,i_0)$, the following identity
\begin{align*}
	\mathscr{B}_{\bot}^{-1}\mathcal{L}_{\bot}\mathscr{B}_{\bot}&=\mathscr{B}_{\bot}^{-1} \Pi_{\mathbb{S}_0}^\bot \big(  \mathcal{L}_{\varepsilon r} - \varepsilon  \partial_{\theta} \mathcal{R} \big)\Pi_{\mathbb{S}_0}^\bot \mathscr{B}\Pi_{\mathbb{S}_0}^\bot \\
    & = \mathscr{B}_{\bot}^{-1}\Pi_{\mathbb{S}_0}^\bot \mathcal{L}_{\varepsilon r} \mathscr{B} \Pi_{\mathbb{S}_0}^\bot - \mathscr{B}_{\bot}^{-1} \mathcal{L}_{\varepsilon r} \Pi_{\mathbb{S}_0} \mathscr{B} \Pi_{\mathbb{S}_0}^\bot - \varepsilon \mathscr{B}_{\bot}^{-1} \partial_{\theta}\mathcal{R} \mathscr{B}_{\bot}  \\
    & = \mathscr{B}_{\bot}^{-1} \Pi_{\mathbb{S}_0}^\bot\mathscr{B} \big(\omega\cdot\partial_{\varphi}+V_{i_{0}}^{\infty}\partial_{\theta}+\tfrac{1}{2}\mathcal{H}-\partial_{\theta}\mathcal{K}_{b}\ast\cdot+\mathfrak{R}_{\varepsilon r}\big)\Pi_{\mathbb{S}_0}^\bot\\
    &\quad- \mathscr{B}_{\bot}^{-1} \mathcal{L}_{\varepsilon r} \Pi_{\mathbb{S}_0} (\mathscr{B}-\textnormal{Id}) \Pi_{\mathbb{S}_0}^\bot - \varepsilon \mathscr{B}_{\bot}^{-1} \partial_{\theta}\mathcal{R} \mathscr{B}_{\bot}+\mathtt{E}_{n}^{(\mathtt{Nor})}\\
     & = \mathscr{B}_{\bot}^{-1} \Pi_{\mathbb{S}_0}^\bot\mathscr{B}(\Pi_{\mathbb{S}_0}^\bot+\Pi_{\mathbb{S}_0}) \big(\omega\cdot\partial_{\varphi}+V_{i_{0}}^{\infty}\partial_{\theta}+\tfrac{1}{2}\mathcal{H}-\partial_{\theta}\mathcal{K}_{b}\ast\cdot+\mathfrak{R}_{\varepsilon r}\big)\Pi_{\mathbb{S}_0}^\bot\\
     &\quad- \mathscr{B}_{\bot}^{-1} \mathcal{L}_{\varepsilon r} \Pi_{\mathbb{S}_0} (\mathscr{B}-\textnormal{Id})\Pi_{\mathbb{S}_0}^\bot - \varepsilon \mathscr{B}_{\bot}^{-1} \partial_{\theta}\mathcal{R} \mathscr{B}_{\bot}+\mathtt{E}_{n}^{(\mathtt{Nor})}\\
    &= \mathscr{D}_0 + \mathscr{R}_{0}+\mathtt{E}_{n}^{(\mathtt{Nor})},
\end{align*}
where 
$$\mathscr{D}_0\triangleq \big(V_{i_{0}}^{\infty}\partial_{\theta}+\tfrac{1}{2}\mathcal{H}-\partial_{\theta}\mathcal{K}_{b}\ast\cdot\big)\Pi_{\mathbb{S}_0}^{\bot}$$
and
\begin{equation}\label{new.finiterank}
	\begin{aligned}
		\mathscr{R}_{0} & \triangleq \Pi_{\mathbb{S}_0}^\bot\mathfrak{R}_{\varepsilon r}\Pi_{\mathbb{S}_0}^\bot+\mathscr{B}_\bot^{-1} \Pi_{\mathbb{S}_0}^\bot \mathscr{B}\Pi_{\mathbb{S}_0}\mathfrak{R}_{\varepsilon r} \Pi_{\mathbb{S}_0}^\bot - \mathscr{B}_\bot^{-1} \Pi_{\mathbb{S}_0}^\bot \mathcal{L}_{\varepsilon r}  \Pi_{\mathbb{S}_0} (\mathscr{B}-\textnormal{Id})\Pi_{\mathbb{S}_0}^\bot - \varepsilon \mathscr{B}_\bot^{-1} \Pi_{\mathbb{S}_0}^\bot \partial_{x} \mathcal{R} \mathscr{B}_\bot .
	\end{aligned}
\end{equation}
	The reversibility of $\mathscr{R}_0$ is a consequence of the reversibility of $\mathcal{L}_{\varepsilon r},\mathfrak{R}_{\varepsilon r}$, \eqref{sym:K} with Lemma \ref{lem:rever_integr}, the fact that $\mathscr{B}$ is a reversibility preserving operator because of \eqref{oddness_beta} and Remark \ref{rem:compo_rev}. For the Toeplitz-in-time property: the first term follows from the Toeplitz-in-time property of $\mathfrak{R}_{\varepsilon r}.$ Thanks to the projection $\Pi_{\mathbb{S}_0}$ on finite Fourier modes, the middle terms can be written as integral operators. The last term can be written as an integral operator since $\mathcal{R}$ is an integral operator. Then, the estimates \eqref{stima:R0}-\eqref{LipschitzR0}, follow from \eqref{stima:Rfrac}- \eqref{stima:Rfrac12}, \eqref{stima:K}-\eqref{stima:K12} with Lemma \ref{lem:rever_integr} and the various estimates on $\mathscr{B}$. For more details, we refer to we refer to \cite[Prop 6.3]{HR21}. Finally, the bound \eqref{stima:ttEn_Nor} follows from \eqref{def:ttEn_nor}, \eqref{stima:CVAR}, \eqref{stima:EnTR} and \eqref{stima:Bbot}. This concludes the proof of Proposition \ref{prop:normal_projection}.
\end{proof}
Finally, we make use of a KAM reducibility technique to get rid of the remainder $\mathscr{R}_0$ of Proposition \ref{prop:normal_projection}. This procedure allows to diagonalise the linearized operator in the normal directions, namely to conjugate it to a constant coefficients operator $\mathscr{L}_{\infty},$ up to fast decaying terms. The result reads as follows.
\begin{prop}\label{prop:KAM}
There exists $\sigma\triangleq\sigma(q,d,\tau_1)>0$ such that under the smallness condition \eqref{petitesse} with $\sigma\leqslant\sigma_{max}$, there exist
	\begin{enumerate}[label=\textbullet]
		\item an invertible linear operator $\Phi_{\infty}:\mathscr{O}\to \mathcal{L}\big(H^{s}\cap L_{\bot}^2\big)$ satisfying
		\begin{equation}\label{stima:Phi_infty}
			\forall s\in[s_{0},S],\quad \mbox{ }\|\Phi_{\infty}^{\pm 1}h\|_{q,s}^{\gamma ,\mathscr{O}}\lesssim \|h\|_{q,s}^{\gamma,\mathscr{O}}+\varepsilon\gamma^{-2}\| \mathfrak{I}_{0}\|_{q,s+\sigma}^{\gamma,\mathscr{O}}\|h\|_{q,s_{0}}^{\gamma,\mathscr{O}};
		\end{equation}
    \item a diagonal operator $\mathscr{L}_\infty=\mathscr{L}_{\infty}(b,\omega,i_{0})$ taking the form
		\begin{equation}\label{def:L_infty}\mathscr{L}_{\infty}=\omega\cdot\partial_{\varphi}\Pi_{\mathbb{S}_0}^{\bot}+\mathscr{D}_{\infty}
		\end{equation}
		where $\mathscr{D}_{\infty}=\mathscr{D}_{\infty}(b,\omega,i_{0})=\Pi_{\mathbb{S}_0}^\bot\mathscr{D}_\infty\Pi_{\mathbb{S}_{0}}^\bot$ is a reversible Fourier multiplier operator given by
		\begin{equation}\label{def:D_infty}
			\forall (l,j)\in \mathbb{Z}^{d}\times\mathbb{S}_{0}^{c},\quad  \mathscr{D}_{\infty}\mathbf{e}_{l,j}=\ii\,\mathtt{d}_{j}^{\infty}\,\mathbf{e}_{l,j},\qquad\mathtt{d}_{j}^{\infty}(b,\omega,i_{0})=\mathtt{d}_{j}^{0}(b,\omega,i_{0})+r_{j}^{\infty}(b,\omega,i_{0})
		\end{equation}
		with
        \begin{equation}\label{Lipschitz_dinfty}
			\forall j\in\mathbb{S}_{0}^{c},\quad\|\Delta_{12}\mathtt{d}_{j}^{\infty}\|_{q}^{\gamma,\mathscr{O}}\lesssim\varepsilon\gamma^{-1}|j|\| \Delta_{12}i\|_{q,\overline{s}_{h}+\sigma}^{\gamma,\mathscr{O}}
		\end{equation}
        and
		\begin{align}\label{stima restes vp}
			\sup_{j\in\mathbb{S}_{0}^{c}}|j|\| r_{j}^{\infty}\|_{q}^{\gamma ,\mathscr{O}}&\lesssim\varepsilon\gamma^{-1},\\
            \label{Lipschitz:rj_infty}
			\forall j\in\mathbb{S}_{0}^{c},\quad\|\Delta_{12}r_{j}^{\infty}\|_{q}^{\gamma,\mathscr{O}}&\lesssim\varepsilon\gamma^{-1}\|\Delta_{12}i\|_{q,\overline{s}_{h}+\sigma}^{\gamma,\mathscr{O}}
		\end{align}
        \end{enumerate}
		such that in the Cantor set
		\begin{align*}
			\mathscr{O}_{n}^{(\mathtt{KAM})}(\gamma,\tau_{1},\tau_2,i_{0})\triangleq \bigcap_{\underset{|l|\leqslant N_{n}}{(l,j,j_{0})\in\mathbb{Z}^{d}\times(\mathbb{S}_0^{c})^{2}}\atop(l,j)\neq(0,j_{0})}\Big\{&(b,\omega)\in\mathscr{O}_{n}^{(\mathtt{TR})}(\gamma,\tau_{1},i_{0}):\big|\omega\cdot l+\mathtt{d}_{j}^{\infty}(b,\omega,i_{0})-\mathtt{d}_{j_{0}}^{\infty}(b,\omega,i_{0})\big|>\tfrac{2\gamma \langle j-j_{0}\rangle}{\langle l\rangle^{\tau_{2}}}\Big\}
		\end{align*}
		we have 
		\begin{align*}
			\Phi_{\infty}^{-1}\mathscr{L}_{0}\Phi_{\infty}&=\mathscr{L}_{\infty}+\mathtt{E}_{n}^{(\mathtt{KAM})},
		\end{align*}
		where 
        \begin{enumerate}[label=\textbullet]
            \item with the notations of Proposition \ref{prop:normal_projection},
            \begin{equation}\label{def:initialKAM_op}
                \mathscr{L}_0\triangleq\omega\cdot\partial_{\varphi}\Pi_{\mathbb{S}_0}^{\bot}+\mathscr{D}_0+\mathscr{R}_0;
            \end{equation}
            \item the linear operator  $\mathtt{E}_{n}^{(\mathtt{KAM})}=\Pi_{\mathbb{S}_0}^{\bot}\mathtt{E}_{n}^{(\mathtt{KAM})}\Pi_{\mathbb{S}_0}^{\bot}$ satisfies the estimate
		\begin{equation}\label{stima:ttEn_KAM}
			\|\mathtt{E}_{n}^{(\mathtt{KAM})}h\|_{q,s_0}^{\gamma ,\mathscr{O}}\lesssim \varepsilon\gamma^{-2}N_{0}^{{\mu}_{2}}N_{n+1}^{-\mu} \|h\|_{q,s_0+1}^{\gamma,\mathscr{O}}.
		\end{equation}	
        \end{enumerate}
\end{prop}
\begin{proof}
$\blacktriangleright$ \textbf{Algebraic scheme :}
We design a KAM reducibility procedure as in \cite[Prop. 6.5]{HR21} allowing to get a sequence
$$\mathscr{L}_{m}=\omega\cdot\partial_{\varphi}\Pi_{\mathbb{S}_0}^{\bot}+\mathscr{D}_m+\mathscr{R}_m$$
where
\begin{enumerate}[label=\textbullet]
    \item $\mathscr{D}_m=\Pi_{\mathbb{S}_0}^{\bot}\mathscr{D}_m
\Pi_{\mathbb{S}_0}^{\bot}$ is a reversible Fourier multiplier given by
$$\mathscr{D}_m\mathbf{e}_{l,j}=\ii\mathtt{d}_{j}^m(b,\omega)\mathbf{e}_{l,j},\qquad\mathtt{d}_{-j}^m(b,\omega)=-\mathtt{d}_{j}^m(b,\omega);$$
    \item $\mathscr{R}_m=\Pi_{\mathbb{S}_0}^{\bot}\mathscr{R}_m\Pi_{\mathbb{S}_0}^{\bot}$ is real, reversible and Toeplitz-in-time.
\end{enumerate}
This sequence starts with $\mathscr{L}_0$ in \eqref{def:initialKAM_op} and the aim is to get in the Toeplits-in-time topology
$$\lim_{m\to\infty}\mathscr{R}_m=0.$$
To get the next step, we choose well a transformation  
$$\Phi_m=\Pi_{\mathbb{S}_0}^{\bot}+\Psi_m$$
such that
\begin{equation}\label{def:next_step}
    \mathscr{L}_{m+1}=\Phi_m^{-1}\mathscr{L}_m\Phi_m\omega\cdot\partial_\varphi\Pi_{\mathbb S_0}^{\bot}
+\mathscr D_{m+1}+\mathscr R_{m+1},
\end{equation}
with a remainder $\mathscr{R}_{m+1}$ "quadratically" smaller than $\mathscr{R}_m.$ A direct calculations shows that
$$\Phi_m^{-1}\mathscr{L}_m\Phi_m=
\omega\cdot\partial_{\varphi}\Pi_{\mathbb{S}_0}^{\bot}+\mathscr{D}_m+\Phi_m^{-1}
\Big([
\omega\cdot\partial_{\varphi}\Pi_{\mathbb{S}_0}^{\bot}+\mathscr{D}_m,\Psi_m]+P_{N_m}\mathscr{R}_m+P_{N_m}^{\bot}\mathscr{R}_m+\mathscr{R}_m\Psi_m\Big).$$
The natural choice of $\Psi_m$ is dictated by the homological equation
\begin{equation}\label{homological-KAM-equation}
\Big[\omega\cdot\partial_{\varphi}\Pi_{\mathbb{S}_0}^{\bot}+\mathscr{D}_m,\Psi_m
\Big]+P_{N_m}\mathscr{R}_m=\lfloor P_{N_m}\mathscr{R}_m\rfloor.
\end{equation}
Here $\lfloor\cdot\rfloor$ denotes the diagonal part. More precisely, if
$$A\mathbf{e}_{l_0,j_0}=\sum_{(l,j)}
A_{l_0,j_0}^{l,j}\mathbf{e}_{l,j}.$$
then
\begin{equation}\label{diagonal-projection-definition}
\lfloor A\rfloor\mathbf{e}_{l_0,j_0}
=A_{l_0,j_0}^{l_0,j_0}\mathbf{e}_{l_0,j_0},\qquad A_{l_0,j_0}^{l_0,j_0}\triangleq\big\langle
A\mathbf{e}_{l_0,j_0},
\mathbf{e}_{l_0,j_0}
\big\rangle_{L^2(\mathbb T^{d+1},\mathbb{C})}.
\end{equation}
The iteration will progressively absorb the diagonal part of the remainder into the normal frequencies and eliminate its off-diagonal component. More precisely, we get \eqref{def:next_step} with
\begin{equation}\label{KAM-remainder-update}
\mathscr D_{m+1}=\mathscr D_m+\lfloor P_{N_m}\mathscr R_m\rfloor,\qquad\mathscr R_{m+1}=\Phi_m^{-1}\Big(-\Psi_m\lfloor P_{N_m}\mathscr R_m\rfloor+P_{N_m}^{\bot}\mathscr R_m+\mathscr R_m\Psi_m\Big).
\end{equation}
The homological equation \eqref{homological-KAM-equation} can be solved modulo some non-resonance conditions. Indeed, observing that the real, reversibility, Toeplitz-in-time nature of $\mathscr{R}_{m}$ allows to write its coefficients as
\begin{equation}\label{KAM-remainder-coefficients}
(\mathscr{R}_m)_{l_0,j_0}^{l,j}
=\ii r_{j_0,m}^{j}(b,\omega,l_0-l)
\in\ii\mathbb R,\qquad(\mathscr{R}_m)_{-l_0,-j_0}^{-l,-j}=-(\mathscr{R}_m)_{l_0,j_0}^{l,j},
\end{equation}
the equation \eqref{homological-KAM-equation} imposes the choice, for $j,j_0\in\mathbb{S}_0^c$ and $|l|\leqslant N_n$
\begin{equation}\label{KAM-generator-definition}
(\Psi_m)_{j_0}^{j}(b,\omega,l)
=\begin{cases}
-\mathtt{c}_{j_0,m}^{j}(b,\omega,l)\,
r_{j_0,m}^{j}(b,\omega,l),&\textnormal{if }(l,j)\neq(0,j_0),\\
0,& \textnormal{if }(l,j)=(0,j_0),
\end{cases}
\end{equation}
and elsewhere
$$(\Psi_m)_{j_0}^{j}(b,\omega,l)=0,$$
where
\begin{equation}\label{KAM-small-divisor-cutoff}
\mathtt{c}_{j_0,m}^{j}(b,\omega,l)\triangleq \frac{\chi\left(
(\omega\cdot l+\mathtt{d}_{j}^m(b,\omega)-\mathtt{d}_{j_0}^m(b,\omega))(\gamma\langle j-j_0\rangle)^{-1}
\langle l\rangle^{\tau_2}
\right)}{\omega\cdot l+\mathtt{d}_{j}^m(b,\omega)-\mathtt{d}_{j_0}^m(b,\omega)},
\end{equation}
with $\chi\in C^\infty(\mathbb R,[0,1])$ an even cut-off function such that
\begin{equation}\label{KAM-cutoff-properties}
\chi(\xi)=0\quad\textnormal{for }|\xi|\leqslant\frac13,
\qquad\chi(\xi)=1\quad\textnormal{for }|\xi|\geqslant\frac12\cdot
\end{equation}
The cut-off is introduced such that the operator $\Psi_m$ is well-defined on the whole set $\mathscr{O}$ and therefore, the homological equation is solved within the non-resonance Cantor set (with the convention $\mathscr O_0^\gamma\triangleq\mathscr{O}$)
\begin{equation}\label{KAM-Cantor-step}
\mathscr O_{m+1}^\gamma\triangleq\bigcap_{\substack{
(l,j,j_0)\in
\mathbb Z^d\times(\mathbb S_0^c)^2\\
\substack{(l,j)\neq(0,j_0)}\\|l|\leqslant N_m}}
\left\{
(b,\omega)\in\mathscr O_m^\gamma\quad\textnormal{s.t.}\quad
\left|\omega\cdot l+\mathtt{d}_{j}^m(b,\omega)-\mathtt{d}_{j_0}^m(b,\omega)\right|>\frac{\gamma\langle j-j_0\rangle}
{\langle l\rangle^{\tau_2}}\right\}.
\end{equation}
In addition, this construction ensures that $\Psi_m$ is real and reversibility preserving with $\Psi_m=\Pi_{\mathbb S_0}^{\bot}\Psi_m
\Pi_{\mathbb S_0}^{\bot}.$ Note that \eqref{KAM-remainder-coefficients} implies that $\lfloor P_{N_m}\mathscr R_m\rfloor$ acts like
\begin{equation}\label{diagonal-frequency-identification}
\lfloor P_{N_m}\mathscr R_m\rfloor\mathbf{e}_{l,j}
=\ii r_j^m\mathbf{e}_{l,j},\qquad r_{j}^m\triangleq r_{j,m}^j.
\end{equation}
Therefore, in view of \eqref{KAM-remainder-update}
\begin{equation}\label{frequency-recursion}
\mathtt{d}_{j}^{m+1}
=\mathtt{d}_{j}^m+r_j^m.
\end{equation}

\medskip

$\blacktriangleright$ \textbf{Topological considerations and convergence :}
Introduce
$$\delta_m(s)\triangleq\gamma^{-1}\|\mathscr{R}_0\|_{\textnormal{\tiny{O-d}},q,s}^{\gamma,\mathscr{O}},\qquad\widehat\delta_m(s)\triangleq\max\left\{\gamma^{-1}\|\partial_\theta\mathscr R_m\|_{\textnormal{\tiny{O-d}},q,s}^{\gamma,\mathscr{O}},\delta_m(s)\right\}.$$
The estimate \eqref{stima:R0} and \eqref{petitesse} imply 
\begin{equation}\label{initial-smallness}
N_0^{\mu}\delta_0(s_h)\leqslant
CN_0^{\mu}\varepsilon\gamma^{-2}\leqslant
C\varepsilon_0.
\end{equation}
By the linear growth in $j$ of $\mathtt{d}_j^0$ in \eqref{initial diago} and the a priori convergence of the scheme, one has
\begin{equation}\label{frequency-gap-control}
\max_{|\alpha|\leqslant q}
\sup_{(b,\omega)\in\mathscr{O}}\left|\partial_{b,\omega}^{\alpha}\big(\mathtt{d}_{j}^m(b,\omega)-\mathtt{d}_{j_0}^m(b,\omega)
\big)\right|\leqslant C|j-j_0|.
\end{equation}
This together with \eqref{KAM-generator-definition} and Lemma \ref{lem:funct_topo}-$(v)$ implies
\begin{equation}\label{generator-estimate}
\|\Psi_m\|_{\textnormal{\tiny{O-d}},q,s}^{\gamma,\mathscr{O}}\leqslant
C\gamma^{-1}\|P_{N_m}\mathscr R_m\|_{\textnormal{\tiny{O-d}},q,s+\tau_2q+\tau_2}^{\gamma,\mathscr{O}}.
\end{equation}
Combining \eqref{KAM-remainder-update}, \eqref{generator-estimate} and Lemma \ref{lem:Toeplitz_topo} gives the following recursive estimates, whenever
$s_0\leqslant s\leqslant\bar s\leqslant S$,
\begin{align}
\delta_{m+1}(s)
&\leqslant
N_m^{s-\bar s}\delta_m(\bar s)+CN_m^{\tau_2q+\tau_2}\delta_m(s_0)\delta_m(s),
\label{remainder-recursion-basic}\\
\widehat\delta_{m+1}(s)&\leqslant
N_m^{s-\bar s}\widehat\delta_m(\bar s)+CN_m^{\tau_2q+\tau_2+1}\widehat\delta_m(s_0)\widehat\delta_m(s).
\label{remainder-recursion-derivative}
\end{align}
With the choices of $\bar s_l$ and $s_h$ fixed by \eqref{prmtrrr} and the initial smallness \eqref{initial-smallness}, the inequalities \eqref{remainder-recursion-basic} and \eqref{remainder-recursion-derivative} yield the inductive estimates
\begin{equation}\label{remainder-induction}
\delta_m(\bar s_l)\leqslant\delta_0(s_h)N_0^{\mu}N_m^{-\mu},\qquad\delta_m(s_h)\leqslant\left(2-\frac1{m+1}\right)\delta_0(s_h),
\end{equation}
and
\begin{equation}\label{derivative-remainder-induction}
\widehat\delta_m(s_0)
\leqslant\widehat\delta_0(s_h)
N_0^{\mu}N_m^{-\mu},\qquad\widehat\delta_m(s_h)
\leqslant\left(2-\frac1{m+1}\right)\widehat\delta_0(s_h).
\end{equation}
We next control the variation of the normal frequencies. From the real-Toeplitz representation and \eqref{frequency-recursion},
\begin{align*}
\mathtt{d}_{j}^{m+1}-\mathtt{d}_{j}^m
&=
r_j^m\\
&=
-\ii
\left\langle
P_{N_m}\mathscr R_m\mathbf{e}_{l,j},
\mathbf{e}_{l,j}
\right\rangle_{L^2(\mathbb T^{d+1},\mathbb{C})}\\
&=
-\ii
\left\langle
P_{N_m}\mathscr R_m\mathbf{e}_{0,j},
\mathbf{e}_{0,j}
\right\rangle_{L^2(\mathbb T^{d+1},\mathbb{C})}.
\end{align*}
A duality estimate, combined with Lemma \ref{lem:Toeplitz_topo}, \eqref{remainder-induction},  and \eqref{initial-smallness}, gives
\begin{equation}\label{frequency-increment-bound}
\|\mathtt{d}_{j}^{m+1}-\mathtt{d}_{j}^m\|_q^{\gamma,\mathscr{O}}\leqslant
C\varepsilon\gamma^{-1}N_0^{\mu}N_m^{-\mu}.
\end{equation}
In particular, the series of frequency corrections converges. Moreover, \eqref{frequency-gap-control} propagates through the iteration, since
\begin{align*}
\max_{|\alpha|\leqslant q}
\sup_{(b,\omega)\in\mathscr{O}}
\left|\partial_{b,\omega}^{\alpha}
\big(\mathtt{d}_{j}^{m+1}-\mathtt{d}_{j_0}^{m+1}\big)
\right|\leqslant
C\left(1+\varepsilon\gamma^{-1-q}
N_0^{\mu}N_m^{-\mu}\right)|j-j_0|.
\end{align*}
Hence $(\mathtt{d}_{j}^m)_{m\geq0}$ converges in
$W^{q,\infty,\gamma}(\mathscr{O},\mathbb C)$. Denote its limit by $\mathtt{d}_{j}^\infty$ and set
$$\mathscr{D}_\infty\triangleq\big(\ii\mathtt{d}_{j}^\infty\big)_{j\in\mathbb S_0^c},\qquad\textnormal{i.e.}\qquad\mathscr{D}_\infty\mathbf{e}_{l,j}
=\ii\mathtt{d}_{j}^\infty\mathbf{e}_{l,j}.$$
Furthermore, \eqref{frequency-increment-bound} gives
\begin{equation}\label{frequency-tail-estimate}
\|\mathtt{d}_{j}^\infty-\mathtt{d}_{j}^m\|_q^{\gamma,\mathscr{O}}\leqslant
C\gamma^{-1}\delta_0(s_h)N_0^{\mu}N_m^{-\mu}.
\end{equation}
Consequently,
\begin{equation}\label{limiting-frequency-decomposition}
\mathtt{d}_{j}^\infty
=\mathtt{d}_{j}^0+\sum_{m=0}^{\infty}
(\mathtt{d}_{j}^{m+1}-\mathtt{d}_{j}^m)
=\mathtt{d}_{j}^0+r_j^\infty.
\end{equation}
By the definition of the diagonal operator norm and \eqref{frequency-tail-estimate},
\begin{equation}\label{diagonal-convergence-rate}
    \|\mathscr D_m-\mathscr D_\infty\|_{\textnormal{\tiny{O-d}},q,s_0}^{\gamma,\mathscr{O}}=\sup_{j\in\mathbb S_0^c}\|\mathtt{d}_{j}^m-\mathtt{d}_{j}^\infty\|_q^{\gamma,\mathscr{O}}\leqslant
C\gamma^{-1}\delta_0(s_h)
N_0^{\mu}N_m^{-\mu}.
\end{equation}
It remains to construct the limiting conjugation. Define
\begin{equation}\label{composition-of-conjugations}
\widehat\Phi_0\triangleq \Phi_0,
\qquad
\widehat\Phi_m
\triangleq 
\Phi_0\circ\Phi_1\circ\cdots\circ\Phi_m,
\qquad m\geq1.
\end{equation}
Since $\Phi_m=\mathrm{Id}+\Psi_m$, one has
\begin{equation}\label{composition-recursion}
\widehat\Phi_{m+1}
=\widehat\Phi_m+\widehat\Phi_m\Psi_{m+1}.
\end{equation}
The summability resulting from \eqref{generator-estimate}, together with the smallness assumptions of the KAM construction, shows that $(\widehat\Phi_m)_m$ converges in the Toeplitz operator topology to an invertible operator $\Phi_\infty$. As in \cite{HR21}, one also obtains
\begin{equation}\label{conjugation-convergence}
\|\widehat\Phi_m^{-1}-\Phi_\infty^{-1}\|_{\textnormal{\tiny{O-d}},q,s_0}^{\gamma,\mathscr{O}}+\|\widehat\Phi_m-\Phi_\infty\|_{\textnormal{\tiny{O-d}},q,s_0}^{\gamma,\mathscr{O}}
\leqslant
C\delta_0(s_h)N_0^{\mu}N_{m+1}^{-\mu}.
\end{equation}
We now pass to the limiting operator. Put
$$\mathscr L_\infty\triangleq \omega\cdot\partial_\varphi\Pi_{\mathbb S_0}^{\bot}+\mathscr D_\infty.$$
Since
$$\mathscr L_m-\mathscr L_\infty=\mathscr D_m-\mathscr D_\infty+\mathscr R_m,$$
the estimates \eqref{diagonal-convergence-rate} and \eqref{remainder-induction} imply the following convergence in the Toeplitz operator topology
$$\mathscr L_m\underset{m\to\infty}{\longrightarrow}\mathscr L_\infty.$$
For a fixed $n\in\mathbb N$, by means of triangular inequality, the limiting non-resonance set defined in Proposition \ref{prop:KAM} satisfies
\begin{equation}\label{limiting-Cantor-inclusion}
\mathscr O_{n}^{(\mathtt{KAM})}(\gamma,\tau_1,\tau_2,i_0)\subset\mathscr O_{n+1}^{\gamma}.
\end{equation}
From \eqref{composition-of-conjugations} and the finite-step identities, for
$(b,\omega)\in\mathscr O_{n+1}^{\gamma}$ we have
\begin{align*}
\widehat\Phi_n^{-1}\mathscr L_0\widehat\Phi_n
&=
\omega\cdot\partial_\varphi\Pi_{\mathbb S_0}^{\bot}
+\mathscr D_{n+1}+\mathscr R_{n+1}=
\mathscr L_\infty
+
(\mathscr D_{n+1}-\mathscr D_\infty)
+
\mathscr R_{n+1}.
\end{align*}
Passing from $\widehat\Phi_n$ to $\Phi_\infty$ gives
$$\Phi_\infty^{-1}\mathscr L_0\Phi_\infty=\mathscr L_\infty+\mathtt E_n^{(\mathtt{KAM})},$$
where we denote the whole error term by
\begin{equation}\label{final-conjugated-error}
\mathtt E_n^{(\mathtt{KAM})}
\triangleq\mathscr D_{n+1}-\mathscr D_\infty+\mathscr R_{n+1}+\Phi_\infty^{-1}\mathscr L_0
(\Phi_\infty-\widehat\Phi_n)+(\Phi_\infty^{-1}-\widehat\Phi_n^{-1})
\mathscr L_0\widehat\Phi_n.
\end{equation}
The estimates in Lemma \ref{lem:Toeplitz_topo}, together with \eqref{diagonal-convergence-rate}, \eqref{initial-smallness}, and \eqref{conjugation-convergence}, give precisely \eqref{stima:ttEn_KAM}.
We finally establish the estimate for $r_j^\infty$. By definition,
$$r_j^\infty=\sum_{m=0}^{\infty}r_j^m,
\qquad
r_j^m=-\ii\left\langle
P_{N_m}\mathscr R_m\mathbf{e}_{0,j},
\mathbf{e}_{0,j}\right\rangle_{L^2(\mathbb T^{d+1})}.$$
Since $\partial_\theta\mathbf{e}_{0,j}=\ii j\,\mathbf{e}_{0,j},$ an integration by parts gives
\begin{align*}
\left\langle
P_{N_m}\mathscr R_m\mathbf{e}_{0,j},
\mathbf{e}_{0,j}
\right\rangle_{L^2(\mathbb T^{d+1})}
&=
\frac{\ii}{j}
\left\langle
P_{N_m}\mathscr R_m\mathbf{e}_{0,j},
\partial_\theta\mathbf{e}_{0,j}
\right\rangle_{L^2(\mathbb T^{d+1},\mathbb{C})}=-\frac{\ii}{j}
\left\langle
P_{N_m}\partial_\theta\mathscr R_m\mathbf{e}_{0,j},\mathbf{e}_{0,j}
\right\rangle_{L^2(\mathbb T^{d+1},\mathbb{C})}.
\end{align*}
The duality estimate between $H^{s_0}$ and $H^{-s_0}$, together with Lemma \ref{lem:Toeplitz_topo}, \eqref{derivative-remainder-induction}, \eqref{stima:R0}, and \eqref{initial-smallness}, therefore yields
\begin{equation}\label{limiting-frequency-correction}
\|r_j^\infty\|_q^{\gamma,\mathscr{O}}
\lesssim
|j|^{-1}\varepsilon\gamma^{-1}.
\end{equation}
Finally, the estimates \eqref{Lipschitz:rj_infty} and \eqref{Lipschitz_dinfty} follow in a similar fashion. For more details, we refer to \cite{HR21}. This completes the proof of Proposition \ref{prop:KAM}.
\end{proof}

With this in hand, we can obtain an almost approximate right inverse for $\mathcal{L}_{\bot}$.

\begin{prop}\label{prop:inverse_normal}
	There exists $\sigma\triangleq \sigma(\tau_1,\tau_2,q,d)$ such that under the smallness condition \eqref{petitesse} with $\sigma\leqslant\sigma_{max},$ the following assertions hold.
	\begin{enumerate}[label=(\roman*)]
		\item There exists a family of  linear operators $\big(\mathtt{T}_n\big)_{n\in\mathbb{N}}$  defined in $\mathscr{O}$ satisfying
		\begin{equation}\label{stima:ttTn}
		    \forall s\in[s_0,S],\quad \sup_{n\in\mathbb{N}}\|\mathtt{T}_{n}h\|_{q,s}^{\gamma ,\mathscr{O}}\lesssim \gamma ^{-1}\|h\|_{q,s+\tau_{1}q+\tau_{1}}^{\gamma ,\mathscr{O}}
		\end{equation}
		and such that, for any $n\in\mathbb{N}$, in restriction to the Cantor set
		$$\mathscr{O}^{(\mathtt{Inv})}_n(\gamma,\tau_{1},i_{0})\triangleq\bigcap_{(l,j)\in\mathbb{Z}^{d }\times\mathbb{S}_{0}^{c}\atop |l|\leqslant N_{n}}\left\lbrace(b,\omega)\in\mathscr{O}\quad\textnormal{s.t.}\quad\left|\omega\cdot l+\mathtt{d}_{j}^{\infty}(b,\omega,i_{0})\right|>\tfrac{\gamma \langle j\rangle }{\langle l\rangle^{\tau_{1}}}\right\rbrace,$$
		the following identity holds
		$$\mathscr{L}_{\infty}\mathtt{T}_n=\textnormal{Id}+\mathtt{E}_{n}^{(\mathtt{Inv})},$$
		where
        \begin{enumerate}[label=\textbullet]
            \item $\mathscr{L}_{\infty}$ and $\mathtt{d}_j^{\infty}$ are defined in \eqref{def:L_infty} and \eqref{def:D_infty}, respectively;
            \item The linear operator $\mathtt{E}_{n}^{(\mathtt{Inv})}$ satisfies
            \begin{equation}\label{stima:ttEInv}
                \forall s_{0}\leqslant s\leqslant\overline{s}\leqslant S, \quad \|\mathtt{E}_{n}^{(\mathtt{Inv})}h\|_{q,s}^{\gamma ,\mathscr{O}} \lesssim 
		N_n^{s-\overline{s}}\gamma^{-1}\|h\|_{q,\overline{s}+1+\tau_{1}q+\tau_{1}}^{\gamma ,\mathscr{O}}.
            \end{equation}
        \end{enumerate}
		
		\item 
		There exists a sequence of linear operators $\big(\mathtt{I}_{\bot,n}\big)_{n\in\mathbb{N}}$ enjoying, for any $s\in[s_0,S],$
		\begin{equation}\label{stima-inverse_nor}
			\sup_{n\in\mathbb{N}}\|\mathtt{I}_{\bot,n}h\|_{q,s}^{\gamma ,\mathscr{O}}\lesssim\gamma^{-1}\left(\|h\|_{q,s+\sigma}^{\gamma ,\mathscr{O}}+\| \mathfrak{I}_{0}\|_{q,s+\sigma}^{\gamma ,\mathscr{O}}\|h\|_{q,s_{0}+\sigma}^{\gamma,\mathscr{O}}\right)
		\end{equation}
		such that in restriction to the Cantor set
		\begin{equation}\label{def:Cant-red}
			\mathcal{C}_n(\gamma,\tau_{1},\tau_{2},i_{0})\triangleq \mathscr{O}_{n}^{(\mathtt{TR})}(\gamma,\tau_1,i_{0})\cap\mathscr{O}_{n}^{(\mathtt{KAM})}(\gamma,\tau_{1},\tau_{2},i_{0})\cap\mathscr{O}_{n}^{(\mathtt{Inv})}(\gamma,\tau_{1},i_{0}),
		\end{equation}
		the following identity holds
		\begin{equation}\label{Inv:Lbot}
		    \mathcal{L}_{\bot}\mathtt{I}_{\bot,n}=\textnormal{Id}+\mathtt{E}_n,
		\end{equation}
		where $\mathtt{E}_n$ satisfies for any $s\in [s_0,S],$
		\begin{equation}\label{stima:ttEn}
			\|\mathtt{E}_nh\|_{q,s_0}^{\gamma ,\mathscr{O}}\lesssim N_n^{s_0-s}\gamma^{-1}\Big( \|h\|_{q,s+\sigma}^{\gamma,\mathscr{O}}+\varepsilon\gamma^{-2}\| \mathfrak{I}_{0}\|_{q,s+\sigma}^{\gamma,\mathscr{O}}\|h\|_{q,s_{0}}^{\gamma,\mathscr{O}} \Big)+ \varepsilon\gamma^{-3}N_{0}^{{\mu}_{2}}N_{n+1}^{-\mu} \|h\|_{q,s_0+\sigma}^{\gamma,\mathscr{O}}.
		\end{equation}
		\item In restriction to the Cantor set $\mathcal{C}_{n}(\gamma,\tau_{1},\tau_{2},i_{0})$, the following decomposition holds
\begin{equation}\label{decomposition:Lbot}
    \mathcal{L}_{\bot}=\mathcal{L}_{\bot,n}+\mathcal{R}_{\bot,n},
\end{equation}
		where the operators $\mathcal{L}_{\bot,n}$ and $\mathcal{R}_{\bot,n}$ are defined in  $\mathscr{O}$ and satisfy for all $s\in[s_{0},S],$
		\begin{align}
        \sup_{n\in\mathbb{N}}\|\mathcal{L}_{\bot,n}h\|_{q,s}^{\gamma,\mathscr{O}}&\lesssim\|h\|_{q,s+1}^{\gamma,\mathscr{O}}+\varepsilon\gamma^{-2}\|\mathfrak{I}_{0}\|_{q,s+\sigma}^{\gamma,\mathscr{O}}\|h\|_{q,s_{0}+1}^{\gamma,\mathscr{O}},\label{stima-tronc:lin_nor}\\
\|\mathcal{R}_{\bot,n}h\|_{q,s_{0}}^{\gamma,\mathscr{O}}&\lesssim N_{n}^{s_{0}-s}\gamma^{-1}\left(\|h\|_{q,s+\sigma}^{\gamma,\mathscr{O}}+\varepsilon\gamma^{-2}\|\mathfrak{I}_{0}\|_{q,s+\sigma}^{\gamma,\mathscr{O}}\|h\|_{q,s_{0}+\sigma}^{\gamma,\mathscr{O}}\right)+\varepsilon\gamma^{-3}N_{0}^{\mu}N_{n+1}^{-\mu}\|h\|_{q,s_{0}+\sigma}^{\gamma,\mathscr{O}}.\label{stima-restn}
		\end{align}
	\end{enumerate}
\end{prop}

\begin{proof}
$(i)$ We first construct an approximate inverse for $\mathscr L_\infty$ in \eqref{def:L_infty}. Using the decomposition associated with the projectors $\Pi_{N_n}$ and $\Pi_{N_n}^\bot$, we write
$$\mathscr L_\infty=\mathscr{L}_{\infty,n}+\mathscr{R}_{\infty,n},\qquad\mathscr{R}_{\infty,n}\triangleq -\Pi_{N_n}^\bot\omega\cdot\partial_\varphi\Pi_{N_n}^\bot\Pi_{\mathbb S_0}^{\bot},$$
where $\mathscr{L}_{\infty,n}$ is diagonal and, on the Fourier basis $\mathbf{e}_{l,j}$,
$$\mathscr{L}_{\infty,n}\mathbf{e}_{l,j}=\begin{cases}
\ii(\omega\cdot l+\mathtt{d}_j^\infty)\mathbf{e}_{l,j},& |l|\leqslant N_n,\\
\ii\mathtt{d}_j^\infty\mathbf{e}_{l,j},& |l|>N_n.
\end{cases}$$
We now define $\mathtt T_n$ by inverting these diagonal coefficients, with the usual cut-off on the low Fourier modes
$$\mathtt T_nh
\triangleq 
-\ii\sum_{\substack{(l,j)\in\mathbb Z^d\times\mathbb S_0^c\\ |l|\leqslant N_n}}
\frac{\chi\!\left((\omega\cdot l+\mathtt{d}_j^\infty)
\gamma^{-1}\langle l\rangle^{\tau_1}\right)}
{\omega\cdot l+\mathtt{d}_j^\infty}
h_{l,j}\mathbf{e}_{l,j}
-\ii\sum_{\substack{(l,j)\in\mathbb Z^d\times\mathbb S_0^c\\ |l|>N_n}}
\frac{h_{l,j}}
{\mathtt{d}_j^\infty}\mathbf{e}_{l,j}.$$
Since, according to \eqref{KAM-cutoff-properties}, the cut-off equals one on $\mathscr{O}^{(\mathtt{Inv})}_n(\gamma,\tau_{1},i_{0})$, the diagonal definition immediately gives
$$\mathscr{L}_{\infty,n}\mathtt T_n=\mathrm{Id},
\qquad\textnormal{on }\mathscr{O}^{(\mathtt{Inv})}_n(\gamma,\tau_{1},i_{0}).$$
Consequently,
\begin{equation}\label{LinftyTn-new}
\mathscr L_\infty\mathtt T_n
=\mathrm{Id}+\mathtt E_n^{(\mathtt{Inv})},\qquad \mathtt{E}_n^{(\mathtt{Inv})}\triangleq\mathscr{R}_{\infty,n}\mathtt T_n.
\end{equation}
We now turn to the estimates. It is readily seen from the expression of $\mathtt{d}_{j}^{\infty}$ in \eqref{def:D_infty} that
$$|j|\lesssim\|\mathtt{d}_j^\infty\|_{q}^{\gamma,\mathcal O}\lesssim |j|.$$
Consequently, using Lemma \ref{lem:funct_topo}-$(vi)$, we get \eqref{stima:ttTn}
The estimate \eqref{stima:ttEInv} follows from \eqref{stima:ttTn} and Lemma \ref{lem:funct_topo}-$(ii)$.\\
$(ii)$ We transport the previous approximate inverse back to the original normal variables by defining
$$\mathtt{I}_{\bot,n}
\triangleq\mathscr{B}_{\bot}\Phi_{\infty}\mathtt{T}_n\Phi_{\infty}^{-1}\mathscr{B}_{\bot}^{-1}.$$
This operator is well-defined for all $(b,\omega)\in\mathscr{O}$ and the bound \eqref{stima-inverse_nor} follows from \eqref{stima:Bbot}, \eqref{stima:Phi_infty} and \eqref{stima:ttTn}.
On $\mathscr{O}_{n}^{(\mathtt{TR})}(\gamma,\tau_1,i_0)
\cap\mathscr{O}_{n}^{(\mathtt{KAM})}(\gamma,\tau_1,\tau_2,i_0),$ Propositions \ref{prop:normal_projection} and
\ref{prop:KAM} imply
$$\Phi_\infty^{-1}\mathscr B_\bot^{-1}
\mathcal{L}_{\bot}
\mathscr B_\bot\Phi_\infty=\mathscr L_\infty+\mathtt E_n^{(\mathtt{KAM})}
+\Phi_\infty^{-1}\mathtt \mathtt{E}_{n}^{(\mathtt{Nor})}\Phi_\infty.$$
Multiplying by $\mathtt T_n$ on the right and using \eqref{LinftyTn-new}, we find, on the additional set
$\mathscr{O}_{n}^{(\mathtt{Inv})}(\gamma,\tau_1,i_0)$,
$$\Phi_\infty^{-1}\mathscr B_\bot^{-1}
\mathcal{L}_{\bot}
\mathscr B_\bot\Phi_\infty\mathtt T_n=\mathrm{Id}
+\mathtt E_n^{(\mathtt{Inv})}
+\mathtt E_n^{(\mathtt{KAM})}\mathtt T_n
+\Phi_\infty^{-1}\mathtt \mathtt{E}_{n}^{(\mathtt{Nor})}\Phi_\infty\mathtt T_n.$$
Conjugating back with $\mathscr B_\bot\Phi_\infty$ therefore gives, on
$\mathcal{C}_n(\gamma,\tau_1,\tau_2,i_0)$, the identity \eqref{Inv:Lbot}, with
$$\mathtt E_n=\mathscr B_\bot\Phi_\infty
\Big(
\mathtt E_n^{(\mathtt{Inv})}+\mathtt E_n^{(\mathtt{KAM})}\mathtt T_n
+\Phi_\infty^{-1}\mathtt \mathtt{E}_{n}^{(\mathtt{Nor})}\Phi_\infty\mathtt T_n\Big)\Phi_\infty^{-1}\mathscr B_\bot^{-1}.$$
The bound \eqref{stima:ttEn} follows from \eqref{stima:Bbot}, \eqref{stima:Phi_infty}, \eqref{stima:ttEn_Nor}, \eqref{stima:ttEn_KAM}, \eqref{stima:ttTn}, \eqref{stima:ttEInv} and \eqref{petitesse}, after increasing $\sigma$ if necessary.\\
$(iii)$ The decomposition \eqref{decomposition:Lbot} follows immediately from \eqref{Inv:Lbot} by setting
$$\mathcal{L}_{\bot,n}\triangleq\mathtt{I}_{\bot,n}^{-1}=\mathscr{B}_{\bot}\Phi_{\infty}
\mathscr{L}_{\infty,n}\Phi_{\infty}^{-1}\mathscr{B}_{\bot}^{-1},\qquad\mathcal{R}_{\bot,n}\triangleq\mathtt{E}_n\mathtt{I}_{\bot,n}^{-1}.$$
The bound \eqref{stima-tronc:lin_nor} follows from \eqref{stima:Bbot}, \eqref{stima:Phi_infty}, Lemma \ref{lem:funct_topo}-$(ii)$ and \eqref{petitesse} while the bound \eqref{stima-restn} follows from \eqref{stima:ttEn}, \eqref{stima-tronc:lin_nor} and \eqref{petitesse}. The proof of Proposition
\ref{prop:inverse_normal}
is complete.

\end{proof}
\subsubsection{Construction and tame estimates for the approximate inverse}
Now, we can finish the construction of the approximate right inverse to the full linearized operator $d_{(i,\alpha)}\mathscr{F}(i_0,\alpha_0)$. The result reads as follows.

\begin{prop}\label{prop:approx_full} 
	There exists $\overline{\sigma}= \overline{\sigma}(\tau_1,\tau_2,d,q)>0$ and a family of reversible operators $\mathtt{I}_{n}$ satisfying for any $n\in\mathbb{N}^*$ and $s\in [s_0,S],$
	\begin{equation}\label{stima:ttIn}
		\|\mathtt{I}_n\hat{\imath}\|_{q,s}^{\gamma,\mathscr{O}}\lesssim\gamma^{-1}\left(\|\hat{\imath}\|_{q,s+\overline{\sigma}}^{\gamma,\mathscr{O}}+\|\mathfrak{I}_{0}\|_{q,s+\overline{\sigma}}^{\gamma,\mathscr{O}}\|\hat{\imath}\|_{q,s_{0}+\overline{\sigma}}^{\gamma,\mathscr{O}}\right).
	\end{equation}
	Moreover, for any $n\in\mathbb{N}^*$, in restriction to the Cantor set $\mathcal{C}_{n}(\gamma,\tau_{1},\tau_{2},i_{0})$ defined by \eqref{def:Cant-red}, the operator $\mathtt{I}_n$ is an almost-approximate right inverse of $d_{(i,\alpha)}\mathscr{F}(i_0,\alpha_0)$, namely  
	\begin{equation}\label{splitting per approximate inverse}
		d_{(i,\alpha)}\mathscr{F} (i_0,\alpha_0)\circ\mathtt{I}_n-\textnormal{Id}=\mathscr{E}_{1,n}+\mathscr{E}_{2,n}+\mathscr{E}_{3,n}, 
	\end{equation}
	where $\mathscr{E}_{1,n}$, $\mathscr{E}_{2,n}$ and $\mathscr{E}_{3,n}$ are operators well-defined in the whole set $\mathscr{O}$ and such that
	\begin{align}	\|\mathscr{E}_{1,n}\hat{\imath}\|_{q,s_0}^{\gamma,\mathscr{O}}&\lesssim\gamma^{-1}\|\mathscr{F}(i_0,\alpha_0)\|_{q,s_0 +\overline{\sigma}}^{\gamma,\mathscr{O}}\|\hat{\imath}\|_{q,s_0+\overline{\sigma}}^{\gamma,\mathscr{O}},\label{stima:scrE1}\\
		\forall\mathtt{t}\geqslant 0,\quad\|\mathscr{E}_{2,n}\hat{\imath}\|_{q,s_0}^{\gamma,\mathscr{O}}&\lesssim\gamma^{-1} N_n^{-\mathtt{t}}\Big(\|\hat{\imath}\|_{q,s_0+\mathtt{t}+\overline{\sigma}}^{\gamma,\mathscr{O}}+\varepsilon\|\mathfrak{I}_{0} \|_{q,s_0+\mathtt{t}+\overline{\sigma}}^{\gamma,\mathscr{O}}\|\hat{\imath}\|_{q,s_0+\overline{\sigma}}^{\gamma,\mathscr{O}} \Big),
		\label{stima:scrE2}\\
		\forall \mathtt{t}\in [0,S-s_0],\quad\|\mathscr{E}_{3,n} \hat{\imath}\|_{q,s_0}^{\gamma,\mathscr{O}}&\lesssim N_n^{-\mathtt{t}}\gamma^{-2}\Big( \|\hat{\imath}\|_{q,s_0+\mathtt{t}+\overline{\sigma}}^{\gamma,\mathscr{O}}+\varepsilon\gamma^{-2}\|\mathfrak{I}_{0}\|_{q,s_0+\mathtt{t}+\overline{\sigma}}^{\gamma,\mathscr{O}}\|\hat{\imath}\|_{q,s_0+\overline{\sigma}}^{\gamma,\mathscr{O}}\Big)\nonumber\\
		&\quad+ \varepsilon\gamma^{-4}N_{0}^{\mu}N_{n}^{-\mu} \|\hat{i}\|_{q,s_0+\overline{\sigma}}^{\gamma,\mathscr{O}}. \label{stima:scrE3} 
	\end{align}
\end{prop}
\begin{proof}
    Let us denote
    $$\mathbb{D}_n\triangleq\omega\cdot\partial_{\varphi}\Pi_{N_n}+\Pi_{N_n,\mathtt{g}}^{\bot},\qquad\mathbb{D}_n^{\bot}\triangleq\omega\cdot\partial_{\varphi}\Pi_{N_n}^{\bot}-\Pi_{N_n,\mathtt{g}}^{\bot},$$
    where the projectors $\Pi_{N_n},\Pi_{N_n}^{\bot}$ are defined through \eqref{projectors_functions} and the projector $\Pi_{N_n,\mathtt{g}}^{\bot},$ introduced to preserve the reversibility on the high Fourier modes, is given by
    $$\Pi_{N_n,\mathtt{g}}^{\bot}u(\varphi)\triangleq\sum_{l\in\mathbb{Z}^d\atop|l|>N_n}\mathtt{g}(l)u_le^{\ii l\cdot\varphi},\qquad\mathtt{g}(-l)=-\mathtt{g}(l).$$
    Accordingly, the conjugation \eqref{conjugaison:Berti-Bolle} becomes
    $$DG_0(i_{\textnormal{\tiny{triv}}})^{-1}d_{(i,\alpha)}\mathscr{F}(i_0,\alpha_0)D\bar{G}_0(i_{\textnormal{\tiny{triv}}})=\mathcal{T}_n+\mathcal{E}_{1,n}+\mathcal{E}_{2,n}+\mathcal{E}_{3,n},$$
    with
    \begin{align*}
        \mathcal{T}_n[\hat{\phi},\hat{\mathtt{y}},\hat{\mathtt{w}},\hat{\alpha}]&\triangleq\begin{bmatrix}
        \mathbb{D}_n\hat{\phi}-\big[H_{20}(\varphi)\hat{\mathtt{y}}+H_{11}^{\top}(\varphi)\hat{\mathtt{w}}+\mathtt{L}_1^{\top}(\varphi)\hat{\alpha}\big]\\
        \mathbb{D}_n\mathtt{y}+\mathtt{P}(\varphi)\hat{\alpha}\\
        \mathcal{L}_{\bot,n}\hat{\mathtt{w}}-\partial_{\theta}\big[H_{11}(\varphi)\hat{\mathtt{y}}-\mathtt{L}_{2}^{\top}\hat{\alpha}\big]
    \end{bmatrix}
    \end{align*}
    and
    \begin{align*}
        \mathcal{E}_{1,n}&\triangleq DG_0(i_{\textnormal{\tiny{triv}}})^{-1}\partial_{\varphi}\mathtt{Z}(\varphi)\hat{\phi}\\
    &\quad+\begin{bmatrix}
0\\
\Pi_{N_n}\mathtt{Q}(\varphi)\big[\big(\partial_{\theta}( \partial_{z}\nabla_z H_\varepsilon^{\alpha_0}) \big(i_0(\varphi)\big)+\mathcal{R}(\varphi)\big)\hat{\mathtt{y}}+H_{11}^{\top}(\varphi)\hat{\mathtt{w}}\big]-R_{10}(\varphi)\hat{\mathtt{y}}-R_{01}(\varphi)\hat{\mathtt{w}}\\
0
\end{bmatrix},\\
        \mathcal{E}_{2,n}&\triangleq\begin{bmatrix}
            \mathbb{D}_{n}^{\bot}\hat{\phi}\\
            \mathbb{D}_n^{\bot}\hat{\mathtt{y}}+\Pi_{n_N}^{\bot}\mathtt{Q}(\varphi)\big[\big(\partial_{\theta}( \partial_{z}\nabla_z H_\varepsilon^{\alpha_0}) \big(i_0(\varphi)\big)+\mathcal{R}(\varphi)\big)\hat{\mathtt{y}}+H_{11}^{\top}(\varphi)\hat{\mathtt{w}}\big]\\
            0
        \end{bmatrix},\\
        \mathcal{E}_{3,n}&\triangleq\begin{bmatrix}
            0\\
            0\\
            \mathcal{R}_{\bot,n}
        \end{bmatrix}.
    \end{align*}
    The operator $\mathbb{D}_n$ is well-defined on the whole parameter set $\mathscr{O}$ but invertible on the Diophantine set
    $$\mathtt{DC}_{N_n}(\gamma,\tau_1)\triangleq\bigcap_{\substack{l\in\mathbb{Z}^d\setminus\{0\}\\|l|\leqslant N_n}}\Big\{\omega\in\mathbb{R}^d\quad\textnormal{s.t.}\quad|\omega\cdot l|>\frac{\gamma}{\langle l\rangle^{\tau_1}}\Big\},$$
    provided the image is with zero time average. Note that 
    $$\mathtt{DC}_{N_n}(\gamma,\tau_1)\subset\mathcal{C}_n(\gamma,\tau_1,\tau_2,i_0).$$ One can check that $\varphi\mapsto\mathtt{P}(\varphi)$ is odd (so has zero average). One can choose $\hat{\alpha}$ so that the time-average of the second term in the first entry vanishes. The operator $\mathcal{L}_{\bot,n}$ can be inverted thanks to Proposition \ref{prop:inverse_normal}.  Consequently, one can invert the triangular system when seen with values onto the space of reversible tori (to have zero average in the action variable). This provides the inverse operator $\mathcal{T}_n^{-1}.$ Then, we define the approximate inverse as
    $$\mathtt{I}_n\triangleq DG_0(i_{\textnormal{\tiny{triv}}})\mathcal{T}_n^{-1}D\bar{G}_0(i_{\textnormal{\tiny{triv}}})^{-1}.$$
    One obtain the desired result by setting
    \begin{align*}
        \mathscr{E}_{1,n}&\triangleq DG_0(i_{\textnormal{\tiny{triv}}})\mathcal{E}_{1,n}D\bar{G}_0(i_{\textnormal{\tiny{triv}}})^{-1}\mathtt{I}_n,\\
        \mathscr{E}_{2,n}&\triangleq DG_0(i_{\textnormal{\tiny{triv}}})\mathcal{E}_{2,n}D\bar{G}_0(i_{\textnormal{\tiny{triv}}})^{-1}\mathtt{I}_n,\\
        \mathscr{E}_{3,n}&\triangleq DG_0(i_{\textnormal{\tiny{triv}}})\mathcal{E}_{3,n}D\bar{G}_0(i_{\textnormal{\tiny{triv}}})^{-1}\mathtt{I}_n.
    \end{align*}
    The various estimates follow from Lemma \ref{lem;tame_calXP}, Lemma \ref{lem:funct_topo}-$(ii)$-$(iii)$-$(iv)$-$(v)$, \eqref{stima-inverse_nor}, \eqref{stima-tronc:lin_nor} and \eqref{stima-restn}. For more details, we refer to \cite[Sec. 6]{HR21}.
\end{proof}
\subsection{Conclusion : construction of solutions}
We can finally perform the construction of a nontrivial solution. The result is given in the following proposition.

\begin{prop}\label{prop:final_solution}
Define
\begin{equation}\label{def:q_0-q}
    q_{0}\triangleq2\max(\mathbb{S})+2,\qquad q=q_0+2.
\end{equation}
The following hold true.
    \begin{enumerate}[label=(\roman*)]
        \item There exists a Cantor set 
        $$\mathcal{C}_{\infty}^{\varepsilon}\subset\mathscr{O}$$
        and a $q$-times differentiable function
        $$\mathtt{X}_{\infty}:\begin{array}[t]{rcl}
             \mathscr{O} & \rightarrow & \left(\mathbb{T}^d\mapsto\mathbb{T}^{d}\times\mathbb{R}^{d}\times (H^{s_0}\cap L_{\bot}^2)\right)\times\mathbb{R}^{d}\\
		(b,\omega) & \mapsto & \big(i_{\infty}(b,\omega),\alpha_{\infty}(b,\omega)\big)
        \end{array}$$
        such that
        \begin{equation}\label{Xinfty_solution}
            \forall(b,\omega)\in\mathcal{C}_{\infty}^{\varepsilon},\quad\mathscr{F}\big(\mathtt{X}_{\infty}(b,\omega),(b,\omega),\varepsilon\big)=0.
        \end{equation}
        In addition, the torus $i_{\infty}$ is reversible, namely, with the notation \eqref{reversiblility_op:action-angle},
        \begin{equation}\label{reversibility_infty}
            \forall\varphi\in\mathbb{T}^d,\quad\mathfrak{S}\big(i_{\infty}(b,\omega)\big)(\varphi)=i_{\infty}(b,\omega)(-\varphi).
        \end{equation}
    \item There exists a $q$-times differentiable function 
    $$\begin{array}{rcl}
         (b_0,b_1) & \rightarrow & \mathbb{R}^d\\
         b & \mapsto & \omega(b,\varepsilon)
    \end{array}$$
    such that
    \begin{equation}\label{cv:omg}
        \alpha_{\infty}\big(b,\omega(b,\varepsilon)\big)=-\omega_{\textnormal{Eq}}(b),\qquad\lim_{\varepsilon\to0}\omega(b,\varepsilon)=-\omega_{\textnormal{Eq}}(b).
    \end{equation}
    In addition, defining
	\begin{equation}\label{def:fin_Cant}
	    \mathscr{C}_{\infty}^{\varepsilon}\triangleq\Big\{b\in(b_0,b_1)\quad\textnormal{s.t.}\quad\big(b,\omega(b,\varepsilon)\big)\in\mathcal{C}_{\infty}^{\varepsilon}\Big\},
	\end{equation}
	there exists $C>0$ and $\delta\triangleq\delta(\mathbb{S})>0$ such that 
	$$\big|\mathscr{C}_{\infty}^{\varepsilon}\big|\geqslant (b_1-b_0)-C \varepsilon^{\delta}.$$
	In particular, 
	$$\displaystyle \lim_{\varepsilon\to0}\big|\mathscr{C}_{\infty}^{\varepsilon}\big|=b_1-b_0.$$
    Finally,
    $$\forall b\in\mathscr{C}_{\infty}^{\varepsilon},\quad\mathscr{F}\Big(\mathtt{X}_{\infty}\big(b,\omega(b,\varepsilon)\big),\big(b,\omega(b,\varepsilon)\big),\varepsilon\Big)=0.$$
    \end{enumerate}
\end{prop}

\begin{proof} We provide the main steps and arguments. For detailed computations, we refer the reader to \cite[Prop. 7.1, Cor. 7.1 and Prop. 7.2]{HR21}.

\medskip

$\blacktriangleright$ \textbf{Nash-Moser iteration :} The first approximate solution is given by the equilibrium reversible quasi-periodic solutions in \eqref{rev sol eq:EDP}. Let us recall that in action-angle-normal variables they correspond to the trivial embedding $i_{\textnormal{\tiny{triv}}}$ in \eqref{def iflat}. Therefore, we take
\begin{equation}\label{initial_guess}
    \mathtt{X}_0\triangleq(i_0,\alpha_0)\triangleq(i_{\textnormal{\tiny{triv}}},\omega).
\end{equation}
Then, we design the iterative scheme
\begin{equation}\label{iter}
    \mathtt{X}_n=(i_n,\alpha_n),\qquad\mathtt{X}_{n+1}=\mathtt{X}_n-\mathbf{\Pi}_{N_n}\mathtt{I}_n\Pi_{N_n}\mathscr{F}(\mathtt{X}_n),
\end{equation}
where
\begin{equation}\label{scale:Nash}
    N_n\triangleq N_0^{\left(\frac{3}{2}\right)^{n}},\qquad N_{0}\triangleq \gamma^{-1},\qquad\gamma\triangleq \varepsilon^{a},\qquad 0<a<\tfrac{1}{\mu+q+2},
\end{equation}
the Fourier projection $\Pi_{N_n}$ is defined in \eqref{projectors_functions},
$$\mathbf{\Pi}_{N_n}(\mathfrak{I},\alpha)\triangleq(\Pi_{N_n}\mathfrak{I},\alpha),\qquad\mathbf{\Pi}_{N_n}^{\bot}(\mathfrak{I},\alpha)\triangleq(\Pi_{N_n}^{\bot}\mathfrak{I},0)$$
and $\mathtt{I}_{n}$ is a reversible approximate right inverse of $\mathtt{L}_n\triangleq d_{(i,\alpha)}\mathscr{F}(\mathtt{X}_n)$ when restricted to the Cantor set, see the definition \eqref{def:Cant-red},
\begin{equation}\label{Inter_Cantor_NM}
    \mathcal{C}_{n}^{\varepsilon}\triangleq\mathcal{C}_{n}\left(\gamma_n,\tau_1,\tau_2,i_{n-1}\right),\qquad\gamma_{n}\triangleq\gamma(1+2^{-n})
\end{equation}
enjoying the tame estimates \eqref{stima:ttIn}, \eqref{stima:scrE1}, \eqref{stima:scrE2}, \eqref{stima:scrE3} and obtained following the construction of Section \ref{sec:approx_inv}. The reversibility to initial guess $\mathtt{i_0}$, of the operator $\mathtt{I}_n$ and of the equation itself allows to propagate the reversibility along the iteration \eqref{iter}, namely
\begin{equation}\label{reversibility_NM}
    \mathfrak{S}i_n(\varphi)=i_n(-\varphi).
\end{equation}
By Taylor's formula, we get
$$\mathscr{F}(\mathtt{X}_{n+1})=\mathscr{F}(\mathtt{X}_n)-\Pi_{N_n}(\mathtt{L}_n\mathtt{I}_n-\textnormal{Id})\Pi_{N_n}\mathscr{F}(\mathtt{X}_n)+(\mathtt{L}_n\mathbf{\Pi}_{N_n}^{\bot}-\Pi_{N_n}^{\bot}\mathtt{L}_n)\mathtt{I}_n\Pi_{N_n}\mathscr{F}(\mathtt{X}_n)+Q_n,$$
where
$$Q_n\triangleq\mathscr{F}(\mathtt{X}_{n+1})-\mathscr{F}(\mathtt{X}_n)-\mathtt{L}_n(\mathtt{X}_{n+1}-\mathtt{X}_n).$$
Under the following choice of parameters
\begin{equation}\label{choice tau 1 tau2 upsilon} 
		\tau_{1}\triangleq dq_{0}+1,\qquad \tau_{2}\triangleq\tau_{1}+dq_{0}+1,\qquad\upsilon\triangleq\frac{1}{q_{0}+3}
	\end{equation}
    and (recall that $\overline{\sigma}$ is the fixed loss of derivative in the construction of the approximate right inverse in Proposition \ref{prop:approx_full})
\begin{equation}\label{parNash}
    \begin{cases}
		\mathtt{p}_{1}=6q(\tau_{2}+2)+12\overline{\sigma}+15,\\
		\mathtt{p}_{2}=3q(\tau_{2}+2)+6\overline{\sigma}+9,\\
        \mathtt{p}_3=3q(\tau_{2}+2)+6\overline{\sigma}+6,\\
		\mu=2q(\tau_{2}+2)+5\overline{\sigma}+7,\\
		s_{h}=s_{0}+4q(\tau_{2}+2)+9\overline{\sigma}+11,\\
		s_{m}=2s_{h}-s_{0},
	\end{cases}
\end{equation}
using the estimates in Lemma \ref{lem;tame_calXP} and Proposition \ref{prop:approx_full}, one can prove by induction
\begin{enumerate}[label=\textbullet]
\item (Cauchy in intermediate norm) 
\begin{equation}\label{CauchyNM}
    \|\mathtt{X}_{n}-\mathtt{X}_{n-1}\|_{q,\overline{s}_h+\overline{\sigma}}^{\gamma,\mathscr{O}}\leqslant C\varepsilon\gamma^{-1}N_{n-1}^{-\mathtt{p}_{2}};
\end{equation}
    \item (Approximate solution in low norm) 
    \begin{equation}\label{approximate_zero}
        \|\mathscr{F}(\mathtt{X}_n)\|_{q,s_0}^{\gamma,\mathscr{O}}\leqslant C\varepsilon N_{n-1}^{-\mathtt{p}_1};
    \end{equation}
    \item (Growth control in high norm) 
    $$\|\mathtt{X}_n-\mathtt{X}_0\|_{q,s_m+\overline{\sigma}}^{\gamma,\mathscr{O}}\leqslant C\varepsilon\gamma^{-1}N_{n-1}^{\mathtt{p_3}}.$$
\end{enumerate}
We define
\begin{equation}\label{inter-Cant}
    \mathcal{C}_{\infty}^{\varepsilon}\triangleq\bigcap_{n\in\mathbb{N}}\mathcal{C}_{n}^{\varepsilon},\qquad\mathcal{C}_{0}^{\varepsilon}\triangleq\mathscr{O}.
\end{equation}
The Cauchy estimate above allows to conclude the convergence of the sequence of approximate solutions
$$\mathtt{X}_n\underset{n\to\infty}{\longrightarrow}\mathtt{X}_{\infty}\triangleq(i_{\infty},\alpha_{\infty}).$$
The reversibility property \eqref{reversibility_infty} is obtained passing to the limit in \eqref{reversibility_NM}. Passing in the limit in the low norm estimate also provides the fact that $\mathtt{X}_{\infty}$ is a solution, see \eqref{Xinfty_solution}. Now, we obtain a solution of the original problem through the rigidity condition
$$\alpha_{\infty}(b,\omega)=-\omega_{\textnormal{Eq}}(b).$$
From the Cauchy convergence, \eqref{initial_guess} and \eqref{scale:Nash}, we deduce that
\begin{equation}\label{perturb_freq_implicit}
    \alpha_{\infty}(b,\omega)=\omega+O(\varepsilon^{\kappa}), \qquad\kappa\triangleq\kappa(\mathbb{S})>0.
\end{equation}
This implies that $\omega\mapsto\alpha_{\infty}(b,\omega)$ is invertible a we get the desired results by setting
$$\omega(b,\varepsilon)\triangleq\alpha_{\infty}^{-1}\big(b,-\omega_{\textnormal{Eq}}(b)\big).$$

\medskip

    $\blacktriangleright$ \textbf{Measure of the admissible parameters :}
    Consider the set $\mathscr{C}_{\infty}^{\varepsilon}$ defined through \eqref{def:fin_Cant} and \eqref{inter-Cant}. We can write
$$\mathscr{C}_{\infty}^{\varepsilon}
=\bigcap_{n\in\mathbb{N}}\mathscr{C}_{n}^{\varepsilon},
\qquad
\mathscr{C}_{n}^{\varepsilon}
\triangleq\Big\{b\in(b_0,b_1)\quad\textnormal{s.t.}\quad\big(b,\omega(b,\varepsilon)\big)\in
\mathcal{C}_{n}^{\varepsilon}\Big\}.$$
Note that from the convergence \eqref{cv:omg}, we have by continuity that 
\begin{equation}\label{perturb_freq_in_ball}
    \omega([b_0,b_1],\varepsilon)\subset B_{\mathbb{R}^{d}}(0,R_0),
\end{equation}
so that $\mathscr{C}_{0}^{\varepsilon}=(b_0,b_1).$ Therefore, the complement writes
$$(b_0,b_1)\setminus\mathscr{C}_{\infty}^{\varepsilon}=
\bigsqcup_{n\in\mathbb{N}}
\big(\mathscr{C}_{n}^{\varepsilon}\setminus
\mathscr{C}_{n+1}^{\varepsilon}\big).$$
By construction,
$$\mathscr{C}_{n}^{\varepsilon}\setminus
\mathscr{C}_{n+1}^{\varepsilon}=\bigcup_{(l,j)\in\mathbb{Z}^{d+1}\setminus\{(0,0)\}\atop|l|\leqslant N_n}\mathcal R_{l,j}^{(\texttt{TR})}(i_n)\bigcup_{(l,j,j_0)\in\mathbb{Z}^{d}\times(\mathbb{S}_0^c)^2\atop|l|\leqslant N_n}\mathcal R_{l,j,j_0}^{(\texttt{KAM})}(i_n)\bigcup_{(l,j)\in\mathbb{Z}^{d}\times\mathbb{S}_0^c\atop|l|\leqslant N_n}\mathcal R_{l,j}^{(\texttt{Inv})}(i_n),$$
with the resonant sets are given by
$$\begin{aligned}
\mathcal R_{l,j}^{(\texttt{TR})}(i_n)
&\triangleq
\left\{
b\in\mathscr{C}_n^\varepsilon\quad\textnormal{s.t.}\quad
|\omega(b,\varepsilon)\cdot l+jV_n^\infty(b,\varepsilon)|
\leqslant
\frac{4\gamma_{n+1}^{\upsilon}\langle j\rangle}
{\langle l\rangle^{\tau_1}}
\right\},\\
\mathcal R_{l,j,j_0}^{(\texttt{KAM})}(i_n)
&\triangleq
\left\{
b\in\mathscr{C}_n^\varepsilon\quad\textnormal{s.t.}\quad
|\omega(b,\varepsilon)\cdot l+\mathtt{d}_j^{\infty,n}(b,\varepsilon)-\mathtt{d}_{j_0}^{\infty,n}(b,\varepsilon)|
\leqslant
\frac{2\gamma_{n+1}\langle j-j_0\rangle}
{\langle l\rangle^{\tau_2}}
\right\},\\
\mathcal R_{l,j}^{(\texttt{Inv})}(i_n)
&\triangleq
\left\{
b\in\mathscr{C}_n^\varepsilon\quad\textnormal{s.t.}\quad|\omega(b,\varepsilon)\cdot l+\mathtt{d}_j^{\infty,n}(b,\varepsilon)|
\leqslant
\frac{\gamma_{n+1}\langle j\rangle}
{\langle l\rangle^{\tau_1}}
\right\},
\end{aligned}$$
where we denoted
\begin{equation}\label{ttdj_measure}
    \mathtt{d}_j^{\infty,n}(b,\varepsilon)\triangleq
\mathtt{d}_j^\infty\big(b,\omega(b,\varepsilon),i_n\big)=\Omega_j(b)+j\big(V_{i_n}^{\infty}(b,\varepsilon)-\tfrac{1}{2}\big)
+r_j^{\infty,n}(b,\varepsilon),
\end{equation}
with
$$V_n^\infty(b,\varepsilon)\triangleq V_{i_n}^\infty\big(b,\omega(b,\varepsilon)\big),\qquad r_j^{\infty,n}(b,\varepsilon)\triangleq r_j^{\infty}\big(b,\omega(b,\varepsilon),i_n\big).$$
\ding{226} We observe that for $n\geqslant2$, we have
\begin{equation}\label{elim1}
    |l|\leqslant N_{n-1}\qquad\Rightarrow\qquad\mathcal R_{l,j}^{(\texttt{TR})}(i_n)=\mathcal R_{l,j}^{(\texttt{KAM})}(i_n)=\mathcal R_{l,j}^{(\texttt{Inv})}(i_n)=\varnothing.
\end{equation}
Let us give the details for the first one. By construction, it is sufficient to prove that $\mathcal R_{l,j}^{(\texttt{TR})}(i_n)\subset\mathcal R_{l,j}^{(\texttt{TR})}(i_{n-1}).$ By using \eqref{Lipschitz V}, \eqref{CauchyNM} and \eqref{scale:Nash}, we obtain
\begin{align*}
    |\omega(b,\varepsilon)\cdot l+jV_{n-1}^{\infty}(b,\varepsilon)|&\leqslant|\omega(b,\varepsilon)\cdot l+jV_{n}^{\infty}(b,\varepsilon)|+|j|\cdot|V_{n-1}^{\infty}(b,\varepsilon)-V_{n}^{\infty}(b,\varepsilon)|\\
    &\leqslant\tfrac{4\gamma_{n+1}^{\upsilon}|j|}{\langle l\rangle^{\tau_1}}+|j|\cdot\|V_{n-1}^{\infty}-V_{n}^{\infty}\|_{q}^{\gamma,\mathscr{O}}\\
    &\leqslant\tfrac{4\gamma_{n+1}^{\upsilon}|j|}{\langle l\rangle^{\tau_1}}+\varepsilon|j|\cdot\|i_{n-1}-i_{n}\|_{q,\overline{s}_h+\overline{\sigma}}^{\gamma,\mathscr{O}}\\
    &\leqslant\tfrac{4\gamma_{n+1}^{\upsilon}|j|}{\langle l\rangle^{\tau_1}}+C\varepsilon^{2-a}N_{n-1}^{-\mathtt{p}_2}|j|.
\end{align*}
By construction \eqref{Inter_Cantor_NM}, 
$$\gamma_{n+1}^{\upsilon}-\gamma_{n}^{\upsilon}\lesssim-\gamma^{\upsilon}2^{-n}.$$
Together with our choice of parameters \eqref{parNash} and the scale \eqref{scale:Nash}, we deduce for $\varepsilon$ small enough
$$|\omega(b,\varepsilon)\cdot l+jV_{n-1}^{\infty}(b,\varepsilon)|\leqslant\tfrac{4\gamma_{n+1}^{\upsilon}|j|}{\langle l\rangle^{\tau_1}}\cdot$$
This provides the desired claim. The other ones are obtained in a similar way using the Lipschitz estimates \eqref{Lipschitz_dinfty} (This is where they are used).\\
\ding{226} Besides, the following facts hold
\begin{equation}\label{elim2}
    \begin{aligned}
    \mathcal R_{l,j}^{(\texttt{TR})}(i_n)\neq\varnothing\quad&\Rightarrow\quad |j|\leqslant C\langle l\rangle,\\
    \mathcal R_{l,j,j_0}^{(\texttt{KAM})}(i_n)\neq\varnothing\quad&\Rightarrow\quad |j-j_0|\leqslant C\langle l\rangle,\\
    \mathcal R_{l,j}^{(\texttt{Inv})}(i_n)\neq\varnothing\quad&\Rightarrow\quad |j|\leqslant C\langle l\rangle.
\end{aligned}
\end{equation}
Indeed, let us prove the first implication, the other ones being obtained in a similar fashion. If $\mathcal R_{l,j}^{(\texttt{TR})}(i_n)\neq\varnothing,$ then there is some $\bar{b}\in(b_0,b_1)$ such that
$$|\omega(\bar{b},\varepsilon)\cdot l+j V_{n}^{\infty}(\bar{b},\varepsilon)|\leqslant\tfrac{4\gamma_{n+1}^{\upsilon}|j|}{\langle l\rangle^{\tau_1}}\cdot$$
Then, by triangular and Cauchy-Schwarz inequalities together with \eqref{stima:Vinfty}, \eqref{Inter_Cantor_NM}, \eqref{scale:Nash} and \eqref{perturb_freq_in_ball}, we get
$$\tfrac{1}{4}|j|\leqslant|j|\cdot|V_n^{\infty}(\bar{b},\varepsilon)|\leqslant R_0\langle l\rangle+4|j|\gamma_{n+1}^{\upsilon}\leqslant R_0\langle l\rangle+8|j|\varepsilon^{a\upsilon},$$
which provides the desired claim for $\varepsilon$ small enough.\\
\ding{226} Moreover, we also have the following inclusion
\begin{equation}\label{elim3}
    \exists\mathtt{c}>0,\quad\min(|j|,|j_0|)\geqslant\mathtt{c}\gamma_{n+1}^{\upsilon}\langle l\rangle^{\tau_1}\quad\Rightarrow\quad\mathcal{R}_{l,j,j_0}^{(\texttt{KAM})}(i_n)\subset\mathcal{R}_{l,j-j_0}^{(\texttt{TR})}(i_n).
\end{equation}
Let us prove this fact. Due to the symmetry $\mathtt{d}_{-j}^{\infty,n}=-\mathtt{d}_j^{\infty,n}$, we can restrict the discussion to $0<j<j_0.$ Take $b\in\mathcal{R}_{l,j,j_{0}}^{(\mathtt{KAM})}(i_{n}).$ Then, according to \eqref{ttdj_measure} and \eqref{equilibrium_spectrum}, we get
	\begin{align*}
		\big|{\omega}(b,\varepsilon)\cdot l+(j\pm j_{0})V_{n}^{\infty}(b,\varepsilon)\big|  &\leqslant  \big|{\omega}(b,\varepsilon)\cdot l+\mathtt{d}_{j}^{\infty,n}(b,\varepsilon)\pm\mathtt{d}_{j_{0}}^{\infty,n}(b,\varepsilon)\big|+\tfrac12|b^{2j}\pm b^{2j_0}|\\
		&\quad+\tfrac12\big|(j-1)\pm (j_0-1)-(j\pm j_0) \big|+\big|r_{j}^{\infty,n}(b,\varepsilon)\pm r_{j_{0}}^{\infty,n}(b,\varepsilon)\big|\\
        &\leqslant \tfrac{2\gamma_{n+1}\langle j\pm j_{0}\rangle}{\langle l\rangle^{\tau_{2}}}+\tfrac12|b^{2j}\pm b^{2j_0}|\nonumber\\
		 &\quad+\tfrac12\big|(j-1)\pm (j_0-1)-(j\pm j_0) \big|+\big|r_{j}^{\infty,n}(b,\varepsilon)\pm r_{j_{0}}^{\infty,n}(b,\varepsilon)\big|.
	\end{align*}
	We have the trivial bounds
	$$b^{2j}+b^{2j_0}\leqslant C\tfrac{\langle j+j_0\rangle}{j},\qquad b^{2j}-b^{2j_0}\leqslant-2\ln(b)\int_{j}^{j_0}b^{2x}dx
 \leqslant \tfrac{\mathtt{c}_1\langle j-j_{0}\rangle}{j},\qquad\mathtt{c}_1\triangleq\sup_{\substack{j\in\mathbb{N}\\b\in(0,1)}}\big(-2\ln(b)jb^{2j}\big)>0$$
and
$$ \big|(j-1)\pm (j_0-1)-(j\pm j_0) \big| =1\pm 1\leqslant \tfrac{\langle j+j_0\rangle}{j}\cdot$$
	Besides, \eqref{stima restes vp} provides
	\begin{align*}
		\big|r_{j}^{\infty,n}(b,\varepsilon)\pm r_{j_{0}}^{\infty,n}(b,\varepsilon)\big|  \leqslant &C \varepsilon^{1-a}\big(|j|^{-1}+|j_0|^{-1}\big)\\
		\leqslant & C \varepsilon^{1-a} \tfrac{\langle j\pm j_{0}\rangle}{j}\cdot
	\end{align*}
	Putting together the previous estimates yields 
	$$\big|{\omega}(b,\varepsilon)\cdot l+(j\pm j_{0})V_{n}^{\infty}(b,\varepsilon)\big|\leqslant \tfrac{2\gamma_{n+1}\langle j\pm j_{0}\rangle}{\langle l\rangle^{\tau_{2}}}+ C \tfrac{\langle j\pm j_{0}\rangle}{j}\cdot$$
	Now, assuming $\displaystyle j\geqslant\tfrac{1}{2} C\gamma_{n+1}^{-\upsilon}\langle l\rangle^{\tau_{1}}$ and using the fact that $\tau_{2}>\tau_{1}$, we get
	$$\big|{\omega}(b,\varepsilon)\cdot l+(j\pm j_{0})V_{n}^{\infty}(b,\varepsilon)\big| \leqslant  \tfrac{4\gamma_{n+1}^{\upsilon}\langle j\pm j_{0}\rangle}{\langle l\rangle^{\tau_{1}}}\cdot$$

Consequently, combining \eqref{elim1}, \eqref{elim2} and \eqref{elim3}, we can write
$$\mathscr{C}_{n}^{\varepsilon}\setminus
\mathscr{C}_{n+1}^{\varepsilon}\subset\bigcup_{(l,j)\in\mathbb{Z}^{d+1}\setminus\{(0,0)\}\atop \underset{|j|\leqslant C\langle l\rangle}{N_{n-1}\leqslant|l|\leqslant N_n}}\mathcal R_{l,j}^{(\texttt{TR})}(i_n)\bigcup_{(l,j,j_0)\in\mathbb{Z}^{d}\times(\mathbb{S}_0^c)^2\atop \underset{\min(|j|,|j_0|)\leqslant\mathtt{c}\gamma_{n+1}^{\upsilon}\langle l\rangle^{\tau_1};|j-j_0|\leqslant C\langle l\rangle}{N_{n-1}\leqslant|l|\leqslant N_n}}\mathcal R_{l,j,j_0}^{(\texttt{KAM})}(i_n)\bigcup_{(l,j)\in\mathbb{Z}^{d}\times\mathbb{S}_0^c\atop\underset{|j|\leqslant C\langle l\rangle}{N_{n-1}\leqslant|l|\leqslant N_n}}\mathcal R_{l,j}^{(\texttt{Inv})}(i_n).$$
We now estimate the measure of the resonant sets. Since
$q=q_0+2$, the resonance functions are of class $C^{q_0+1}$, and the
transversality estimates for the perturbed frequencies follow from those of $\omega_{\mathrm{Eq}}$ \cite[Lem. 3.5]{HR21-1} (this is where appears the choice of $q_0$ in \eqref{def:q_0-q}) by the above perturbative bounds, see \eqref{perturb_freq_implicit}, \eqref{stima:Vinfty} and \eqref{stima restes vp}. More precisely, there exists $\rho_0>0$ such that, for
all relevant Fourier indices,
$$\begin{aligned}
\inf_{b\in[b_0,b_1]}\max_{0\leqslant q\leqslant q_0}
|\partial_b^q(\omega(b,\varepsilon)\cdot l)|
&\geqslant \frac{\rho_0}{2}\langle l\rangle,\\
\inf_{b\in[b_0,b_1]}\max_{0\leqslant q\leqslant q_0}
|\partial_b^q(\omega(b,\varepsilon)\cdot l+jV_n^\infty(b,\varepsilon))|
&\geqslant \frac{\rho_0}{2}\langle l\rangle,\\
\inf_{b\in[b_0,b_1]}\max_{0\leqslant q\leqslant q_0}
|\partial_b^q(\omega(b,\varepsilon)\cdot l+\mu_j^{\infty,n}(b,\varepsilon))|
&\geqslant \frac{\rho_0}{2}\langle l\rangle,\\
\inf_{b\in[b_0,b_1]}\max_{0\leqslant q\leqslant q_0}
|\partial_b^q(\omega(b,\varepsilon)\cdot l+\mu_j^{\infty,n}(b,\varepsilon)-\mu_{j_0}^{\infty,n}(b,\varepsilon))|
&\geqslant \frac{\rho_0}{2}\langle l\rangle.
\end{aligned}$$
Rüssmann's lemma \cite[Thm.~17.1]{R01} consequently gives
$$
\begin{aligned}
|\mathcal R_{l,j}^{(\mathtt{TR})}(i_n)|
&\lesssim
\gamma^{\upsilon/q_0}
\langle j\rangle^{1/q_0}
\langle l\rangle^{-1-(\tau_1+1)/q_0}
\langle l,j\rangle,\\
|\mathcal R_{l,j}^{(\mathtt{Inv})}(i_n)|
&\lesssim
\gamma^{1/q_0}
\langle j\rangle^{1/q_0}
\langle l\rangle^{-1-(\tau_1+1)/q_0}
\langle l,j\rangle,\\
|\mathcal R_{l,j,j_0}^{(\mathtt{KAM})}(i_n)|
&\lesssim
\gamma^{1/q_0}
\langle j-j_0\rangle^{1/q_0}
\langle l\rangle^{-1-(\tau_2+1)/q_0}
\langle l,j-j_0\rangle.
\end{aligned}
$$
Also observe that, 
	$$\begin{cases}
	    |j-j_{0}|\leqslant C\langle l\rangle,\\
        \min(|j|,|j_{0}|)\leqslant\gamma_{n+1}^{-\upsilon}\langle l\rangle^{\tau_{1}}
\end{cases}\quad\Rightarrow\quad\max(|j|,|j_{0}|)=\min(|j|,|j_{0}|)+|j-j_{0}|\leqslant\gamma_{n+1}^{-\upsilon}\langle l\rangle^{\tau_{1}}+C_{0}\langle l\rangle\lesssim\gamma^{-\upsilon}\langle l\rangle^{\tau_{1}}.$$
Hence, combining the foregoing estimates, we find, for $n\geqslant2$,
$$|\mathscr{C}_{n}^{\varepsilon}\setminus\mathscr{C}_{n+1}^{\varepsilon}|
\lesssim
\gamma^{1/q_0}
\left(
\sum_{|l|>N_{n-1}}\langle l\rangle^{-\tau_1/q_0}+\gamma^{-\upsilon}
\sum_{|l|>N_{n-1}}
\langle l\rangle^{\tau_1-\tau_2/q_0}
\right)+\gamma^{\upsilon/q_0}
\sum_{|l|>N_{n-1}}
\langle l\rangle^{-\tau_1/q_0}.$$
The cases $n=0,1$ are treated similarly (and are actually simpler since one does not need to track the decay with respect to $N_{n-1}$). Therefore, by virtue of \eqref{choice tau 1 tau2 upsilon}, we infer
$$\big|(b_0,b_1)\setminus\mathscr{C}_\infty^\varepsilon\big|\leqslant\sum_{n\in\mathbb{N}}\mathcal |\mathscr{C}_{n}^{\varepsilon}\setminus\mathscr{C}_{n+1}^{\varepsilon}|\lesssim
\gamma^{\min\{\upsilon/q_0,\,1/q_0-\upsilon\}}=\varepsilon^{\frac{a\upsilon}{q_0}}.$$
This ends the proof of Proposition \ref{prop:final_solution}.
\end{proof}
\appendix
\section{Topological and algebraical materials for Section \ref{sec:QP}}\label{appendix:topo}
This appendix collects the fundamental definitions and results concerning the functional analysis used in Section \ref{sec:QP}. First we introduce the weighted Sobolev function spaces. Then, we provide the operator topology for linear Toeplitz transformations. A special focus is made on integral operators and their reversibility.

\subsection{Weighted Sobolev spaces}\label{Functionspsaces}
We denote $L^{2}(\mathbb{T}^{d+1},\mathbb{C})$ the space of square integrable complex-valued function in the variables $(\varphi,\theta).$ It has a Hilbert structure with inner product
$$\big\langle u_{1},u_{2}\big\rangle_{L^{2}(\mathbb{T}^{d+1},\mathbb{C})}=\bigintssss_{\mathbb{T}^{d+1}}u_{1}(\varphi,\theta)\overline{u_{2}(\varphi,\theta)}d\varphi d\theta$$ so that
It also possesses the exponential Hilbert basis $(\mathbf{e}_{l,j})_{(l,j)\in\mathbb{Z}^{d}\times\mathbb{Z}}$ given by
$$\mathbf{e}_{l,j}(\varphi,\theta)\triangleq e^{\ii(l\cdot\varphi+j\theta)},$$
so that the following decomposition holds
$$u=\sum_{(l,j)\in\mathbb{Z}^{d+1 }}u_{l,j}\,\mathbf{e}_{l,j},\qquad u_{l,j}\triangleq\big\langle u,\mathbf{e}_{l,j}\big\rangle_{L^{2}(\mathbb{T}^{d+1},\mathbb{C})}.$$
Given $s\in\mathbb{R},$ we define the Sobolev space
$$H^{s}(\mathbb{T}^{d+1},\mathbb{C})\triangleq\Big\{u\in L^2(\mathbb{T}^{d+1},\mathbb{C})\quad\textnormal{s.t.}\quad\|u\|_{s}<\infty\Big\},$$
where
$$\|u\|_{H^{s}}^{2}\triangleq \sum_{(l,j)\in\mathbb{Z}^{d+1 }}\langle l,j\rangle^{2s}|u_{l,j}|^{2},\qquad\langle l,j\rangle\triangleq \max(1,|l|,|j|).$$
The real-valued Sobolev space is simply denoted
$$H^{s}\triangleq\Big\{u\in H^{s}(\mathbb{T}^{d+1},\mathbb{C})\quad\textnormal{s.t.}\quad\overline{u}=u\Big\}.$$
Now we introduce the weighted Sobolev spaces for a given $\gamma>0$ and onpen set $\mathscr{O}\subset\mathbb{R}^{d+1}.$  
\begin{align*}
	W^{q,\infty,\gamma}(\mathscr{O},H^{s})&\triangleq\Big\lbrace u:\mathscr{O}\rightarrow H^{s}\quad\textnormal{s.t.}\quad\|u\|_{q,s}^{\gamma,\mathscr{O}}<\infty\Big\rbrace,\\
	W^{q,\infty,\gamma}(\mathscr{O},\mathbb{C})&\triangleq\Big\lbrace u:\mathscr{O}\rightarrow\mathbb{C}\quad\textnormal{s.t.}\quad\|u\|_{q}^{\gamma,\mathscr{O}}<\infty\Big\rbrace,
\end{align*}
where 
$$\|u\|_{q,s}^{\gamma,\mathscr{O}}\triangleq\sum_{\underset{|\alpha|\leqslant q}{\alpha\in\mathbb{N}^{d+1}}}\gamma^{|\alpha|}\sup_{\mathtt{p}\in{\mathscr{O}}}\|\partial_{\mathtt{p}}^{\alpha}u(\mathtt{p},\cdot)\|_{s-|\alpha|},\qquad\|u\|_{q}^{\gamma,\mathscr{O}}\triangleq\sum_{\underset{|\alpha|\leqslant q}{\alpha\in\mathbb{N}^{d+1}}}\gamma^{|\alpha|}\sup_{\mathtt{p}\in{\mathscr{O}}}|\partial_{\mathtt{p}}^{\alpha}u(\mathtt{p})|.$$
In the following lemma, we collect some classical properties of weighted Sobolev norms.
\begin{lem}\label{lem:funct_topo}
	The following assertions hold true.
		\begin{enumerate}[label=(\roman*)]
			\item Space translation invariance:  Define for $\eta\in\mathbb{T},$ the translation
            $$(\tau_{\eta}u)(\mathtt{p},\theta)\triangleq u(\mathtt{p},\theta+\eta).$$Let $u\in W^{q,\infty,\gamma}(\mathscr{O},H^{s}).$ Then, for any $\eta\in\mathbb{T},$ we have $\tau_{\eta}u\in W^{q,\infty,\gamma}(\mathscr{O},H^{s})$ and
			$$\|\tau_{\eta}u\|_{q,s}^{\gamma,\mathscr{O}}=\|u\|_{q,s}^{\gamma,\mathscr{O}}.$$
			\item Projectors : Define 
            for $N\in\mathbb{N}^{*},$ the frequency projectors  
\begin{equation}\label{projectors_functions}
	\Pi_{N}u\triangleq\sum_{\underset{\langle l,j\rangle\leqslant N}{(l,j)\in\mathbb{Z}^{d+1 }}}u_{l,j}\mathbf{e}_{l,j}\qquad \textnormal{and}\qquad \Pi^{\bot}_{N}\triangleq\textnormal{Id}-\Pi_{N}.
\end{equation}
Let $u\in W^{q,\infty,\gamma}(\mathscr{O},H^{s}).$ Then, for all $N\in\mathbb{N}^{*}$ and for all $t\geqslant0,$
			$$\|\Pi_{N}u\|_{q,s+t}^{\gamma,\mathscr{O}}\leqslant N^{t}\|u\|_{q,s}^{\gamma,\mathscr{O}}\qquad\textnormal{and}\qquad\|\Pi_{N}^{\bot}u\|_{q,s}^{\gamma,\mathscr{O}}\leqslant N^{-t}\|u\|_{q,s+t}^{\gamma,\mathscr{O}}.$$
			\item Law products:
			\begin{enumerate}[label=(\alph*)]
				\item Let $u_{1},u_{2}\in W^{q,\infty,\gamma}(\mathscr{O},H^{s}).$ Then $u_{1}u_{2}\in W^{q,\infty,\gamma}(\mathscr{O},H^{s})$ and 
				$$\| u_{1}u_{2}\|_{q,s}^{\gamma,\mathscr{O}}\lesssim\| u_{1}\|_{q,s_{0}}^{\gamma,\mathscr{O}}\| u_{2}\|_{q,s}^{\gamma,\mathscr{O}}+\| u_{1}\|_{q,s}^{\gamma,\mathscr{O}}\| u_{2}\|_{q,s_{0}}^{\gamma,\mathscr{O}}.$$
				\item Let $u_{1},u_{2}\in W^{q,\infty,\gamma}(\mathscr{O},\mathbb{C}).$ Then $u_{1}u_{2}\in W^{q,\infty,\gamma}(\mathscr{O},\mathbb{C})$ and $$\| u_{1}u_{2}\|_{q}^{\gamma,\mathscr{O}}\lesssim\| u_{1}\|_{q}^{\gamma,\mathscr{O}}\| u_{2}\|_{q}^{\gamma,\mathscr{O}}.$$
				\item Let $(u_{1},u_{2})\in W^{q,\infty,\gamma}(\mathscr{O},\mathbb{C})\times W^{q,\infty,\gamma}(\mathscr{O},H^{s}).$ Then $u_{1}u_{2}\in W^{q,\infty,\gamma}(\mathscr{O},H^{s})$ and 
				$$\| u_{1}u_{2}\|_{q,s}^{\gamma,\mathscr{O}}\lesssim\| u_{1}\|_{q}^{\gamma,\mathscr{O}}\| u_{2}\|_{q,s}^{\gamma,\mathscr{O}}.$$
			\end{enumerate}
			\item Composition law: Let $F\in C^{\infty}(\mathscr{O}\times\mathbb{R},\mathbb{R})$ and  $u_{1},u_{2}\in W^{q,\infty,\gamma}(\mathscr{O},H^{s})$  such that $$\| u_{1}\|_{q,s}^{\gamma,\mathscr{O}},\|u_{2}\|_{q,s}^{\gamma,\mathscr{O}}\leqslant C_{0}$$ for an  arbitrary  constant  $C_{0}>0$ and define the pointwise composition $$\forall (\mathtt{p},\varphi,\theta)\in \mathscr{O}\times\mathbb{T}^{d+1},\quad F(u)(\mathtt{p},\varphi,\theta)\triangleq  F\big(\mathtt{p},u(\mathtt{p},\varphi,\theta)\big).$$
			Then $F(u_{1})-F(u_{2})\in W^{q,\infty,\gamma}(\mathscr{O},H^{s})$ with 
			$$\|F(u_{1})-F(u_{2})\|_{q,s}^{\gamma,\mathscr{O}}\leqslant C(s,d,q,F,C_{0})\|u_{1}-u_{2}\|_{q,s}^{\gamma,\mathscr{O}}.$$
			\item Composition law 2: Let $F\in C^{\infty}(\mathbb{R},\mathbb{R})$ with bounded derivatives. Let $u\in W^{q,\infty,\gamma}(\mathscr{O},\mathbb{C}).$ Then 
			$$\|F(u)-F(0)\|_{q}^{\gamma,\mathscr{O}}\leqslant C(q,d,F)\|u\|_{q}^{\gamma,\mathscr{O}}\left(1+\|u\|_{L^{\infty}(\mathscr{O})}^{q-1}\right).
			$$
			This estimate is also true for $\gamma=1$, corresponding to the classical Sobolev space $W^{q,\infty}(\mathscr{O},\mathbb{C}).$
	\end{enumerate}
\end{lem}

\subsection{Operators}\label{section-ope}
Here we discuss the Toeplitz structure for operators and their associated topology. In particular, we provide important estimates for integral-type operators.\\

Consider a family of bounded operators over $H^s(\mathbb{T}^{d+1},\mathbb{C})$ and depending smoothly on a parameter $\mathtt{p}\in\mathscr{O}$: 
$$A:\mathtt{p}\in\mathscr{O}\mapsto A(\mathtt{p})\in\mathcal{L}\big(H^s(\mathbb{T}^{d+1},\mathbb{C})\big).$$
In Section \ref{sec:QP}, such can of operators act on $W^{q,\infty,\gamma}\big(\mathscr{O},H^{s}(\mathbb{T}^{d+1},\mathbb{C})\big)$ through
$$(Au)(\mathtt{p},\varphi,\theta)\triangleq A(\mathtt{p})u(\mathtt{p},\varphi,\theta).$$
We can characterize the linear operator $A(\mathtt{p})$ by its action on the Fourier basis: for any $(l_0,j_0)\in\mathbb{Z}^{d+1},$
$$A(\mathtt{p})\mathbf{e}_{l_{0},j_{0}}=\sum_{(l,j)\in\mathbb{Z}^{d+1}}A_{l_{0},j_{0}}^{l,j}(\mathtt{p})\mathbf{e}_{l,j},\qquad A_{l_{0},j_{0}}^{l,j}(\mathtt{p})\triangleq\big\langle A(\mathtt{p})\mathbf{e}_{l_{0},j_{0}},\mathbf{e}_{l,j}\big\rangle_{L^{2}(\mathbb{T}^{d+1},\mathbb{C})}.$$
\begin{defin}
	Consider the involution
	$$(\mathscr{S}_{2}u)(\varphi,\theta)\triangleq u(-\varphi,-\theta).$$
	We say that an operator $A$ is 
	\begin{enumerate}[label=\textbullet]
		\item \textit{real} if it sends real-valued function onto real-valued functions, namely 
$$\overline{u}=u\quad\Longrightarrow\quad\overline{Au}=Au.$$
		\item \textit{reversible} if it sends even (resp. odd) functions of $(\varphi,\theta)$ onto odd (resp. even) functions of $(\varphi,\theta)$, namely
		$$A\circ\mathscr{S}_{2}=-\mathscr{S}_{2}\circ A.$$
		\item \textit{reversibility preserving} if it sends even (resp. odd) functions of $(\varphi,\theta)$ onto even (resp. odd) functions of $(\varphi,\theta)$, namely
		$$A\circ\mathscr{S}_{2}=\mathscr{S}_{2}\circ A.$$
	\end{enumerate}
\end{defin}
\begin{remark}\label{rem:compo_rev}
    The composition of reversible and reversibility preserving operators follows the sign rule with
    $$\textnormal{reversible}=-\qquad\textnormal{and}\qquad\textnormal{reversibility preserving}=+.$$
\end{remark}
\subsubsection{Time-Toeplitz operators and their topology}\label{Top-Sec-11}
Now, we present the special class of Toeplitz-in-time operators defined as follows.
\begin{defin}
	An operator $A(\mathtt{p})$ is said  Toeplitz-in-time (i.e. in the variable $\varphi\in\mathbb{T}^d$) if and only if we can write 
	$$A_{l_{0},j_{0}}^{l,j}(\mathtt{p})=A_{j_{0}}^{j}(\mathtt{p},l-l_{0}),\qquad A_{j_{0}}^{j}(\mathtt{p},l)\triangleq A_{0,j_{0}}^{l,j}(\mathtt{p}).$$
\end{defin}

The class of Toeplitz-in-time operators is endowed with the off-diagonal decay normed topology
$$\|A\|_{\textnormal{\tiny{O-d}},q,s}^{\gamma,\mathscr{O}}\triangleq\sum_{\underset{|\alpha|\leqslant q}{\alpha\in\mathbb{N}^{d+1}}}\gamma^{|\alpha|}\sup_{\mathtt{p}\in{\mathscr{O}}}\|\partial_{\mathtt{p}}^{\alpha}(A)(\mathtt{p})\|_{\textnormal{\tiny{O-d}},s-|\alpha|},\qquad \|A\|_{\textnormal{\tiny{O-d}},s}^{2}\triangleq\sum_{(l,m)\in\mathbb{Z}^{d+1}}\langle l,m\rangle^{2s}\sup_{j-k=m}|A_{j}^{k}(l)|^{2}.$$
The Toeplitz-in-time operatorial norm presented above enjoys several properties listed below. We refer the reader to \cite{BM18}.
\begin{lem}\label{lem:Toeplitz_topo}
	Let $A,$ $A_{1}$ and $A_{2}$ be Toeplitz-in-time operators.
	\begin{enumerate}[label=(\roman*)]
		\item Projectors : Given $N\in\mathbb{N}^{*},$ we define the projectors
        $$\left(P_{N}A(\mathtt{p})\right)\mathbf{e}_{l_{0},j_{0}}\triangleq\sum_{\underset{|l-l_{0}|,|j-j_{0}|\leqslant N}{(l,j)\in\mathbb{Z}^{d+1}}}A_{j_{0}}^{j}(\mathtt{p},l-l_0)\mathbf{e}_{l,j}\qquad\mbox{and}\qquad P_{N}^{\bot}A\triangleq A-P_{N}A.$$
        Then, for any $t\geqslant0,$ we have
		$$\|P_{N}A\|_{\textnormal{\tiny{O-d}},q,s+t}^{\gamma,\mathscr{O}}\leqslant N^{t}\| A\|_{\textnormal{\tiny{O-d}},q,s}^{\gamma,\mathscr{O}}\qquad\mbox{and}\qquad\| P_{N}^{\bot}A\|_{\textnormal{\tiny{O-d}},q,s}^{\gamma,\mathscr{O}}\leqslant N^{-t}\|A\|_{\textnormal{\tiny{O-d}},q,s+t}^{\gamma,\mathscr{O}}.$$
		\item Composition law :
		$$\|A_{1}A_{2}\|_{\textnormal{\tiny{O-d}},q,s}^{\gamma,\mathscr{O}}\lesssim\|A_{1}\|_{\textnormal{\tiny{O-d}},q,s}^{\gamma,\mathscr{O}}\| A_{2}\|_{\textnormal{\tiny{O-d}},q,s_{0}}^{\gamma,\mathscr{O}}+\|A_{1}\|_{\textnormal{\tiny{O-d}},q,s_{0}}^{\gamma,\mathscr{O}}\|A_{2}\|_{\textnormal{\tiny{O-d}},q,s}^{\gamma,\mathscr{O}}.$$ 
		\item The off-diagonal norm controls the operator norm :
		$$\|Au\|_{q,s}^{\gamma,\mathscr{O}}\lesssim\|A\|_{\textnormal{\tiny{O-d}},q,s_{0}}^{\gamma,\mathscr{O}}\|u\|_{q,s}^{\gamma,\mathscr{O}}+\|A\|_{\textnormal{\tiny{O-d}},q,s}^{\gamma,\mathscr{O}}\|u\|_{q,s_{0}}^{\gamma,\mathscr{O}}.$$
	\end{enumerate}
\end{lem}
\subsubsection{Integral operators}
	Integral operators form a special subclass of the Toeplitz-in-time family, that can be written as
	\begin{equation}\label{def:integr}
	    (\mathcal{I}_Ku)(\mathtt{p},\varphi,\theta)\triangleq\int_{\mathbb{T}}K(\mathtt{p},\varphi,\theta,\eta)u(\mathtt{p},\varphi,\eta)d\eta
	\end{equation}
    and the function $K$ in the integrand is called \textit{kernel}.
	We have the following symmetry result, see \cite[Lem. 4.4]{HR21}.
	\begin{lem}\label{lem:rever_integr}
		Let $\mathcal{I}_K$ be an integral operator with a real-valued kernel $K$.
		\begin{enumerate}[label=\textbullet]
			\item If $K$ is even in $(\varphi,\theta,\eta)$, then $\mathcal{I}_K$ is a reversibility preserving Toeplitz-in-time operator.
			\item If $K$ is odd in $(\varphi,\theta,\eta)$,
			then $\mathcal{I}_K$ is a reversible Toeplitz-in-time operator.
		\end{enumerate}
		In addition, 
	$$\|\mathcal{I}_K\|_{\textnormal{\tiny{O-d}},q,s}^{\gamma,\mathscr{O}}\lesssim\|K\|_{q,s}^{\gamma,\mathscr{O}}.$$
	\end{lem}
	In the sequel, we provide some results dealing with the action of symplectic diffeomorphisms on Toeplitz-in-time integral operators. The first one treats the smooth kernel case.
	\begin{lem}\label{lem:conj_intgral}
		Consider a smooth real-valued kernel
		$$K:(\mathtt{p},\varphi,\theta,\eta)\mapsto K(\mathtt{p},\varphi,\theta,\eta).$$
		Let $\beta\in W^{q,\infty,\gamma}(\mathscr{O},H^{\infty})$ odd in the variables $(\varphi,\theta)$ and such that there exists $\varepsilon_0>0$ small enough such that
		\begin{equation}\label{small_beta_technic}
		    \|\beta\|_{q,2s_0}^{\gamma,\mathscr{O}}\leqslant \varepsilon_0.
		\end{equation}
		Consider the quasi-periodic change of variables
		$$\mathscr{B}\triangleq (1+\partial_{\theta}\beta)\mathcal{B},\qquad\mathcal{B}u(\mathtt{p},\varphi,\theta)\triangleq u\big(\mathtt{p},\varphi,\theta+\beta(\mathtt{p},\varphi,\theta)\big),$$
		Then, the following holds.
		\begin{enumerate}[label=(\roman*)]
			\item The operator $\mathscr{B}^{-1}\mathcal{I}_K\mathscr{B}$ is an integral operator satisfying the estimates: for $s\geqslant s_0,$ 
			\begin{align}
				\|\mathscr{B}^{-1}\mathcal{I}_K\mathscr{B}\|_{\textnormal{\tiny{O-d}},q,s}^{\gamma,\mathscr{O}}&\lesssim\|K\|_{q,s}^{\gamma,\mathscr{O}}+\|K\|_{q,s_0}^{\gamma,\mathscr{O}}\|\beta\|_{q,s+1}^{\gamma,\mathscr{O}}, \label{stima:conj_int}\\
			\|\mathscr{B}^{-1}\mathcal{I}_K\mathscr{B}-\mathcal{I}_K\|_{\textnormal{\tiny{O-d}},q,s}^{\gamma,\mathscr{O}}&\lesssim\|K\|_{q,s+1}^{\gamma,\mathscr{O}}\|\beta\|_{q,s_0}^{\gamma,\mathscr{O}}+\|K\|_{q,s_0}^{\gamma,\mathscr{O}}\|\beta\|_{q,s+1}^{\gamma,\mathscr{O}}.	\label{stima:conj_int_bis}
			\end{align}
            \item If $K$ is even in all the variables $(\varphi,\theta,\eta)$ (resp. odd), then $\mathscr{B}^{-1}\mathcal{I}_K\mathscr{B}$ is a reversibility preserving (resp. reversible) integral operator.			
			\item Assume the smooth functional dependence $r\in W^{q,\infty,\gamma}(\mathscr{O},H^{s})\mapsto K(r), \beta(r)$. Then, take two states $r^{[1]},r^{[2]}\in W^{q,\infty,\gamma}(\mathscr{O},H^s)$ and denote
			$$\forall k\in\{1,2\},\quad f^{[k]}\triangleq f(r^{[k]})\qquad\textnormal{and}\qquad\Delta_{12}f\triangleq f^{[1]}-f^{[2]}.$$ 
			Assume that there exists $\varepsilon_0>0$ small enough such that 
			$$\max_{k\in\{1,2\}}\|\beta^{[k]}\|_{q,2s_0}^{\gamma,\mathscr{O}}+\max_{k\in\{1,2\}}\|K^{[k]}\|_{q,s_0+1}^{\gamma,\mathscr{O}}\leqslant\varepsilon_0.$$
			Then, the following estimate holds, for $s\geqslant s_0$,
			\begin{equation}\label{stima:D12_conj}
				\begin{aligned}
				    \|\Delta_{12}\mathscr{B}^{-1}\mathcal{I}_{K}\mathscr{B}\|_{\textnormal{\tiny{O-d}},q,s}^{\gamma,\mathscr{O}}&\lesssim\|\Delta_{12}K\|_{q,s}^{\gamma,\mathscr{O}}+\|\Delta_{12}K\|_{q,s_0}^{\gamma,\mathscr{O}}\max_{k\in\{1,2\}}\|\beta^{[k]}\|_{q,s}^{\gamma,\mathscr{O}}\\
				&\quad+\Big(\max_{k\in\{1,2\}}\|K^{[k]}\|_{q,s+1}^{\gamma,\mathscr{O}}+\max_{k\in\{1,2\}}\|\beta^{[k]} \|_{q,s+1}^{\gamma,\mathscr{O}}\Big)\|\Delta_{12}\beta\|_{q,s_0+1}^{\gamma,\mathscr{O}}+\|\Delta_{12}\beta\|_{q,s+1}^{\gamma,\mathscr{O}}.
				\end{aligned} 
            \end{equation}
		\end{enumerate}
	\end{lem}
	
	The next result treats the conjugation of the Hilbert transform via diffeomorphisms. 
	\begin{lem} \label{lem:Hilbert}
    Consider $\mathcal{H}$ the $2\pi$-periodic Hilbert transform defined by
	\begin{equation}\label{def Hilbert}
		\mathcal{H}(1)=0,\qquad\forall j\in\mathbb{Z}^*,\quad\mathcal{H}\mathbf{e}_j=-\ii\, \mathtt{sgn}(j)\mathbf{e}_j,
	\end{equation}
    where $\mathtt{sgn}$ denotes the usual sign function.
		Let $\beta\in W^{q,\infty,\gamma}(\mathscr{O},H^{\infty})$ odd in the variables $(\varphi,\theta)$ and subject to the smallness condition \eqref{small_beta_technic}. Then,
		$$(\mathscr{B}^{-1}\mathcal{H}\mathscr{B}-\mathcal{H})u (\mathtt{p},\varphi,\theta) 
		=\int_{\mathbb{T}}  \,K(\mathtt{p},\varphi,\theta,\eta)u(\mathtt{p},\varphi,\eta)\,d\eta,$$
		defines a reversible integral operator with the estimates : for all $s\geqslant s_0,$  
		\begin{equation}\label{stime:Hilbert}
		    \begin{aligned}
		    \|K\|_{q,s}^{\gamma,\mathscr{O}}&\lesssim\|\beta\|_{q,s+2}^{\gamma,\mathscr{O}},\\
            \|\mathscr{B}^{-1}\mathcal{H}\mathscr{B}-\mathcal{H}\|_{\textnormal{\tiny{O-d}},q,s}^{\gamma,\mathscr{O}}&\lesssim\|\beta \|_{q,s+2}^{\gamma,\mathscr{O}},\\
            \|\Delta_{12}(\mathscr{B}^{-1}\mathcal{H}\mathscr{B}-\mathcal{H})\|_{\textnormal{\tiny{O-d}},q,s}^{\gamma,\mathscr{O}}&\lesssim\Big[\|\Delta_{12}\beta\|_{q,s+2}^{\gamma,\mathscr{O}}+\|\Delta_{12}\beta\|_{q,s_0+1}^{\gamma,\mathscr{O}}\max_{k\in\{1,2\}}\|\beta^{[k]}\|_{q,s+3}^{\gamma,\mathscr{O}}\Big].
		\end{aligned}
		\end{equation}
	\end{lem}

\noindent \textbf{Emeric Roulley}\\
Dipartimento di Matematica “Federigo Enriques”\\
Università degli Studi di Milano\\
Via Cesare Saldini 50, 20133 Milano, Italy.\\
Email : emeric.roulley@unimi.it.
\end{document}